\documentclass[12pt]{amsart}
\usepackage[utf8]{inputenc}
\usepackage{url,hyperref}
\usepackage{tikz}
\usetikzlibrary{shapes.geometric}
\usepackage{amsthm,amsmath,amssymb}
\usepackage{caption}

\newcommand{\RR}{\mathbb{R}}

\newcommand{\ZZ}{\mathbb{Z}}
\newcommand{\QQ}{\mathbb{Q}}

\newcommand{\sym}{\mathfrak{S}}

\newcommand{\wcomp}{{\rm WComp}}

\newcommand{\E}{\mathcal{E}}
\newcommand{\xx}{\mathbf{x}}
\newcommand{\cc}{\mathbf{c}}
\newcommand{\clr}{{\bf\mathop{clr}}}

\newcommand{\internal}{{\bf\mathop{int}}}
\newcommand{\BT}{\mathcal{BT}}

\newcommand{\Lyn}{\mathcal{L}yn}

\newcommand{\bT}{\mathbf{T}}

\newcommand{\bpi}{\pmb{\pi}}

\newcommand{\balpha}{\pmb{\alpha}}
\newcommand{\bbeta}{\pmb{\beta}}

\newcommand{\PF}{\mathnormal{PF}}

\renewcommand{\P}{\mathcal{P}}
\newcommand{\B}{\mathcal{B}}

\def\newop#1{\expandafter\def\csname #1\endcsname{\mathop{\rm #1}\nolimits}}

\newop{sgn}
\newop{supp}
\newop{init}
\newop{des}
\newop{peak}
\newop{asc}
\newop{nlyn}
\newop{aapair}
\newop{Max}
\newop{co}
\newop{Des}

\newtheorem{theorem}{Theorem}[section]
\newtheorem{lemma}[theorem]{Lemma}
\newtheorem{proposition}[theorem]{Proposition}
\newtheorem{corollary}[theorem]{Corollary}

\newtheorem{question}[theorem]{Question}

\newtheorem{conjecture}[theorem]{Conjecture}
\theoremstyle{definition}
  \newtheorem{remark}[theorem]{Remark}
  \newtheorem{definition}[theorem]{Definition}
  \newtheorem{example}[theorem]{Example}

\counterwithin{figure}{section}
\numberwithin{equation}{section}

\title[Weighted bond posets]{Weighted bond posets and a new chromatic symmetric function}

\author[Gonz\'{a}lez D'Le\'{o}n]{Rafael S. Gonz\'alez D'Le\'on$^{*}$}
\address[R. S. Gonz\'{a}lez D'Le\'{o}n]{Department of Mathematics and Statistics, Loyola University of Chicago, Chicago, USA}
\email{rgonzalezdleon@luc.edu}
\urladdr{\url{https://dleon.combinatoria.co/}}
\thanks   {$^{*}$  Partially supported by an AMS-Simons Research Enhancement Grant for Primarily Undergraduate Institution Faculty and by the program of postdoctoral fellowships of Minciencias (Colombian Ministry of Science).
}

\author[Wachs]{Michelle L. Wachs$^{**}$}
\address[M.\ L.\ Wachs]{Department of Mathematics\\University of Miami\\Coral Gables\\Florida\\USA} 
\email{wachs@math.miami.edu}
\urladdr{\url{http://www.math.miami.edu/~wachs/}}
\thanks  { {$^{**}$Partially supported by NSF grants
DMS  1502606 and DMS 2207337}}
		
\keywords{weighted partitions, bond lattice, chromatic polynomial, graph associahedra, gamma-positivity and unimodality, symmetric functions, lexicographic shellability}

\date{\today}
\subjclass[2020]{Primary: 06A07, 06A11, 05A18, 05C31, 05E05; Secondary: 05A15, 05C05, 05C15, 05C40, 52B05, 52B12}  

\begin{document}

\begin{abstract}

The classical bond lattice of a graph was used in a formula of Whitney to compute the chromatic polynomial.  In this paper, we use weighted versions of the bond lattice, to introduce and study a new polynomial and a new symmetric function invariant of a graph. Examples of the new polynomial invariant include the classical Narayana polynomials (for the path graph), 
the tree-Eulerian polynomials (for the complete graph), and the binomial-Eulerian polynomials (for the star graph).  These examples suggest an interesting connection to  
$h$-polynomials of general graph-associahedra. We prove that for any chordal graph, our polynomial graph invariant is $\gamma$-positive.

Multiweighted bond posets yield a symmetric function graph invariant that is an analog of the chromatic polynomial. The highest degree homogeneous component of this symmetric function is of particular interest. The parking function symmetric function introduced by Haiman arises as an example, as do symmetric functions studied by the first author in connection with  multibracketed Lie algebras and with colored exterior algebras. The $\gamma$-positivity result mentioned above is a specialization of an $e$-positivity result for the highest degree homogeneous component, which we prove for any chordal graph using the theory of lexicographic shellability. We conjecture that this symmetric function is Schur-log-concave, which specializes to  Huh's log-concavity theorem for the chromatic polynomial.
\end{abstract}

\maketitle

\setcounter{tocdepth}{1}
\tableofcontents

\section{Introduction}

In \cite{BjornerWelker2005} Bj\"orner and Welker introduce the notion of Rees product of posets, which is a combinatorial analog of the Rees construction for semigroup algebras. A particular example that they consider  
can be viewed as a weighted version of the Boolean algebra $\mathcal B_n$, i.e., the  lattice of subsets of the set $[n]:=\{1,2,\dots,n\}$.   A  {\em weighted set} $B^v$  is a  finite set $
B$ together with a weight $v \in \{0,1,\dots,|B|-1\}$ attached (the empty set does not get any weight).  The {\em weighted Boolean algebra}  $\mathcal {WB}_n$ consists of the weighted subsets of $[n]$ with a certain order relation that reduces to the usual inclusion order on subsets when the weights are ignored.

 The poset $\mathcal {WB}_n$ is pure (or graded) of length $n$, but  is not bounded.   Its minimum element is $$\hat 0 = \emptyset$$  and its maximal elements are 
 $$[n]^0, [n]^1, \dots, [n]^{n-1}.$$  The Hasse diagram of  $\mathcal {WB}_3$ is shown in Figure~\ref{figweightedboolean_n3}, where the  commas and the set brackets have been omitted. 
 
 \begin{figure}[ht]

\begin{center} 
\begin{tikzpicture}[line join=bevel,scale=1]

\tikzstyle{every node}=[inner sep=0pt, scale=0.8, minimum width=4pt]

\node (nempty) at (0,0)  {$\emptyset$};
\node (n10) at (-3,2)  {$1^0$};
\node (n20) at (0,2)  {$2^0$};
\node (n30) at (3,2)  {$3^0$};

\node (n120) at (-5,4)  {$12^0$};
\node (n130) at (-3,4)  {$13^0$};
\node (n230) at (-1,4)  {$23^0$};
\node (n121) at (1,4)  {$12^1$};
\node (n131) at (3,4)  {$13^1$};
\node (n231) at (5,4)  {$23^1$};

\node (n1230) at (-3,6)  {$123^0$};
\node (n1231) at (0,6)  {$123^1$};
\node (n1232) at (3,6)  {$123^2$};

\draw (nempty)--(n10);
\draw (nempty)--(n20);
\draw (nempty)--(n30);

\draw (n10)--(n120);
\draw (n10)--(n130);
\draw (n10)--(n121);
\draw (n10)--(n131);

\draw (n20)--(n120);
\draw (n20)--(n230);
\draw (n20)--(n121);
\draw (n20)--(n231);

\draw (n30)--(n130);
\draw (n30)--(n230);
\draw (n30)--(n131);
\draw (n30)--(n231);

\draw (n120)--(n1230);
\draw (n120)--(n1231);

\draw (n130)--(n1230);
\draw (n130)--(n1231);

\draw (n230)--(n1230);
\draw (n230)--(n1231);

\draw (n121)--(n1231);
\draw (n121)--(n1232);

\draw (n131)--(n1231);
\draw (n131)--(n1232);

\draw (n231)--(n1231);
\draw (n231)--(n1232);

\end{tikzpicture}
\end{center}
\caption[]{Weighted Boolean algebra $\mathcal {WB}_3$}\label{figweightedboolean_n3}
\end{figure}
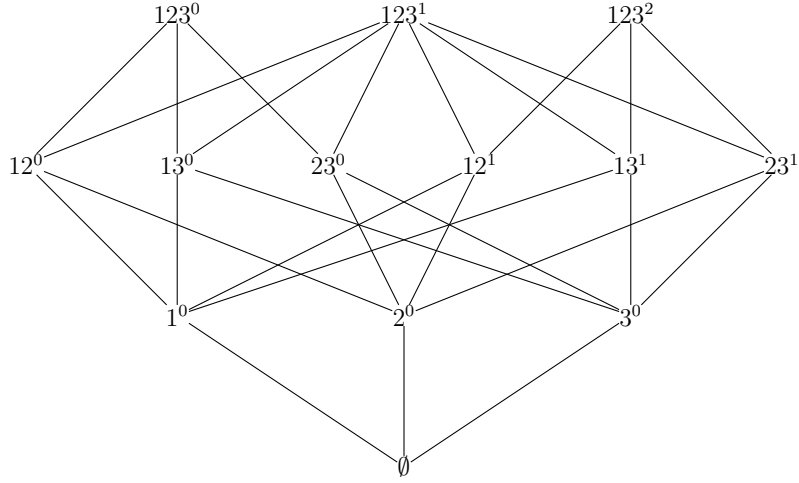
 
 In \cite{ShareshianWachs2009} Shareshian and Wachs  prove that   $\mathcal{WB}_n \cup \{\hat 1\}$, where $\hat 1$ is an attached maximum element, is  EL-shellable and  that  the classical Eulerian polynomial $A_n(t)$  arises when one applies the M\"obius function $\mu_{\mathcal{WB}_n}$ to the maximal intervals of ${\mathcal{WB}_n}$.  More precisely they prove 
\begin{equation}  \label{intro:equation:boolean}  (-1)^n \sum_{j=0}^{n-1} \mu_{\mathcal{WB}_n}(\hat 0, [n]^j) \,  t^j =   A_n(t):=\sum_{\sigma \in \mathfrak S_n} t^{\des(\sigma)},\end{equation}
 where  $\mathfrak S_n$ is the symmetric group on $[n]$ and $\des$ is the descent statistic defined by $$\des(\sigma):=| \{i \in [n-1] : \sigma(i) > \sigma(i+1) \}|.$$
  
A weighted version of the partition lattice $\Pi_n$, analogous to the weighted Boolean algebra $\mathcal {WB}_n$, was proposed by Dotsenko and Khoroshkin in \cite{DotsenkoKhoroshkin2007} as a tool for studying  the  free Lie algebra with two compatible brackets.    In \cite{GonzalezDLeonWachs2016} Gonz\'alez D'Le\'on and Wachs  carry out a study of this poset and  its connection to  the free Lie algebra with two compatible brackets.  
 
 A {\em weighted partition} of $[n]$ is a set $\{B_1^{v_1},B_2^{v_2},...,B_t^{v_t}\}$
where 
\begin{itemize}
\item $\{B_1,B_2,...,B_t\}$ is a partition of $[n]$
\item each $B_i^{v_i}$ is a weighted set, i.e.,  $v_i \in \{0,1,2,...,|B_i|-1\}$ for all~$i$.
\end{itemize}
The {\em poset of weighted partitions} $\mathcal W\Pi_n$  is the set of weighted partitions of $[n]$ with 
a certain order relation that reduces to the usual refinement order on partitions when the weights are ignored.\footnote{Precise definitions are given in Section~\ref{section:definitions}.} \footnote{The notation $\mathcal W\Pi_n$ used in this paper is different from that of \cite{GonzalezDLeonWachs2016} where the notation $\Pi_n^w$ is used instead.}
The poset $\mathcal W\Pi_n$   is pure of length $n-1$, but is not bounded. Its minimum element is 
\begin{equation} \label{equation:min_elt} \hat 0:= \{\{1\}^0,\{2\}^0,\dots, \{n\}^0\}\end{equation}
and it has  $n$
maximal elements 
\begin{equation} \label{equation:max_elt} \{[n]^0\}, \, \{[n]^1\}, \dots,  \{[n]^{n-1}\}.\end{equation}  

As is customary for ordinary partitions, we sometimes use the notation $B_1^{v_1} |B_2^{v_2}|\cdots |B_k^{v_k}$ with set brackets omitted to denote the weighted partition $\{B_1^{v_1}, \dots, B_k^{v_k}\}$.  The Hasse diagram of  $\mathcal W\Pi_3$ is shown in Figure~\ref{fign3k2}.

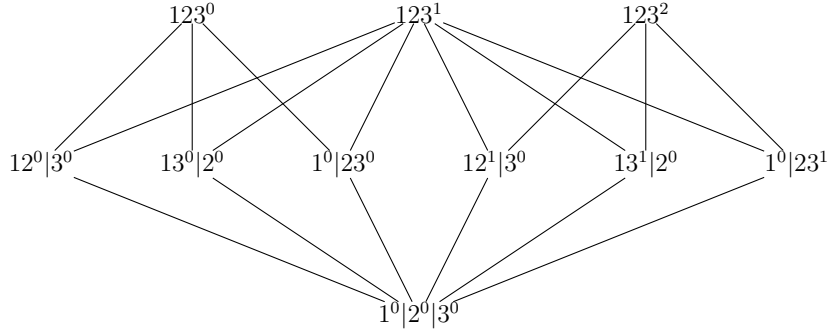
\begin{figure}[ht]

\begin{center} 
\begin{tikzpicture}[line join=bevel,scale=1]

\tikzstyle{every node}=[inner sep=0pt, scale=0.8, minimum width=4pt]
\node (n1232) at (3,4) {$123^{2}$};
  \node (n13020) at (-3,2) {$13^{ 0}| 2^{ 0}$};
  \node (n102030) at (0,0)  {$1^{0}| 2^{0}| 3^{0}$};
  \node (n1231) at (0,4) {$123^{1}$};
  \node (n12030) at (-5,2) {$12^{0}| 3^{0}$};
  \node (n13120) at (3,2)  {$13^{1}| 2^{0}$};
  \node (n1230) at (-3,4){$123^ {0}$};
  \node (n10230) at (-1,2)  {$1^{0}| 23^{0}$};
  \node (n12130) at (1,2)  {$12^{ 1}| 3^{0}$};
  \node (n10231) at (5,2) {$1^{0}| 23^ {1}$};

  \draw (n1231) -- (n10230) ;
  \draw [] (n13020) -- (n102030);
  \draw [] (n1232) -- (n13120);
  \draw [] (n1231)-- (n13020);
  \draw [] (n10230)--(n102030);
  \draw [] (n1230) -- (n10230);
  \draw [] (n1231) -- (n13120);
  \draw [] (n12030)-- (n102030);
  \draw [] (n1231) --(n12130);
  \draw [] (n1232) -- (n12130);
  \draw [] (n13120) --(n102030);
  \draw [] (n1231) --(n10231);
  \draw [] (n1230) -- (n13020);
  \draw [] (n1230)  -- (n12030);
  \draw [] (n12130) --  (n102030);
  \draw [] (n1232)  --  (n10231);
  \draw [] (n10231)  --  (n102030);
  \draw [] (n1231) -- (n12030);

\end{tikzpicture}
\end{center}
\caption[]{Weighted partition poset $\mathcal W\Pi_3$}\label{fign3k2}
\end{figure}

In \cite{GonzalezDLeonWachs2016} Gonz\'alez D'Le\'on and Wachs  prove that   $\mathcal{W}\Pi_n \cup \{\hat 1\}$  is   EL-shellable and   that  the {\em tree-Eulerian polynomial} (cf. \cite{GonzalezDLeon2016-2}), studied by Gessel and Seo  \cite[Theorem 9.1]{GesselSeo2004},  arises when one applies the M\"obius function $\mu_{\mathcal{W}\Pi_n}$ to the maximal intervals of ${\mathcal{W}\Pi_n}$.  More precisely, they prove 
$$   \sum_{j=0}^{n-1} \mu_{\mathcal{W}\Pi_n}(\hat 0, \{[n]^j\}) \,  t^j =  (-1)^{n-1} \sum_{T \in \mathcal T_n} t^{\des(T)},$$
 where  $\mathcal T_n$ is the set of (nonplanar) rooted trees on 
node set $[n]$ and $\des(T)$ is the \emph{descent} statistic on $T\in \mathcal T_n$ that gives the number of 
pairs  $(p,c)$ for which $p$ is the parent of $c$ in $T$ and $p>c$.

In this paper we introduce a weighted version of the bond lattice of a graph, which generalizes the weighted partition poset.  The {\em bond lattice} $\Pi_G$ of a graph $G$ on vertex set $[n]$ is the induced subposet of $\Pi_n$ consisting of the partitions  whose blocks induce connected subgraphs of $G$.  For the path graph $P_3= ([3], \{\{1,2\},\{2,3\}\})$, the Hasse diagram of  $\Pi_{P_3}$ is given in Figure~\ref{figure:example_bond_lattice}. Note that the partition $13|2$ does not belong to $\Pi_{P_3}$ since the induced subgraph $P_3|_{\{1,3\}}$ is not connected.

Bond lattices play an important role in the study of chromatic polynomials.  Indeed,  a formula of Whitney \cite{Whitney1932} for the chromatic polynomial  $\chi_G(t)$ of $G$ is given by
\begin{equation} \label{equation:chromatic} \chi_G(t) = \sum_{\pi \in \Pi_G} \mu_{\Pi_G}(\hat 0, \pi) t^{|\pi|}.\end{equation}
Through the weighted  version of the bond lattice, we  introduce and study a  different  polynomial graph invariant. 

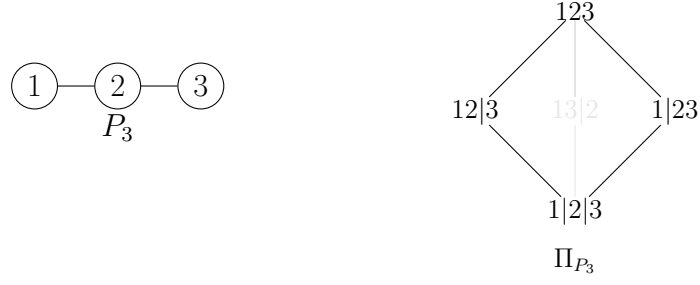
\begin{figure}
    \centering
    \begin{tikzpicture}[line join=bevel,scale=1.1]
\begin{scope}[scale=0.5,yshift=2cm]
        \node  at (-4,0) { $P_3$};
\tikzstyle{every node}=[draw,circle,scale=0.7]
\node (n0) at (-6,1) {\Large $1$};
\node (n1) at (-4,1) {\Large $2$};
\node (n2) at (-2,1) {\Large $3$};
\draw (n0) -- (n1) -- (n2);
\end{scope}
\begin{scope}[xshift=3.5cm,scale=0.6]

\tikzstyle{every node}=[inner sep=0pt, scale=0.85, minimum width=4pt]
 \node (n102030) at (0,0)  {$1 | 2 | 3$};
\node at (0,-1) {$\Pi_{P_3}$}; 
  \node (n12i30) at (-2,2) {$12 | 3$};
  \node[color=black!10] (n13i20) at (0,2) {$13 | 2$};
  \node (n1023i) at (2,2)  {$1 | 23$};

 \node (n123ii) at (0,4){$123$};
 \draw (n123ii) -- (n1023i) ;
 \draw[color=black!20] (n123ii)-- (n13i20);
  \draw (n123ii) -- (n12i30);  

  \draw [color=black!10] (n13i20) -- (n102030);
  \draw [] (n12i30)-- (n102030);
  \draw [] (n1023i)--(n102030);

\end{scope}

\end{tikzpicture}
    \caption{Bond lattice of $P_3$.}
    \label{figure:example_bond_lattice}
\end{figure}

Let $G$ be a graph on vertex set $[n]$. We define the \emph{weighted bond poset} $\mathcal {W}\Pi_G$ of  $G$ to be the induced subposet of $\mathcal W\Pi_n$ consisting of weighted partitions $\bpi$ whose underlying partitions $\pi$ are in $\Pi_G$.   Note that for the complete graph $K_n$, we have  
$$\mathcal W\Pi_{K_n} = \mathcal W\Pi_n.$$ 

The weighted bond poset $\mathcal W\Pi_G$ is pure, but not bounded. Its minimum element is the same as that for $\mathcal W\Pi_n$,  given in (\ref{equation:min_elt}). If $G$ is connected, its maximal elements are also the same as those for $\mathcal W\Pi_n$, which are given in (\ref{equation:max_elt}). For the path graph $P_3$, the Hasse diagram of  $\mathcal W\Pi_{P_3}$ is given in Figure~\ref{figure:weighted_bond_poset_P3r2}.  

\begin{figure}
    \centering
    \begin{tikzpicture}[line join=bevel,scale=1.1]
\begin{scope}[scale=0.5,yshift=2cm]
        \node  at (-4,0) { $P_3$};
\tikzstyle{every node}=[draw,circle,scale=0.7]
\node (n0) at (-6,1) {\Large $1$};
\node (n1) at (-4,1) {\Large $2$};
\node (n2) at (-2,1) {\Large $3$};
\draw (n0) -- (n1) -- (n2);
\end{scope}
\begin{scope}[xshift=3.5cm,scale=0.6]

\tikzstyle{every node}=[inner sep=0pt, scale=0.85, minimum width=4pt]
 \node (n102030) at (-2.5,0) {$1^{0}| 2^{0}| 3^{0}$};

  \node (n12i30) at (-5.5,2) {$12^{0}| 3^{0}$};
  
  \node (n1023i) at (-3.5,2) {$1^{0}| 23^{0}$};

 \node (n12j30) at (-1,2) {$12^{1}| 3^{0}$};
 
  \node (n1023j) at (1,2) {$1^{0}| 23^{1}$};

 \node (n123ii) at (-4.5,4) {$123^ {0}$};
 \node (n123ij) at (-2.5,4) {$123^ {1}$};

 \node (n123jj) at (0,4) {$123^ {2}$};

 \draw (n123ii) -- (n1023i) ;

  \draw (n123ii) -- (n12i30);  

\draw (n123jj) -- (n1023j) ;

  \draw (n123jj) -- (n12j30);  

 \draw (n123ij) -- (n1023i) ;

  \draw (n123ij) -- (n12i30);  
\draw (n123ij) -- (n1023j) ;

  \draw (n123ij) -- (n12j30);
	  \draw [] (n12i30)-- (n102030);
  \draw [] (n1023i)--(n102030);
 \draw [] (n1023j)  --  (n102030);
  \draw [] (n12j30) --  (n102030);

\end{scope}
\end{tikzpicture}
    \caption{Weighted bond poset $\mathcal W\Pi_{P_3}$}
    \label{figure:weighted_bond_poset_P3r2}
\end{figure}
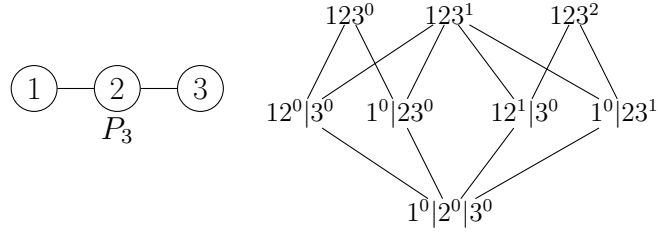

 Now let $G$ be a  connected graph on vertex set $[n]$ and  define the {\em M\"obius polynomial} of $G$ to be
 $$\mu_G(t):=\sum_{j=0}^{n-1} \mu_{\mathcal W\Pi_G} (\hat 0, \{[n]^j\}) t^j.$$
 For graphs $G$ that are not connected,  the sum is taken over the maximal elements of $\mathcal W\Pi_G$ in the definition of $\mu_G(t)$.
 The constant term $\mu_G(0)$ is the M\"obius invariant $\mu(\Pi_G)$ of the bond poset $\Pi_G$.
 From Figures~\ref{fign3k2} and~\ref{figure:weighted_bond_poset_P3r2}, we see that 
 \begin{equation} \label{equation:K_3_P_3} \mu_{K_3}(t) = 2+5t+2t^2 \,\, \mbox{ and } \,\, \mu_{P_3}(t) = 1+3t+t^2.\end{equation}
 
 We show  that if $G$ is the path graph on $[n]$ then 
 $(-1)^{n-1}\mu_G(t)$ is a Narayana polynomial and  if  $G$ is the star graph on $[n]$ then $(-1)^{n-1}\mu_G(t)$ is a polynomial, arising in the work of Postnikov, Reiner, and Williams in \cite{PostnikovReinerWilliams2008}, called the binomial-Eulerian polynomial in \cite{ShareshianWachs2020}.\footnote{Definitions are given in Sections~\ref{subsection:path} and~\ref{subsection:star}.}   As was discussed in \cite{PostnikovReinerWilliams2008}, the Narayana polynomials and the binomial-Eulerian polynomials are $h$-polynomials of certain    graph associahedra.  These M\"obius polynomials  motivated the following conjecture.\footnote{See Remark~\ref{remark:h_T} for a recent development.}

\begin{conjecture}\label{conjecture:forest_conjecture} 
Let $G$ be a forest  with $n$ vertices and $k$ connected components.  Let $\mathcal P_G$ be the graph associahedron associated with $G$ and let  $h_{\P_{G}}(t)$ denote its $h$-polynomial.  Then 
\begin{equation*} (-1)^{n-k} \mu_{G}(t) = h_{\P_{G}}(t). \end{equation*}
  \end{conjecture}

 In \cite{PostnikovReinerWilliams2008} Postnikov, Reiner, and Williams establish the important $\gamma$-positivity property for $h_{\mathcal \mathcal P_G}(t) $ when $G$ is a {\em chordal graph} in  order to  confirm, for this particular class of polytopes, a well known conjecture of  Gal \cite{Gal2005}, which asserts that every  simple flag polytope is $\gamma$-positive.
 The $\gamma$-positivity property implies unimodality (and palindromicity), which was established by Stanley~\cite{Stanley1980} for all convex polytopes.  In this paper we prove the analogous result for $\mu_G(t)$.  
  
  \begin{theorem}\label{intro_theorem:gamma} If $G$ is a chordal graph with $n$ vertices and $k$ connected components  then
 $ (-1)^{n-k}\mu_G(t)$ is $\gamma$-positive.
 \end{theorem}

 Our proof of Theorem~\ref{intro_theorem:gamma} relies on an EL-labeling of the weighted bond poset of a chordal graph and the theory of lexicographic shellability.  
Previously this technique was  used in \cite{ShareshianWachs2009,LinussonShareshianWachs2012,GonzalezDLeon2018}  to establish $\gamma$-positivity of  variations and generalizations of $(-1)^n\sum_{j=0}^{n}\mu_{\mathcal {WB}_n}(\hat 0,[n]^j) t^j$   and in \cite{GonzalezDLeonWachs2016} to establish $\gamma$-positivity of $(-1)^{n-1} \mu_{K_n}(t)$.

Our results hold in greater generality. In \cite{GonzalezDLeon2016} Gonz\'alez D'Le\'on considers  a multiweighted generalization of the weighted partition poset $\mathcal W\Pi_n$ in order to study the multilinear component of the free Lie algebra with more than two  compatible brackets.  In \cite{GonzalezDLeon2018} he also introduces  a multiweighted generalization of a variant of the weighted Boolean algebra $\mathcal{WB}_n$  in order to study the multilinear component of colored exterior algebras.  In this paper we consider an analogous multiweighted generalization of the weighted bond posets $\mathcal{W}\Pi_G$.  This yields a symmetric function generalization $M_{G}(\xx)$ of the M\"obius polynomial $\mu_G(t)$, which we call the {\em M\"obius symmetric function} of $G$.  Indeed, the specialization $M_{G}(\xx)(1,t,0,\dots)$ yields $\mu_G(t)$.  The specialization $M_{G}(\xx)(1,0,0,\dots)= \mu_G(0)$ yields the M\"obius invariant $\mu(\Pi_G)$.

When $G$ is the complete graph $K_n$,  we have that $(-1)^{n-1}M_G(\xx)$ is equal to the symmetric function studied in \cite{GonzalezDLeon2016}, whose  coefficients give the dimensions of  components of the free Lie algebra with multiple compatible brackets. When $G$ is the star graph with $n$ vertices, $(-1)^{n-1} M_G(\xx)$ is equal to the symmetric function studied in \cite{GonzalezDLeon2018}, whose  coefficients give the dimensions of  components of the colored exterior algebra.
We show that when $G$ is equal to the path graph with $n$ vertices, $(-1)^{n-1}\omega M_G(\xx)$  is the  parking function symmetric function of Haiman \cite{Haiman1994}, where $\omega$ is the standard involution on the ring of symmetric functions taking the complete homogeneous symmetric function $h_n$ to the elementary symmetric function $e_n$.  

We also obtain a multiweighted generalization  of Theorem~\ref{intro_theorem:gamma}  whose proof relies on an EL-labeling of the multiweighted bond poset of a chordal graph.

\begin{theorem} \label{intro_theorem:e_pos} If $G$ is a chordal graph with $n$ vertices and $k$ connected components  then the symmetric function
 $ (-1)^{n-k}M_G(\xx)$ is $e$-positive, in the sense that when expanded in the basis of elementary symmetric functions $\{e_\lambda: \lambda \vdash n-1\}$, the coefficients are nonnegative.
\end{theorem}

 We conjecture that Theorem~\ref{intro_theorem:e_pos} is true for all graphs.  
 
 A symmetric function generalization of the chromatic polynomial can also be obtained from the multiweighted bond posets.  The maximal degree component of this symmetric function is $M_G(\xx)$.  For chordal graphs $G$, we establish a property that we call alternating $e$-positivity and we conjecture that it holds for all graphs.   
 We also present a conjecture asserting that this symmetric function has a property that we call Schur-log-concavity. This conjecture generalizes Huh's celebrated log-concavity result for the chromatic polynomial \cite{Huh2018}.

The paper is organized as follows. Sections~\ref{section:definitions} - \ref{section:gamma_positive}  deal with the M\"obius polynomial $\mu_G(t)$ and Sections~\ref{section:multiweight} - ~\ref{section:new_chromatic_symmetric_function} deal with the symmetric function generalizations of the M\"obius polynomial and the chromatic polynomial.

 In Section~\ref{section:definitions}, we present the definitions and basic properties  of the weighted bond lattice and M\"obius polynomial of a graph. 
In Section~\ref{section:recurrence}, we present recurrence relations which   enable us to show that $(-1)^{n-1}\mu_G(t)$ is the tree-Eulerian polynomial when $G$ is the complete graph, the Narayana polynomial when $G$ is the path graph, and the binomial-Eulerian polynomial when $G$ is the star graph.  We also discuss in Section~\ref{section:recurrence} the connection with graph associahedra. 
In Section~\ref{section:gamma_positive}, we present  $\gamma$-positivity results for M\"obius polynomials of chordal graphs, which are analogous to those of Postnikov, Reiner and Williams \cite{PostnikovReinerWilliams2008} for $h$-polynomials of graph associahedra of chordal graphs.  Proofs of some of the results in Sections~\ref{section:recurrence}  and~\ref{section:gamma_positive}  are deferred to Sections~\ref{section:multiweight} and~\ref{section:e_positivity}
where more general results are proved.

 In     Section~\ref{section:multiweight},  we introduce the  multiweighted generalization of the weighted bond poset and establish recurrence relations for the M\"obius symmetric functions, which enable us to derive the relationship between  the M\"obius symmetric function of the path graph and  the parking function symmetric function. Section~\ref{section:e_positivity} deals with EL-shellability of the multiweighted bond poset and its $e$-positivity implications, which include Theorem~\ref{intro_theorem:e_pos}.    The symmetric function generalization of the chromatic  polynomial is  discussed in Section~\ref{section:new_chromatic_symmetric_function}.

\section{Definitions and basic properties} \label{section:definitions}
Let $\Pi_S$ denote the lattice of partitions of the finite set $S$ partially ordered by refinement, that is $\pi \le \pi^\prime$ in $\Pi_S$  if  blocks of $\pi$ are merged to obtain blocks of $\pi^\prime$.  See \cite{Stanley2012} for basic information,  terminology and notation on partially ordered sets.

A {\em weighted partition} of a set $S$ is a set  $\bpi =\{B_1^{v_1},B_2^{v_2},...,B_k^{v_k}\}$ of weighted blocks such that  $\{B_1,B_2,...,B_k\}$ is a partition of $S$ and $v_i \in \{0,1,2,...,|B_i|-1\}$ for all $i\in [k]$.  We refer to $\pi = \{B_1,B_2,...,B_k\}$ as the {\em underlying partition} of $\bpi$ and $v_i$ as the weight of the block $B_i$.
The {\it poset of weighted partitions} $\mathcal W\Pi_S$  is the set of weighted partitions of $S$ with 
order relation given by 
$$\{A_1^{w_1},A_2^{w_2},...,A_j^{w_j}\}\le\{B_1^{v_1}, B_2^{v_2},...,B_k^{v_k}\}$$ if the following
conditions hold:
\begin{itemize}
 \item $\{A_1,A_2,...,A_j\} \le \{B_1,B_2,...,B_k\}$ in $\Pi_S$
 \item if $B_m=A_{i_1}\cup A_{i_2}\cup ... \cup A_{i_l} $ then 
 $v_m-(w_{i_1} + w_{i_2} + ... + w_{i_l})\in \{0,1,...,l-1\}$.
\end{itemize}
Equivalently,  the covering relation $\lessdot$  is given by 
$$\{A_1^{w_1},A_2^{w_2},...,A_j^{w_j}\}\lessdot \{B_1^{v_1}, B_2^{v_2},...,B_k^{v_k}\}$$ if the
following conditions hold:
\begin{itemize}
 \item $\{A_1,A_2,\dots,A_j\} \lessdot \{B_1,B_2,\dots,B_k\}$ in $\Pi_S$
 \item if $B_m=A_{h}\cup A_{i}$, where $h \ne i$, then $v_m-(w_{h} + w_{i}) \in \{0,1\}$
 \item if $B_m = A_i$ then $v_m = w_i$.
 \end{itemize}
 
 For each $n \in \ZZ_{>0}$, let $\Pi_n:= \Pi_{[n]}$ and $\mathcal W\Pi_n := \mathcal W\Pi_{[n]}$, where
 $[n]:=\{1,2,\dots,n\}$.  Recall that the Hasse diagram of $\mathcal W\Pi_3$ (using the notation $B_1^{w_1}|B_2^{w_2}|,\dots, |B_k^{w_k}$ with set brackets omitted to denote the weighted partition $\{B_1^{w_1}, \dots, B_k^{w_k}\}$) is given in Figure~\ref{fign3k2}. 
 
Let $G=(V,E)$ be a finite graph. The 
{\em bond lattice} $\Pi_G$ of  $G$  is the induced subposet of $\Pi_V$ consisting of the partitions  of $V$ whose blocks induce connected subgraphs of $G$.  The \emph{weighted bond poset} $\mathcal {W}\Pi_G$ of  $G$ is defined to be the induced subposet of $\mathcal W\Pi_V$ consisting of all weighted partitions  whose underlying partitions are in $\Pi_G$.  Recall that the Hasse diagrams of $\Pi_{P_3}$ and $\mathcal W\Pi_{P_3}$ are given in Figure~\ref{figure:example_bond_lattice} and~\ref{figure:weighted_bond_poset_P3r2}, respectively.

Let $K_n$ be the complete graph on vertex set $[n]$.  Since 
$$\mathcal W\Pi_{K_n} = \mathcal W\Pi_n,$$   the weighted bond poset of a graph is a more general structure than the weighted partition poset.

For $ \balpha \le \bbeta$ in $\mathcal W \Pi_G$,   let $$[\balpha , \bbeta]_G=\{ \bpi \in \mathcal W \Pi_G : \balpha \le \bpi \le \bbeta\}.$$  When $G=(V,E)$ is connected, for convenience, we will drop the set brackets for the maximal elements of $\mathcal W \Pi_G$.  That is, we  write the weighted partition $\{V^w\}$ as $V^w$ and the maximal interval $[\hat 0, \{V^w\}]_G$ as $[\hat 0, V^w]_G$

 A poset is said to be {\it bounded} if it has a minimum element, which is denoted $\hat 0$, and a maximum element, which is denoted $\hat 1$.
 A poset is said to be {\em pure} (or {\em graded}) if all its maximal chains have the same length, where the {\em length} of a chain is one less than the number of elements in the chain.   The length $\ell(P)$ of a pure poset $P$ is the common length of its maximal chains.

 The bond lattice  $\Pi_G$ of any connected graph $G$ is pure and bounded, while the weighted bond poset $\mathcal W\Pi_G$ is pure but not bounded even when $G$ is connected.  Indeed if  $V= \{v_1,v_2, \dots,v_n\}$ then 
 $\mathcal W\Pi_G$ has minimum element
$$\hat 0:= \{\{v_1\}^0,\{v_2\}^0,\dots, \{v_n\}^0\}$$ and if $G$ is connected, $\mathcal W\Pi_G$ it has  $n$
maximal elements 
$$V^0, \, V^1, \dots,  V^{n-1},$$  
but no maximum element unless $n=1$.   If $G=(V,E)$ has connected components $G_i= (V_i,E_i)$, where $i \in [k]$, then 
the maximal elements of $\mathcal W\Pi_G$ are all the weighted partitions whose underlying partition is 
$\{V_1,\dots,V_k\}$.  Note that $$\ell(\mathcal W\Pi_G)= \ell(\Pi_G) = |V|-k.$$

We have the following observation whose straightforward proof is left to the reader.
\begin{proposition} \label{prop:product} Let $G$ be a graph whose connected components are  $G_1,\dots, G_k$.
Then the map $$\varphi: \mathcal W \Pi_{G_1} \times \cdots \times \mathcal W \Pi_{G_k} \to \mathcal W \Pi_G $$  defined by $\varphi(\bpi_1,\dots, \bpi_k)= \bpi_1 \cup  \cdots \cup \bpi_k$ is a poset isomorphism.
\end{proposition}

Recall that the \emph{M\"obius function} of a poset $P$ is defined recursively on intervals $[x,y]$ of $P$ by
\begin{equation}\label{equation:definition_mobius}
\mu_P(x,y)=\begin{cases} 1 &\mbox{if } x=y, \\ 
-\displaystyle \sum_{x\le z < y} \mu_P(x,z) & \mbox{if } x < y. \end{cases}
\end{equation}
Define the {\em M\"obius polynomial} of a graph $G$ by
$$\mu_G(t) := \sum_{\bpi \in  \Max(\mathcal W \Pi_G)} \mu_{\mathcal W \Pi_G}(\hat 0, \bpi) t^{w(\bpi)},$$
where  $\Max(P)$ is the set of maximal elements of  a poset $P$, and $w(\bpi)$ is the sum of the weights of the blocks of $\bpi$ for any weighted partition $\bpi$.  If $G$ is a connected graph on vertex set $V$ then
\begin{equation} \label{equation:mu_def_connected}  \mu_G(t) = \sum_{j=0}^{|V|-1} \mu_{\mathcal W \Pi_G}(\hat 0, V^j)\, t^j.\end{equation}

The following proposition, which is a consequence of Proposition \ref{prop:product} and of the multiplicative property of the M\"obius function (see \cite[Proposition~3.8.2]{Stanley2012}) tells us that we can focus on M\"obius polynomials of connected graphs.

\begin{proposition} \label{proposition:product_components} Let $G$ be a graph whose connected components are $G_1,\dots,G_k$.  Then
$$ \mu_G(t) =  \prod_{i=1}^k \mu_{G_i} (t).$$
\end{proposition}

\begin{proof} For each $i \in [k]$, let $V_i$ be the vertex set of $G_i$.  The maximal elements of  $\mathcal W\Pi_G$ are of the form $\{V_1^{j_1},V_2^{j_2}, \dots, V_k^{j_k} \}$, where $0 \le j_i \le |V_i|-1$ for each $i$.  It follows that

$$\mu_G(t) = \sum_{j_1 =0}^{|V_1|-1}  \sum_{j_2 =0}^{|V_2|-1} \dots \sum_{j_k=0}^{|V_k|-1} \mu_{\mathcal W\Pi_G}(\hat 0, \{V_1^{j_1}, V_2^{j_2}, \dots, V_k^{j_k} \}) t^{j_1+j_2+ \dots + j_k}.$$
By Proposition~\ref{prop:product}
$$ \mu_{\mathcal W\Pi_G}(\hat 0, \{V_1^{j_1},  \dots, V_k^{j_k} \}) = \mu_{\times_{i=1}^k \mathcal W\Pi_{G_i}} ((\hat 0_1,\cdots,\hat 0_k),( \{V_1^{j_1}\},\dots,\{V_k^{j_k}\})),$$
where $\hat 0_i$ is the minimum element of $\mathcal W\Pi_{G_i}$. Now by the multiplicative property of the M\"obius function (see \cite[Proposition~3.8.2]{Stanley2012}),
$$\mu_{\times_{i=1}^k \mathcal W\Pi_{G_i}} ((\hat 0_1,\cdots,\hat 0_k),( \{V_1^{j_1}\},\dots,\{V_k^{j_k}\})) = \prod_{i=1}^k \mu_{\mathcal W\Pi_{G_i}} (\hat 0_i, \{V_i^{j_i}\}) .$$
Putting these equations together yields,
\begin{align*} \mu_G(t) &= \sum_{j_1 =0}^{|V_1|-1}  \sum_{j_2 =0}^{|V_2|-1} \dots \sum_{j_k=0}^{|V_k|-1} \left (\prod_{i=1}^k \mu_{\mathcal W\Pi_{G_i}} (\hat 0_i, \{V_i^{j_i}\})\right) t^{j_1+j_2+ \dots + j_k} \\
&= \prod_{i=1}^k \sum_{j_i=0}^{|V_i|-1} \mu_{\mathcal W \Pi_{G_i} } (\hat 0, V_i^{j_i}) t^{j_i} \\
&=\prod_{i=1}^k \mu_{G_i} (t).\qedhere
\end{align*}
\end{proof}

\begin{proposition} \label{proposition:intervals}Let $G$ be a connected graph on vertex set $V$.  Then for all $j=0,\dots, |V|-1$,
$$ [\hat 0, V^j]_G \cong [\hat 0, V^{|V|-1-j}]_G,$$
and both intervals  
$[\hat 0, V^0]_G$ and $ [\hat 0, V^{|V|-1}]_G$
are isomorphic to  the bond lattice $\Pi_G$.
\end{proposition}

\begin{proof}It is not difficult to see that the map from $[\hat 0, V^j]_G$ to $ [\hat 0, V^{|V|-1-j}]_G$ which takes 
$\{B_1^{w_1},\dots, B_k^{w_k}\}$ to $\{B_1^{|B_1|-1-w_1},\dots, B_k^{|B_k| -1-w_k}\}$ is a well-defined  poset isomorphism.  Since  all the blocks of the weighted partitions in $[\hat 0, V^0]_G$ have weight $0$, the map that removes these  weights is an isomorphism from $[\hat 0, V^0]_G$ to $ \Pi_G$. 
\end{proof}

\begin{proposition} \label{corollary:palindromic} Let $G$ be a  graph with $n$ vertices and $k$ connected components.  Then
\begin{enumerate}
\item  $\mu_G(t)$ has degree  $n-k$.
\item $\mu_G(0)$ is equal to the M\"obius invariant $\mu(\Pi_G)$ of the bond lattice $\Pi_G$.
\item  $\mu_G(t)$ is palindromic in the sense that the coefficient of $t^j$ equals the coefficient of $t^{n-k-j}$ for all $j = 0,\dots,n-k$.
\end{enumerate}
\end{proposition}

\begin{proof} Item (1) for connected graphs  follows from Equation~\eqref{equation:mu_def_connected}, 
Proposition~\ref{proposition:intervals}, and the fact that the M\"obius 
invariant of a geometric lattice is never $0$; see \cite[Exercise 3.100b]{Stanley2012}. 
Indeed, by \eqref{equation:mu_def_connected}, the degree of $\mu_G(t)$ is at most $n-1$ and by Proposition~\ref{proposition:intervals}, the coefficient of $t^{n-1}$ is $\mu(\Pi_G)$.  Thus since $\Pi_G$ is  a geometric lattice,  the coefficient of $t^{n-1}$ is not $ 0$. 
 For graphs that are not connected, item (1) follows from the connected case and 
Proposition~\ref{proposition:product_components}. 

Item (2) for connected graphs follows from Proposition~\ref{proposition:intervals}.  To prove item (2) in general, we first note that Proposition~\ref{proposition:product_components} implies that $\mu_G(0) = \prod_{i =1}^k \mu_{G_i}(0)$, where the $G_i$ are the  connected components of $G$.  Since $\Pi_G \cong \Pi_{G_1} \times \cdots \times \Pi_{G_k} $, item (2)  now follows from the multiplicative property of the M\"obius invariant and the connected case.

To prove item (3), we use Propositions~\ref{proposition:product_components} and~\ref{proposition:intervals} and the fact that products of palindromic polynomials are palindromic.
\end{proof}

 \section{Recurrence relations and examples} \label{section:recurrence}
 We present recurrence relations for the M\"obius polynomials, which  can be used to study specific examples.

\subsection{Recurrence relations for the M\"obius polynomials} \label{section:recursive_structure} Let $G=(V,E)$.
For $B \subseteq V$, the  {\em restriction}  $G|_B$ is defined to be the  subgraph of $G$ induced by the vertex set $B$.  For $\pi \in \Pi_G$, the {\em contraction} $G/\pi$ is defined to be the graph  whose vertices are the blocks of $\pi$ and whose edge set is $$\{\{A,B\}\subseteq \pi \mid A \neq B, \, \{a,b\}\in E \text{ for some } a\in A \text{ and } b\in B\}.$$  Note that if $ \pi \in \Pi_G$ then the restriction $G|_B$ is connected for all $B \in \pi$ and the contraction $G/\pi$ is  connected if and only if $G$ is connected.

Although the following result can be proved directly by simply using the recursive definition of M\"obius function  
for the first recurrence relation \eqref{equation:recursion_polynomial_2-1} and the dual version of the definition of M\"obius function for the second recurrence relation \eqref{equation:recursion_polynomial_2-2}, we prove it 
as a special case of Theorem~\ref{theorem:multi_recurrence} below. 

\begin{theorem} \label{theorem:recurrence} Let  $G$ be a connected graph with more than one vertex.  Then  
\begin{align}\label{equation:recursion_polynomial_2-1}
    \mu_G(t) = -\sum_{\pi \in \Pi_G \setminus \{\hat 1\}} [|\pi|]_t\prod_{B\in \pi} \mu_{G|_B}(t)
\end{align}
and
\begin{align}\label{equation:recursion_polynomial_2-2}
    \mu_G(t) = - \sum_{\pi \in \Pi_G\setminus \{\hat{0}\}} \mu_{G/\pi}(t)\prod_{B\in \pi} [|B|]_t,
\end{align}
where $[k]_t:=1+t+\cdots+t^{k-1}$ is the $t$-analog of $k$. 
\end{theorem}

 The recurrence relation (\ref{equation:recursion_polynomial_2-2})   in the case of the complete graph $$K_n:=\left ([n],\{\{i,j\}: 1\le i < j <n\}\right )$$ and the compositional formula (see \cite[Theorem~5.1.4]{Stanley1999})
 were used in \cite{GonzalezDLeonWachs2016} to obtain the  exponential generating function formula
 \begin{equation} \label{equation:K_n_generating} \sum_{n\ge 1} \mu_{K_n}(t) \dfrac{y^n}{n!}=\left(\sum_{n\ge 1}[n]_t\dfrac{y^n}{n!}\right)^{\langle -1 \rangle} = \left(\frac{e^{ty} -e^y} {t-1}\right )^{\langle -1 \rangle},\end{equation}
where  $(\bullet)^{\langle-1\rangle}$ denotes the inverse with respect to composition (substitution) of power series.
 This and  \cite[Equation (10)]{Drake2008} were used in \cite{GonzalezDLeonWachs2016} to obtain the  formula,
\begin{equation} \label{equation:complete_formula} (-1)^{n-1} \mu_{K_n}(t)   = \prod_{i=1}^{n-1}((n-i)+it)  .\end{equation} 
 A nice consequence of this formula is
 $$(-1)^{n-1} \mu_{K_n}(1) = n^{n-1},$$ which by Cayley's formula is the number of (nonplanar) rooted trees on node set $[n]$.
 Equation~\ref{equation:complete_formula} and a refinement of Cayley's formula obtained by Gessel and Seo  \cite[Theorem 9.1]{GesselSeo2004} yield
\begin{equation} \label{equation:complete_tree} (-1)^{n-1} \mu_{K_n}(t)  = \sum_{T \in \mathcal T_n} t^{\des(T)}, \end{equation}
where $\mathcal T_n$ is the set of (nonplanar) rooted trees on node set $[n]$ and $\des(T)$ is the number of pairs  $(p,c)$, where $p$ is the parent of $c$ in $T$ and $p>c$.  
We call the polynomial on the right hand side of (\ref{equation:complete_tree}), the {\em tree-Eulerian polynomial}; cf. \cite{GonzalezDLeon2016-2}.

\subsection{The path graph and Narayana polynomials} \label{subsection:path}
Define the  {\em path graph}  to be $$P_n := ([n], \{ \{i,i+1\} : i \in [n-1]\}).$$
In this section, we apply the recurrence relations of Theorem~\ref{theorem:recurrence}  to the path graphs to show that the signless M\"obius polynomials for these graphs are  equal to the well known Narayana polynomials.

The Narayana numbers are defined by 
\begin{equation} \label{equation:narayana} N(n,j) := \frac 1 n \binom n j \binom n {j+1}\end{equation}
and the Narayana polynomials by
$$N_n(t) := \sum_{j=0}^{n-1} N(n,j) t^j .$$
The Narayana numbers refine the Catalan numbers 
$$C_n := \frac 1 {n+1}  \binom {2n}{n}.$$
Indeed, by Vandermonde's identity
$$N_n(1) = C_{n}.$$

\begin{theorem}  \label{theorem:path_narayana}For $n \ge 1$, 
$$(-1)^{n-1} \mu_{P_n}(t) = N_{n}(t) .$$
Consequently, 
$$(-1)^{n-1} \mu_{P_n}(1) = C_{n}.$$
\end{theorem}

\begin{proof}The recurrence relations of Theorem~\ref{theorem:recurrence} applied to $G=P_n$ can be expressed as 
$$
 \sum_{k=1}^n  [k]_t \sum_{(n_1,\dots, n_k) \vDash n   }   \prod_{i=1}^{k} \mu_{P_{n_i}}(t) = \delta_{n,1}$$ and
$$ \sum_{k=1}^{n}  \mu_{P_k}(t) \sum_{ (n_1,\dots, n_k) \vDash n   }   \prod_{i=1}^{k} [n_i]_t  = \delta_{n,1},$$
where $\vDash$ denotes ``is a composition of". 
Both  equations  are equivalent to the generating function formula
\begin{equation} \label{equation:path_generating} \sum_{n\ge1}(-1)^{n-1} \mu_{P_n}(t)y^n = \left(\sum_{n\ge1}(-1)^{n-1} [n]_t \, y^n\right )
^{\left\langle -1 \right\rangle} = \left ( \frac{y}{(1+y) (1+ty)}\right ) ^{\left\langle -1 \right\rangle}, \end{equation}
where again   $(\bullet)^{\langle-1\rangle}$ denotes the compositional inverse.\footnote{It is  known that $\frac{y}{(1+y) (1+ty)}$ is the compositional inverse of the generating function for the Narayana polynomials.  For the sake of completeness we are including this proof.}  
Hence, by the  Lagrange inversion formula (see \cite[Theorem 5.4.2]{Stanley1999}), 
$n(-1)^{n-1} \mu_{P_n}(t)$ is equal to the coefficient of $y^{n-1}$ in $((1+y)(1+ty))^n$.  Since
 \begin{align*} ((1+y)(1+ty))^n &= \sum_{j=0}^n \binom{n}{j} y^j \sum_{j=0}^n \binom{n}{j} (ty)^j \\ 
 &= \sum_{k\ge 0} \left (\sum_{j=0}^k \binom n j \binom n {k-j} t^j\right ) y^k, 
 \end{align*}
 the coefficient of $y^{n-1}$ is $$\sum_{j=0}^{n-1} \binom n j \binom n{n-1-j} t^j =
 \sum_{j=0}^{n-1} \binom n j \binom n {j+1} t^j .$$  This means that
$$(-1)^{n-1} \mu_{P_n}(t) = \frac 1 n  \sum_{j=0}^{n-1} \binom n j \binom n {j+1} t^j ,$$
 which by (\ref{equation:narayana}) yields the desired result.
 \end{proof}

\subsection{The star graph and binomial-Eulerian polynomials} \label{subsection:star}
Define the  {\em star graph}  to be $$St_n := ([n], \{ \{1,i\} : i \in \{2,\dots,n\}\}).$$
In this section, we use the recurrence relations of Theorem~\ref{theorem:recurrence} applied to the star graphs to show that the signless M\"obius polynomials for these graphs are  equal to a variant of the Eulerian polynomials. 

Recall that the classical Eulerian polynomials are defined by 
\begin{equation} \label{equation:combinatorial_euler} A_n(t):=\sum_{\sigma \in \mathfrak S_n} t^{\des(\sigma)},\end{equation} 
where  $$\des(\sigma):=| \{i \in [n-1] : \sigma(i) > \sigma(i+1) \}|.$$
 Following \cite{ShareshianWachs2020}, we refer to the  polynomial
$$\tilde A_n(t) := 1+ t \sum_{m=1}^n  \binom n m A_m(t)$$  as the  {\em binomial-Eulerian polynomial} of degree $n$.
We refer to the coefficient
$\tilde a_{n,j}$ of $t^j$ in $\tilde A_n(t)$ as  a binomial-Eulerian number.
The binomial-Eulerian polynomials arose in the work of Postinkov, Reiner and Williams \cite{PostnikovReinerWilliams2008} as the $h$-polynomial of a polytope that they called the stellohedron.

Just as the Eulerian numbers enumerate permutations with a given number of descents, the binomial-Eulerian numbers enumerate
partial permutations with a given number of descents, where a partial permutation of $[n]$ is a permutation of a  nonempty subset of $[n]$.
For example, when $n=3$, the partial permutations of $[n]$ written in one line notation are: 
$$1,2,3,12,21,13,31,23,32,123,132,213,231,312,321.$$

Let $\tilde {\mathfrak S}_n$ be the set of partial permutations of $[n]$.
The descent number  of a partial permutation $\sigma=\sigma_1\sigma_2\cdots \sigma_m$ is defined in the usual way as
$\des (\sigma) = |\{i \in [m-1] : \sigma_i > \sigma_{i+1} \}|$.  
 For the partial permutations listed above the respective descent numbers are:
$$0,0,0,0,1,0,1,0,1,0,1,1,1,1,2.$$  
The following combinatorial characterization of $\tilde A_n(t)$ is straightforward to prove:
   \begin{equation} \label{equation:binom_euler} \tilde A_n(t) = 1+\sum_{\sigma \in \tilde{ \mathfrak S}_n } t^{\des(\sigma)+1}.\end{equation}
It follows from this  and the list of descent numbers  above that $$\tilde A_3(t) = 1+ 7t+7t^2 + t^3.$$

\begin{theorem}  \label{theorem:star_binom_euler} For $n \ge 1$, 
$$(-1)^{n-1} \mu_{St_n}(t) = \tilde A_{n-1}(t) .$$
Consequently 
$$(-1)^{n-1} \mu_{St_n}(1) = 1+\sum_{m=1}^{n-1} (n-1)_m,$$
where $(n-1)_m$ is the falling factorial $(n-1)(n-2) \dots (n-m)$.
\end{theorem}

\begin{proof} The recurrence relations of Theorem~\ref{theorem:recurrence} for $G=St_n$ reduce to
$$\sum_{k=0}^{n-1} \binom {n-1} k [k+1]_t \,\, \mu_{St_{n-k}}(t) = \delta_{n,1}.$$
This is equivalent to the exponential generating function formula 
\begin{equation} \label{equation:star_exponential_2}  \sum_{n\ge0}\mu_{St_{n+1}}(t)\frac{y^n}{n!}  = \left ( \sum_{n \ge 0} [n+1]_t \frac {y^n}{n!} \right )^{-1} = \frac{t-1}{te^{ty} - e^y} .\end{equation}
Replacing $y$ by $-y$ yields
\begin{align}\nonumber \sum_{n\ge0}(-1)^{n} \mu_{St_{n+1}}(t)\frac{y^n}{n!}  &=  \frac{t-1}{te^{-ty} - e^{-y}} \\
\label{equation:star_exponential} &= \frac{(t-1)e^{ty}}{t-e^{(t-1)y}}.
\end{align}

It is straightforward using Euler's exponential generating formula 
\begin{equation} \label{equation:Euler} \sum_{n\ge 1} A_n(t) \frac{y^n}{n!} = \frac{e^y - e^{ty}}{e^{ty} -te^y}, \end{equation}
to show that the  formula given in (\ref{equation:star_exponential}) also holds for the exponential generating function for $\tilde A_n(t)$.
Indeed we have
\begin{align}\nonumber \sum_{n\ge 0} \tilde A_n(t) \frac{y^n}{n!} &= \sum_{n\ge 0} \left(1+ t\sum_{m=1}^n \binom n m A_m(t)\right) \frac{y^n}{n!} \\
\nonumber &= e^y + t \sum_{n\ge 1} A_n(t) \frac{y^n}{n!} \sum_{n\ge 0} \frac{y^n}{n!} \\
\nonumber &= e^y + t \frac{e^y - e^{ty}}{e^{ty} -te^y} e^y \\
&= \frac{(t-1)e^{ty}}{t-e^{(t-1)y}} \label{equation:tilde_exponential}
\end{align}
with the second to last equality following from (\ref{equation:Euler}) 
and the last equality following from straightforward manipulation.
Comparing this with (\ref{equation:star_exponential}) yields the desired result.
\end{proof}

\begin{remark} A $q$-analog and symmetric function analog of  the exponential generating function formula given in \eqref{equation:tilde_exponential} appears in   \cite[Propositions~4.3 and~3.3]{ShareshianWachs2020}  respectively. 
\end{remark}

\subsection{A connection with  graph associahedra} \label{section:connection_with_graph_associahedra}

The $h$-polynomial of a {\it simple}\footnote{Here we will be following the convention in \cite{PostnikovReinerWilliams2008}, in which the $h$-polynomial is defined for simple polytopes rather than, as is often done, for simplicial polytopes.} $d$-dimensional polytope $\mathcal P$ is defined by  
\begin{align}
h_{\mathcal P}(t)=\sum_{i=0}^d f_i(t-1)^i,    
\end{align}
where $f_i$ is the number of $i$ dimensional faces  of $\mathcal P$.  

It is well known that the Eulerian polynomials and the Narayana polynomials are $h$-polynomials
of simple polytopes known as the permutohedron and the associahedron, respectively.  
It is shown in \cite{PostnikovReinerWilliams2008} that the binomial-Eulerian  polynomials are $h$-polynomials of  simple polytopes as well. Hence by Theorems~\ref{theorem:path_narayana} and~\ref{theorem:star_binom_euler}, respectively, we have that $(-1)^{n-1} \mu_G(t)$ is the $h$-polynomial of a simple polytope when $G$ is the path graph $P_n$ or the star graph $St_n$.  It is natural to ask if there are any other graphs $G$ for which this is the case; in this section we address this question.

Graph associahedra are a family of simple polytopes that were studied by Carr and Devadoss \cite{CarrDevadoss2006} and that form part of the larger class of generalized permutahedra studied by   Postnikov, Reiner, and Williams in \cite{PostnikovReinerWilliams2008} and Postnikov in \cite{Postnikov2009}. Several ocurrances of equivalent families have appeared in the literature as well. In particular, the dual simplicial complexes of graph associahedra are included in the family of nested set complexes studied by De Concini and Procesi in \cite{DeConciniProcesi1995}. We should also mention that in recent years a great many papers have appeared with further connections between generalized permutahedra and other areas of mathematics.  
A paper we would  like to mention is that of Aguiar and Ardila  \cite{AguiarArdila2023}, in which  a beautiful structure of a Hopf monoid on generalized permutahedra is described, for which the restriction to graph associahedra yields a known Hopf monoid on graphs. 

We follow the notation and terminology used in \cite{PostnikovReinerWilliams2008, Postnikov2009}.
A \emph{building set} on  a finite set $V$ is a collection $\B \subseteq 2^{V}$ satisfying:
\begin{itemize}
    \item $\{v\} \in \B$ for all $v\in V$.
    \item If $B, B' \in \B$ and $B\cap B' \neq \emptyset$ then $B\cup B' \in \B$.
\end{itemize}
Given a graph $G$ on vertex set $V$, one can define the building set,
$$\B(G):=\{\emptyset \neq B\subseteq V \mid G|_{B}\text{ is connected}\}.$$

The \emph{nestohedron} $\P_{\B}$ associated to a building set $\B$ on $V \subseteq [n]$ is the polytope defined as the Minkowski sum
$$\P_{\B}:=\sum_{B \in \B}\Delta_{B},$$
where the \emph{Minkowski sum} of two subsets $S,T\in \RR^n$ is defined as the subset $S+T:=\{x+y \mid x\in S\text{ and } y \in T\}\subseteq \RR^n$ and $\Delta_B$ is the simplex in $\RR^n$ obtained as the convex hull of the standard unit vectors in $\{{\bf e}_i\mid i\in B\}$.  In \cite{Postnikov2009} it is shown that the nestohedron $\P_{\B}$ is a simple polytope.

Nestohedra of the form  $\P_G:=\P_{\B(G)}$, where   $G$ is a graph on $V \subseteq [n]$, are normally equivalent to what are called \emph{graph-associahedra} in \cite{CarrDevadoss2006}.
  It follows from \cite[Theorem 7.11]{Postnikov2009} that if $G_1,\dots,G_k$ are the connected components of $G$ then
\begin{equation} \label{equation:prodform}  h_{\P_{G}}(t) = h_{\P_{G_1}}(t)  \cdots h_{\P_{G_k}}(t) .\end{equation}

The graph associahedron for the complete graph is the standard permutohedron and the graph associahedron for the path graph is the standard associahedron; see \cite[Section 10]{PostnikovReinerWilliams2008}.   The graph associahedron for the star graph is called the {\em stellohedron} in
 \cite[Section 10.4]{PostnikovReinerWilliams2008}.
As was mentioned above, it is well known that the $h$-polynomials for the permutohedron and the associahedron are the Eulerian and Narayana polynomials, respectively. 
Thus
\begin{align} 
 \label{equation:h_complete} h_{\mathcal{P}_{K_n} }(t)&= A_n(t) \\
\nonumber h_{\mathcal{P}_{P_n}}(t) &= N_n(t). 
\end{align}
In \cite[Section 10.4]{PostnikovReinerWilliams2008}, it is shown that,
$$ h_{\mathcal {P}_{St_n } }(t) =  \tilde A_{n-1}(t).
$$
It follows from these equations and Theorems~\ref{theorem:path_narayana} and~\ref{theorem:star_binom_euler} that $$(-1)^{n-k} \mu_{G}(t)=h_{\P_{G}}(t)$$ is true for the special cases that $G$ is the  path graph and the star graph.  By (\ref{equation:prodform}) and Proposition~\ref{proposition:product_components},  it is also true for forests $G$ whose trees are path graphs or star graphs.

The examples above suggest the following conjecture, which was tested by computer for $n \le 7$.
For a polynomial  $p(t)=\sum_{i=0}^d a_i t^i$, we write  $p(t) \succeq 0$ if all its coefficients are nonnegative.

\begin{conjecture} \label{conjecture:h_T} 
 For any graph $G$ with $n$ vertices and $k$ connected components, we have that
 \begin{equation} \label{equation:mobius_assoc} (-1)^{n-k} \mu_{G}(t)-h_{\P_{G}}(t)\succeq 0,\end{equation}
 with equality 
\begin{equation} \label{conjecture:h_T.eq} (-1)^{n-k} \mu_{G}(t)=h_{\P_{G}}(t) \end{equation} if  (and only if)\footnote{See Proposition~\ref{proposition:necessary_forest_condition}}
 $G$ is a forest.
\end{conjecture}

\begin{remark} \label{remark:h_T} The equality part of the above conjecture was first formulated in 2014 (as was Conjecture~\ref{conjecture:recurrence_tree} below).  It was subsequently presented in several talks, including  an AMS special session talk at the 2016 Joint Mathematics Meeting (AMS abstract 1116-05-2827). The full 
conjecture appeared in a 2021 draft of the current paper. This led to subsequent work of Avila, Carrillo, and Gonz\'alez D'Le\'on, who recently 
posted in \cite[Theorem~1.6]{AvilaCarrilloGonzalezDleon2026A} a proof of an  identity analogous to (\ref{conjecture:h_T.eq}) that holds for all graphs (and building sets) and reduces to (\ref{conjecture:h_T.eq}) when the graph is a forest. The analogous identity involves a variant of $\mu_G(t)$ in which the lattice of  spanning subgraphs  of $G$ plays the role that the bond lattice played in the definition of $\mu_G(t)$.  A proof of the full conjecture will appear in \cite{AvilaCarrilloGonzalezDleon2026B}.
\end{remark}

The necessity part of \eqref{conjecture:h_T.eq} is a consequence of the following proposition.

\begin{proposition}\label{proposition:necessary_forest_condition} Let $G$ be a  graph with $n$ vertices and $k$ connected components.   If $(-1)^{n-k} \mu_{G}(t)$ is equal to the $h$-polynomial of a simple polytope then $G$ is a forest.
\end{proposition}
\begin{proof} The  $h$-polynomial  of every simple polytope has constant term equal to $1$. Recall from Proposition~\ref{corollary:palindromic} that the constant term  $ \mu_G(0)$ of  $ \mu_G(t)$  is equal to the M\"obius invariant  $\mu(\Pi_G)$. The result now follows from the known fact that $(-1)^{n-k}\mu(\Pi_G)=1$ if and only if $G$ is a forest.  One can prove this by using  Whitney's Broken Circuit Theorem  \cite{Whitney1932}, which tells us that $(-1)^{n-k} \mu(\Pi_G)$ is equal to the number of spanning forests of $G$ that contain no broken circuits.  (A broken circuit is  a subgraph of $G$ obtained by removing the smallest edge, under a fixed total ordering of the edges of $G$, from a cycle of $G$.)
\end{proof}

 Next we consider recurrence relations for $h$-polynomials that  are analogous to those  given in Theorem~\ref{theorem:recurrence} for  M\"obius polynomials. They  are given in the following conjecture, which, by Theorem~\ref{theorem:recurrence}, Proposition~\ref{proposition:product_components}, and~(\ref{equation:prodform}), is equivalent to the equality part of Conjecture~\ref{conjecture:h_T}
(now proved in \cite{AvilaCarrilloGonzalezDleon2026A} as noted in Remark~\ref{remark:h_T}).

\begin{conjecture}
\label{conjecture:recurrence_tree} Let $G$ be a tree with $n$ vertices.  If   $n > 1$ then
\begin{equation}\label{equation:recursion_polynomial_h2-3}
    h_{\mathcal P_{G}}(t) =   \sum_{\pi \in \Pi_G \setminus \{\hat 1\}} (-1)^{|\pi|} [|\pi|]_t\prod_{B\in \pi} 
    h_{\mathcal P_{G|_B}}(t)
\end{equation}
and
\begin{equation}\label{equation:recursion_polynomial_h2-4}
   h_{\mathcal P_{G}}(t) =   \sum_{\pi \in \Pi_G\setminus \{\hat{0}\}} (-1)^{n-|\pi|-1} h_{\mathcal P_{G/\pi}}(t)\prod_{B\in \pi} [|B|]_t.
\end{equation}
\end{conjecture}

Note that it follows from Theorem~\ref{theorem:recurrence} and Proposition~\ref{proposition:necessary_forest_condition} that the recurrence relations of Conjecture~\ref{conjecture:recurrence_tree} cannot hold more generally for all  connected graphs.

\begin{remark}  In \cite[Theorem 7.11]{Postnikov2009}, Postnikov gives a recurrence  for $h_{\mathcal P_{\mathcal B}}(t)$, where $\mathcal B$ is a  building set.  When specialized to $\mathcal B(T)$, where $T$ is a tree, this recurrence differs from both 
(\ref{equation:recursion_polynomial_h2-3}) and (\ref{equation:recursion_polynomial_h2-4}).   Indeed, for a tree $T$ on $[n]$, Postnikov's recurrence becomes
$$h_{\mathcal P_T}(t) = \sum_{ U \subsetneq [n]} (t-1)^{n-|U| -1} \prod_{R \in T|_U}  h_{\mathcal P_R}(t), $$
where the forest $T|_U$ is viewed as a set of trees.   It is natural to ask if there are recurrence relations for general connected graphs (or more generally connected building sets) that  specialize  to  (\ref{equation:recursion_polynomial_h2-3}) and (\ref{equation:recursion_polynomial_h2-4}). 
Avila, Carrillo, and Gonz\'alez D'Le\'on recently obtained such recurrence relations in \cite[Theorem~1.7]{AvilaCarrilloGonzalezDleon2026A}.
\end{remark}

\section{\texorpdfstring{$\gamma$}{}-positivity for chordal graphs} \label{section:gamma_positive}
In this section we discuss a property of polynomials in $\RR[t]$ called $\gamma$-positivity.  Postnikov, Reiner, and Williams \cite{PostnikovReinerWilliams2008} established this property for $h$-polynomials of graph associahedra of chordal graphs and here we do the same for M\"obius polynomials.  We review their $\gamma$-positivity result in Section~\ref{subsection:associahedra_gamma} for the sake of comparison with our  analogous result presented in Section~\ref{subsection:chordal_graphs}.

Let $p(t) = \sum_{j=0}^d a_j t^j \in \RR[t]$.   Recall that $p(t)$ is said to be {\em palindromic} if $a_j = a_{d-j}$ for all $j$, and $p(t)$ is said to be {\em unimodal} if  
\begin{equation} \label{equation:unimodal} 0 \le a_0 \le a_1 \le \cdots \le a_c \ge \cdots \ge a_{d-1} \ge a_d \ge 0,\end{equation} for some  $c \in \{0,1,\dots,d\}$.
Note that if $p(t)$ is palindromic and unimodal and $a_0 >0$ (which means $d$ is the degree of $p(t)$) then all the coefficients of $p(t)$ are positive, denoted $p(t) \succ 0$.

It is straightforward to see that 
 $p(t)$  is palindromic if and only if
$$p(t) = \sum_{j=0}^{\lfloor \frac{d}{2} \rfloor } \gamma_j t^j (1+t)^{d-2j} $$ 
for some $\gamma_j \in \RR$.  The coefficients $\gamma_j$ are commonly called {\em $\gamma$-coefficients}.  
If $\gamma_j \ge 0$ for all $j$ then the polynomial $p(t)$ is said to be {\em $\gamma$-positive}.\footnote{This is often called {\em $\gamma$-nonnegative}.}  It is well known and straightforward to prove that if  $p(t)$ is $\gamma$-positive then $p(t)$ is both palindromic and  unimodal.  

A celebrated result of Stanley \cite{Stanley1980} asserts that the $h$-polynomial of any convex simple polytope is unimodal as well as palindromic (which  was already well-known).
Interest in $\gamma$-positivity stems in part from Gal's conjecture \cite{Gal2005} which in the case of convex polytopes asserts that the $h$-polynomial of any simple flag polytope is $\gamma$-positive.

A graph is said to be \emph{chordal} if it does not have any induced cycles of length greater than $3$. Chordal graphs are exactly those that admit a \emph{(reverse) perfect elimination order (PEO)}, i.e., a total order of the vertices $v_1, v_2,\dots,v_n$  such that for every $i$ the set of neighbors $v_j$ of $v_i$, such that $j <i$,  form a  clique, see \cite[Section~9.2]{PostnikovReinerWilliams2008}. 

We will say that a graph $G$ on vertex set $V = \{v_1, v_2,\dots,v_n\} \subset \ZZ_{>0}$ is a {\it perfectly labeled chordal graph} if the natural order $v_1<v_2<\dots<v_n$ is a PEO for $G$.  Note that every chordal graph  is isomorphic to a perfectly labeled chordal graph.  The complete graph on $[n]$ is a perfectly labeled chordal graph, as is every increasing tree on $[n]$, where a tree on $[n]$ is {\em increasing} if when rooted at $1$, every vertex is smaller than its children.

\subsection{$h$-polynomials of chordal graph associahedra} \label{subsection:associahedra_gamma}
In  \cite{PostnikovReinerWilliams2008},  Postnikov, Reiner, and Williams establish $\gamma$-positivity of the $h$-polynomial of the graph associahedron of  any chordal graph, 
thereby confirming Gal's conjecture for this class of flag simple polytopes.\footnote{They do this more generally for chordal nestohedron.}  They do this by giving a  combinatorial characterization of the coefficients and the $\gamma$-coefficients of the $h$-polynomial   in terms of a certain set  of 
permutations.  We review this result  in this section.

For any 
graph $G$ on vertex set $[n]$,   a \emph{$G$-permutation}\footnote{Note that $\max$ is used instead of $\min$ in \cite{PostnikovReinerWilliams2008}.  Hence,  the roles of  descents and ascents are exchanged.}    is a permutation $\sigma \in \sym_n$ such that     for every $i \in [n]$, the elements $\sigma(i)$ 
and $\min \{\sigma(1),\dots,\sigma(i)\}$ are in the same connected component of 
$G|_{\{\sigma(1),\dots,\sigma(i)\}}$.  Let $\sym(G)$ be the set 
of $G$-permutations and let $\widehat{\sym}(G):= \sym(G)\cap \widehat \sym_n$, where  $\widehat \sym_n$ is the subset of $\sym_n$ consisting of permutations $\sigma$ with no  consecutive ascents; that is, for all $i \in [n-2]$, if
$\sigma(i)<\sigma(i+1)$ then $\sigma(i+1) >\sigma(i+2)$.\footnote{Note that  in \cite{PostnikovReinerWilliams2008}, $\widehat{\sym}_n$ denotes a smaller subset  in which a final descent is also required.} 
For $\sigma \in \sym_n$, let $$\asc(\sigma) = |\{i \in [n-1]  : \sigma(i) < \sigma(i+1)\}|.$$

\begin{theorem}[Postnikov, Reiner, Williams \cite{PostnikovReinerWilliams2008}]\label{theorem:h_vector}
For $n\ge 1$ and $G$ a connected perfectly labeled chordal graph on $[n]$, we have that 
\begin{equation} \label{equation:assoc} h_{\P_G}(t)=\sum_{\sigma\in \sym(G)}t^{\asc(\sigma)}\end{equation}
and 
\begin{equation} \label{equation:gamma_assoc} h_{\P_G}(t)=\sum_{\substack{\sigma\in \widehat{\sym}(G)\\ \sigma(n-1) > \sigma(n)}}t^{\asc(\sigma)}(1+t)^{n-1-2\asc(\sigma)}.\end{equation}
\end{theorem}

Now by (\ref{equation:prodform}) and the fact that products of $\gamma$-positive polynomials are $\gamma$-positive we have the following consequence. 
\begin{corollary}
 If $G$ is a chordal graph then $h_{\P_G}(t)$ is $\gamma$-positive.  
\end{corollary}

\begin{example} \label{example:G_permutation} 
(1) For $G = K_n$, we have $\mathfrak S(G) = \mathfrak S_n$.  Thus, in this case, (\ref{equation:assoc}) reduces to the  combinatorial definition of  the Eulerian polynomial given in (\ref{equation:combinatorial_euler}),
and (\ref{equation:gamma_assoc}) reduces to a well known formula of Foata and Sch\"utzenberger \cite[Theorem 5.6]{FoataSchutzenberger1970} for  the Eulerian polynomials.

(2) For $G= P_n$, we have that $\sigma \in \mathfrak S(G)$  if and only if
$\sigma$ avoids the pattern $132$, i.e.,  if $1 \le i <j < k \le n$ and $\sigma(i) <\sigma(k)$ then $\sigma(j) < \sigma(k)$. Thus, in this case, (\ref{equation:assoc}) and (\ref{equation:gamma_assoc}) reduce to  well known combinatorial formulas for  the Narayana polynomials in terms of descent numbers of pattern avoiding permutations; see \cite[Sections~2.3 and~4.3]{Petersen2015}.

(3) For $G= St_{n+1}$, we have that $\sigma \in \mathfrak S(G)$ if and only if 
$$\sigma(1) > \sigma(2) > \dots > \sigma(k) =1 $$ for some $k \in [n+1]$. Thus, in this case, (\ref{equation:assoc}) reduces to a formula equivalent to the combinatorial characterization of  the binomial-Eulerian polynomials given in (\ref{equation:binom_euler}), and (\ref{equation:gamma_assoc}) reduces to the formula
\begin{equation} \label{equation:gamma_tilde} \tilde A_n(t) =  \sum_{\substack{\sigma \in \widehat{\mathfrak S}_{n+1}\\ \sigma(n)>\sigma(n+1) \\ \sigma(1) > \sigma(2) > \dots > \sigma(k) =1}} t^{\asc(\sigma)} (1+t)^{n-1-2\asc(\sigma)}.\end{equation}
 \end{example}

\subsection{M\"obius polynomials of chordal graphs} \label{subsection:chordal_graphs}
In this subsection we present results for  M\"obius polynomials that are analogous to the results for $h$-polynomials reviewed in Section~\ref{subsection:associahedra_gamma}.  

By {\em binary tree} we mean a planar rooted tree in which every internal node has a left child and a
right child.     A {\it leaf-labeled  binary tree}  is a  binary tree whose leaves are labeled with distinct labels.  
If $A$ is the set of leaf labels of the binary tree $T$ we say that $T$ is a binary tree on leaf set $A$.
Let $\BT_n$ denote the set of  binary trees on leaf set $[n]$.

For each node $x$ of $T \in \BT_n$,  let 
\begin{itemize} 
\item $A_x $ be the subset of $[n]$ consisting of the leaf labels of the subtree rooted at $x$
\item  $L(x)$ be the left child of $x$ 
\item $R(x)$ be the right child of $x$. 
\end{itemize}
We say that $T \in \BT_n$ is {\em normalized} if 
$$\min A_{L(x)}  < \min A_{R(x)}  $$ for all internal nodes $x$ of $T$. Let $$\mathcal N_n = \{T \in \BT_n : T \mbox{ is normalized}\}.$$ 
We say that an internal node $x$ of  $T \in \mathcal N_n$ is a {\it Lyndon node} if its left child is a leaf or
\begin{equation} \label{eq:lynnode} \min A_{R(L(x))}> \min A_{R(x)}.\end{equation}
We refer to these nodes as Lyndon nodes because the classical Lyndon trees are the normalized binary trees in which all internal nodes are Lyndon.
Now let
$$\widehat{\mathcal N_n} = \{ T \in  \mathcal N_n: \forall x \in \internal(T), \mbox{ $x$ is not Lyndon $\implies L(x)$ is Lyndon} \},$$
where $\internal(T)$ is the set of internal nodes of  $T$.  See Figure \ref{figure:non_double_lyndon} for an example of a tree in $\widehat{\mathcal N_n}$ and a normalized tree not in $\widehat{\mathcal N_n}$.

\begin{figure}
    \centering
    \begin{tikzpicture}[scale=1.1]

\begin{scope}

\tikzstyle{every node}=[fill,circle, draw,inner sep=2pt, scale=1.1, minimum width=4pt,scale=1.1]

\node (v2) at (-2,-0.5) {};

\node  (v8) at (-0.5,-0.5) {};
\node  (v9) at (0,0) {};
\node (v3) at (-0.75,0.75) {};
\node [pin={[red, inner sep =0.5pt, pin distance = 8pt, pin edge={red, thick, <-}]170:\footnotesize *}
] (v4) at (-0.25,1.25) {};

\tikzstyle{every node}=[inner sep=0pt, scale=1.1, minimum width=4pt,scale=1.1]

\node (v1) at (-2.5,-1) {$1$};
\node (v6) at (-1.5,-1) {$3$};
\node (v7) at (-1,-1) {$2$};
\node (v10) at (0,-1) {$6$};
\node (v11) at (0.5,-0.5) {$4$};
\node (v5) at (0.25,0.75) {5};
\tikzstyle{every path}=[ thick]

\draw  (v1) edge (v2);
\draw  (v2) edge (v3);
\draw  (v3) edge (v4);
\draw  (v4) edge (v5);
\draw  (v2) edge (v6);
\draw  (v7) edge (v8);
\draw  (v8) edge (v9);
\draw  (v8) edge (v10);
\draw  (v9) edge (v11);
\draw  (v3) edge (v9);

\node at (-1,-1.5) {$T \in \widehat{\mathcal N_n}$};
\end{scope}

\begin{scope}[shift={(5,0)}]

\tikzstyle{every node}=[fill,circle, draw,inner sep=2pt, scale=1.1, minimum width=4pt,scale=1.1]

\node (v2) at (-2,-0.5) {};

\node  (v8) at (-0.5,-0.5) {};
\node  (v9) at (0,0) {};
\node  [pin={[red, inner sep =0.5pt, pin distance = 8pt, pin edge={red, thick, <-}]170:\footnotesize *}
] (v3) at (-0.75,0.75) {};
\node [pin={[red, inner sep =0.5pt, pin distance = 8pt, pin edge={red, thick, <-}]170:\footnotesize *}
] (v4) at (-0.25,1.25) {};

\tikzstyle{every node}=[inner sep=0pt, scale=1.1, minimum width=4pt,scale=1.1]

\node (v1) at (-2.5,-1) {$1$};
\node (v6) at (-1.5,-1) {$2$};
\node (v7) at (-1,-1) {$3$};
\node (v10) at (0,-1) {$6$};
\node (v11) at (0.5,-0.5) {$4$};
\node (v5) at (0.25,0.75) {5};
\tikzstyle{every path}=[ thick]

\draw  (v1) edge (v2);
\draw  (v2) edge (v3);
\draw  (v3) edge (v4);
\draw  (v4) edge (v5);
\draw  (v2) edge (v6);
\draw  (v7) edge (v8);
\draw  (v8) edge (v9);
\draw  (v8) edge (v10);
\draw  (v9) edge (v11);
\draw  (v3) edge (v9);

\node at (-1,-1.5) {$T \in \mathcal N_n \setminus \widehat{\mathcal N_n}$};
\end{scope}

\end{tikzpicture}
    \caption{Example of a normalized tree that is in $\widehat{\mathcal N_n}$ and one that is not. The red {\color{red} $*$} indicate nodes that are not Lyndon.}
    \label{figure:non_double_lyndon}
\end{figure}
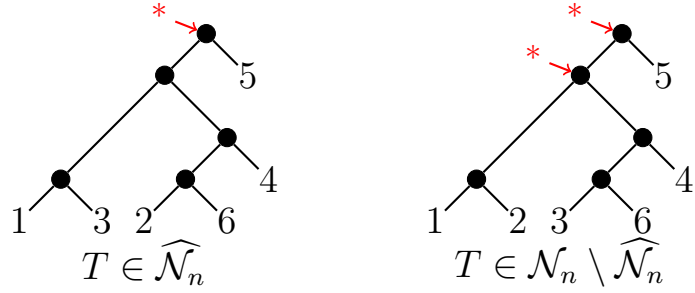

Let $G$ be a  graph on vertex set $[n]$. Define
$$\mathcal{N}_G := \{T \in \mathcal{N}_n : G|_{A_x} \mbox{ is connected } \forall  x \in \internal(T) \}$$
and
$$ \widehat{\mathcal N_G} := \mathcal{N}_G \cap \widehat{\mathcal N_n}.$$

The proof of the following theorem is deferred to Section~\ref{subsection:chordal}, where the more general Theorem~\ref{theorem:e_positive} is proved.
 
\begin{theorem} \label{theorem:gamma_positive} Let $G$ be a connected perfectly labeled chordal graph on vertex set $[n]$. Then
$$(-1)^{n-1}\mu_G(t) =  \sum_{T \in \widehat{\mathcal N_G} } t^{m(T)} (1+t)^{n-1 - 2m(T)},$$
where $m(T)$ is the number of nonLyndon internal nodes of $T$.
\end{theorem}

\begin{example}\label{example:gamma_trees} \label{example:gamma_positive_theorem} Let $G$ be $P_4$ with the additional edge $\{1,3\}$. Since $G$ is a perfectly labeled chordal graph, we can apply Theorem~\ref{theorem:gamma_positive}. In Figure~\ref{figure:example_modified_P4}, the trees of $\widehat{\mathcal N_G}$ are listed and their nonLyndon nodes are indicated.  By Theorem~\ref{theorem:gamma_positive}, we have that 
$$- \mu_G(t) = 2(1+t)^3 + 5t(1+t).$$  By eliminating the trees that are not in $\mathcal N_{P_4}$, we see that
$$- \mu_{P_4}(t) = (1+t)^3 + 3t(1+t) = N_4(t).$$
 \end{example} 

\begin{figure}
    \centering
    \begin{tikzpicture}[scale=1]

\tikzstyle{internal_nodes} = [fill,circle, draw,inner sep=2pt, scale=1.1, minimum width=4pt,scale=1]
\tikzstyle{leaves} = [inner sep=0pt, scale=1.1, minimum width=4pt,scale=1]
\tikzstyle{path_styles} = [thick]

\begin{scope}

\tikzstyle{every node} = [internal_nodes]

\node (i1) at (-2,-0.5) {};
\node (i2) at (-1.5,0) {};
\node [pin={[red, inner sep =0.5pt, pin distance = 8pt, pin edge={red, thick, <-}]170:\footnotesize *}
] (i3) at (-1,0.5) {};

\tikzstyle{every node} = [leaves]
\node (v1) at (-2.5,-1) {$1$};
\node (v2) at (-1,-0.5) {$2$};
\node (v3) at (-1.5,-1) {$3$};
\node (v4) at (-0.5,0) {$4$};

\tikzstyle{every path}=[path_styles]

\draw  (v1) edge (i1);
\draw  (v3) edge (i1);
\draw  (v2) edge (i2);
\draw  (v4) edge (i3);
\draw  (i1) edge (i2);
\draw  (i2) edge (i3);
\end{scope}

\begin{scope}[shift={(3.5,0)}]

\tikzstyle{every node} = [internal_nodes]

\node (i1) at (-2,-0.5) {};
\node [pin={[red, inner sep =0.5pt, pin distance = 8pt, pin edge={red, thick, <-}]170:\footnotesize *}
] (i2) at (-1.5,0) {};
\node (i3) at (-1,0.5) {};

\tikzstyle{every node} = [leaves]
\node (v1) at (-2.5,-1) {$1$};
\node (v4) at (-1,-0.5) {$4$};
\node (v3) at (-1.5,-1) {$3$};
\node (v2) at (-0.5,0) {$2$};
\tikzstyle{every path}=[path_styles]

\draw  (v1) edge (i1);
\draw  (v3) edge (i1);
\draw  (v4) edge (i2);
\draw  (v2) edge (i3);
\draw  (i1) edge (i2);
\draw  (i2) edge (i3);
\end{scope}

\begin{scope}[shift={(7,0)}]

\tikzstyle{every node} = [internal_nodes]

\node (i1) at (-2,0) {};
\node (i2) at (-1.5,-0.5) {};
\node (i3) at (-1.5,0.5) {};

\tikzstyle{every node} = [leaves]
\node (v1) at (-2.5,-0.5) {$1$};
\node (v4) at (-1,-1) {$4$};
\node (v3) at (-2,-1) {$3$};
\node (v2) at (-1,0) {$2$};
\tikzstyle{every path}=[path_styles]

\draw  (v1) edge (i1);
\draw  (v3) edge (i2);
\draw  (v4) edge (i2);
\draw  (v2) edge (i3);
\draw  (i1) edge (i2);
\draw  (i1) edge (i3);
\end{scope}

\begin{scope}[shift={(-1,2.5)}]

\tikzstyle{every node} = [internal_nodes]

\node (i1) at (-2,0) {};
\node (i2) at (-1.5,-0.5) {};
\node[pin={[red, inner sep =0.5pt, pin distance = 8pt, pin edge={red, thick, <-}]170:\footnotesize *}
]  (i3) at (-1.5,0.5) {};

\tikzstyle{every node} = [leaves]
\node (v1) at (-2.5,-0.5) {$1$};
\node (v3) at (-1,-1) {$3$};
\node (v2) at (-2,-1) {$2$};
\node (v4) at (-1,0) {$4$};
\tikzstyle{every path}=[path_styles]

\draw  (v1) edge (i1);
\draw  (v3) edge (i2);
\draw  (v2) edge (i2);
\draw  (v4) edge (i3);
\draw  (i1) edge (i2);
\draw  (i1) edge (i3);
\end{scope}

\begin{scope}[shift={(2,2.5)}]

\tikzstyle{every node} = [internal_nodes]

\node (i1) at (-2.5,-0.5) {};
\node[pin={[red, inner sep =0.5pt, pin distance = 8pt, pin edge={red, thick, <-}]170:\footnotesize *}
] (i2) at (-1.75,0.25) {};
\node (i3) at (-1,-0.5) {};

\tikzstyle{every node} = [leaves]
\node (v1) at (-3,-1) {$1$};
\node (v3) at (-1.5,-1) {$3$};
\node (v2) at (-2,-1) {$2$};
\node (v4) at (-0.5,-1) {$4$};
\tikzstyle{every path}=[path_styles]

\draw  (v1) edge (i1);
\draw  (v3) edge (i3);
\draw  (v2) edge (i1);
\draw  (v4) edge (i3);
\draw  (i1) edge (i2);
\draw  (i2) edge (i3);
\end{scope}

\begin{scope}[shift={(5,2.5)}]

\tikzstyle{every node} = [internal_nodes]

\node (i1) at (-2,0.5) {};
\node[pin={[red, inner sep =0.5pt, pin distance = 8pt, pin edge={red, thick, <-}]170:\footnotesize *}
] (i2) at (-1.5,0) {};
\node (i3) at (-2,-0.5) {};

\tikzstyle{every node} = [leaves]
\node (v1) at (-2.5,0) {$1$};
\node (v3) at (-1.5,-1) {$3$};
\node (v2) at (-2.5,-1) {$2$};
\node (v4) at (-1,-0.5) {$4$};
\tikzstyle{every path}=[path_styles]

\draw  (v1) edge (i1);
\draw  (v3) edge (i3);
\draw  (v2) edge (i3);
\draw  (v4) edge (i2);
\draw  (i1) edge (i2);
\draw  (i2) edge (i3);
\end{scope}

\begin{scope}[shift={(7.5,2.5)}]

\tikzstyle{every node} = [internal_nodes]

\node (i1) at (-2,0.5) {};
\node (i2) at (-1.5,0) {};
\node (i3) at (-1,-0.5) {};

\tikzstyle{every node} = [leaves]
\node (v1) at (-2.5,0) {$1$};
\node (v3) at (-1.5,-1) {$3$};
\node (v2) at (-2,-0.5) {$2$};
\node (v4) at (-0.5,-1) {$4$};
\tikzstyle{every path}=[path_styles]

\draw  (v1) edge (i1);
\draw  (v3) edge (i3);
\draw  (v2) edge (i2);
\draw  (v4) edge (i2);
\draw  (i1) edge (i2);
\draw  (i2) edge (i3);
\end{scope}

\end{tikzpicture}
    \caption{Trees in  $\widehat{\mathcal N_G}$ for $G$ in Example \ref{example:gamma_trees}.}
    \label{figure:example_modified_P4}
\end{figure}
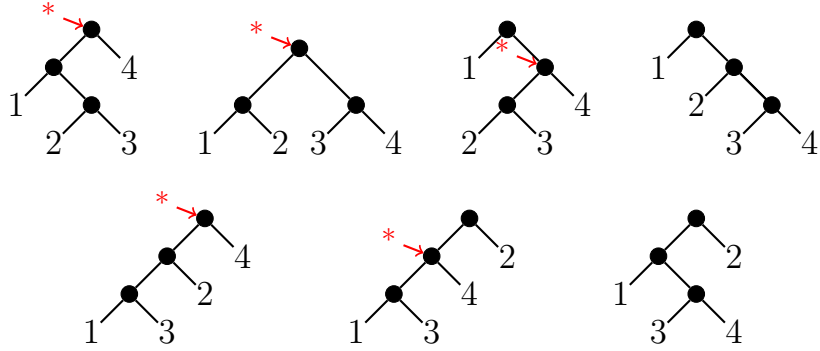

\begin{corollary} \label{corollary:gamma_positive} Let $G$ be a  chordal graph with $n$ vertices and $k$ connected components. 
Then $(-1)^{n-k} \mu_G(t)$ is $\gamma$-positive.

\end{corollary}

\begin{proof}This follows from Theorem~\ref{theorem:gamma_positive}, Proposition~\ref{proposition:product_components}, and the fact that  products of $\gamma$-positive polynomials are $\gamma$-positive.  
\end{proof}

 \begin{corollary} \label{corollary:gamma_difference} Let $G$ and $H$ be    perfectly labeled chordal graphs  with $n$ vertices and $k$ connected components.  If  $H$ is a subgraph of $G$  
then $(-1)^{n-k}(\mu_G({t}) -  \mu_H({t}))$ is $\gamma$-positive.  
\end{corollary}

\begin{proof} In the connected case, this result follows immediately from  Theorem~\ref{theorem:gamma_positive}, since clearly $\widehat {\mathcal N_H} \subseteq  \widehat {\mathcal N_G} $. In the general case we also use Proposition~\ref{proposition:product_components}.  Details are given  in the proof of the more general Corollary~\ref{corollary:e_positive_difference}.
\end{proof}

Recall that for a polynomial $p(t)$ with real coefficients, $p(t) \succ 0$ means that all coefficients, including the constant term, are positive.
\begin{corollary}\label{corollary:positivity} Let $G$  and $H$ be    perfectly labeled chordal graphs  with $n$ vertices and $k$ connected components. Then \begin{equation}\label{equation:postivity_of_mu}
     (-1)^{n-k} \mu_G(t)\succ 0.\end{equation}
 If $H$ is a subgraph of $G$ then
 \begin{equation}\label{equation:postivity_of_differences} (-1)^{n-k} (\mu_G({t}) -  \mu_H({t}))\succ 0. \end{equation} 
 \end{corollary}

\begin{proof} Note that (\ref{equation:postivity_of_mu}) follows from Corollary~\ref{corollary:gamma_positive} and  the  the fact that $\mu_G(0) \ne 0$, cf.  Proposition~\ref{corollary:palindromic}.

Note that (\ref{equation:postivity_of_differences}) follows from Corollary~\ref{corollary:gamma_difference} and  the  the fact that $\mu_G(0) -\mu_H(0) \ne 0$. Through a similar argument to the one used in proof of Proposition \ref{proposition:necessary_forest_condition}, one can prove this fact 
by using  Whitney's Broken Circuit Theorem  \cite{Whitney1932}. Indeed,  removing one edge that does not increase the number of connected components decreases the number of spanning forests that are free of broken circuits. 
\end{proof}

\begin{example} Let $G$ be the graph in Example~\ref{example:gamma_positive_theorem} and let $H=P_4$. We have
\begin{align*}-(\mu_G(t) -\mu_{H}(t)) &= 2(1+t)^3+5t(1+t) - ((1+t)^3 + 3t(1+t)) \\ 
&= (1+t)^3 + 2t(1+t) ,
\end{align*} 
which confirms Corollary~\ref{corollary:gamma_difference} for these graphs.
\end{example}

\begin{example}  \label{example:path} We apply Theorem~\ref{theorem:gamma_positive} to the path graphs.
A  binary tree $T$ on leaf set $[n]$ is in  $\mathcal N_{P_n}$ if and only if the leaves of $T$ are labeled from left to 
right in increasing order.  Consequently, the only Lyndon nodes of $T$ are the internal nodes whose left child is a 
leaf. Since there can be only one leaf labeling for the trees in $\mathcal N_{P_n}$, we can view these trees as unlabeled binary trees.  Thus by  Theorems~\ref{theorem:path_narayana} and~\ref{theorem:gamma_positive}, we get the following combinatorial interpretation of the Narayana polynomials
$$N_n(t) = \sum_{T } t^{m(T)} (1+t)^{n-1-2m(T)} ,$$
where the sum is taken over the unlabeled binary trees $T$ with $n$ leaves in which the left child of every left child is a leaf, and $m(T)$ is the number of internal nodes of $T$ that are left children.\end{example} 
 
 \begin{example}  \label{example:star} We apply Theorem~\ref{theorem:gamma_positive} to the star graphs $St_{n+1}$.  We will represent a binary tree $T$ as $T_L \land T_R$, where $T_L$ is the left subtree of $T$ and $T_R$ is the right subtree of $T$.
A  binary tree $T$ on leaf set $[n+1]$ is  in $ \mathcal N_{St_{n+1}}$ if and only if $T$  is of the form
 $$(\dots ((\sigma(1) \land \sigma(2) ) \land \sigma(3)) \land \cdots)\land \sigma(n+1),$$
 where $\sigma \in \mathfrak S_{n+1}$ and $\sigma(1) = 1$. (See Figure~\ref{figure:normal_star_graph}.)   Moreover, $x$ is a Lyndon node of $T$ if and only if 
 $R(x) = \sigma(2)$, or $R(x) = \sigma(i)$ where $3 \le i \le n+1$ and  $\sigma(i-1) > \sigma(i)$.
 \begin{figure}
    \centering
     \begin{tikzpicture}[thick,scale=0.6]

\tikzstyle{every node}=[draw,scale=0.5]

    \draw [circle,radius=20pt,color=black] (1,1)  node (i1){};
    \draw [circle,color=black] (2,2)  node (i2){};
    \draw [circle,color=black] (4,4)  node (i4){};
    \draw [circle,color=black] (3,3)  node (i3){};

\tikzstyle{every node}=[inner sep=1pt, minimum width=14pt,scale=0.7]

    \draw (0,0)  node (m){$\sigma(1)=1$};
    \draw (2,0)  node (l1){$\sigma(2)$};
    \draw (3,1)  node (l2){$\sigma(3)$};
    \draw (4,2)  node (l3){$\sigma(n)$};
    \draw (5,3)  node (l4){$\sigma(n+1)$};
    \draw (6,2) node (comma){\Large ,};

    \draw (m) --  (i1) ;
    \draw (i1) --  (l1) ;
    \draw (i2) --  (l2) ;
    \draw (i3) --  (l3) ;
    \draw (i4) --  (l4) ;
    \draw (i1) --  (i2) ;
    \draw [dashed, thick] (i2) --  (i3) ;
    \draw [dotted, thick] (2.6,1.6) --  (3.3,2.3) ;

    \draw (i3) --  (i4) ;

\end{tikzpicture}
    \caption{Trees in $\mathcal N_{St_{n+1}}$}
    \label{figure:normal_star_graph}
\end{figure}
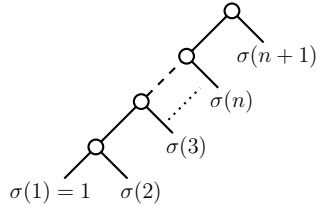
 
  Let $\sigma_T $ be 
 the permutation in $\mathfrak S_{n}$ defined by letting $\sigma_T(i)= \sigma(i+1)-1$ for all $i \in[n-1]$. It follows that $T \in \widehat{ \mathcal N_{St_{n+1}}}$ if and only if $\sigma_T$ has no consecutive ascents. Thus by Theorems~\ref{theorem:star_binom_euler} and~\ref{theorem:gamma_positive},
 \begin{equation} \label{equation:tildeA_gamma} \tilde A_{n}(t) =  \sum_{\sigma \in \widehat{\mathfrak S}_n} t^{\asc(\sigma)} (1+t)^{n-2\asc(\sigma)} ,\end{equation}
 where $\widehat{\mathfrak S}_n$ is the set of permutations in $\mathfrak S_n$ with no consecutive ascents and 
 $\asc(\sigma)$ is the number of ascents of $\sigma$.  This combinatorial description of the $\gamma$-coefficients of $\tilde A_{n}(t) $  was previously  obtained in \cite{ShareshianWachs2020}, along with a $q$-analog and a symmetric function analog. 
 
 Note that  the Postnikov-Reiner-Williams formula \eqref{equation:gamma_tilde} yields
 a different combinatorial description of the $\gamma$-coefficients.
 Equating the two descriptions yields $$ |\{ \sigma \in \widehat{\mathfrak S}_n | \asc(\sigma) = j \}| = $$ $$ |\{ \sigma \in \widehat{\mathfrak S}_{n+1} |  \sigma(n)>\sigma(n+1), \sigma(i ) > \sigma(i+1)  \forall i \in [\sigma^{-1}(1)-1],\asc(\sigma) = j \}|. $$   A nice bijective proof of this was obtained by Ellzey \cite{Ellzey2014}.
 \end{example}

We observe that (\ref{equation:postivity_of_mu}) (without the chordality assumption) is a consequence of Conjecture~\ref{conjecture:h_T},  whose proof, as was noted in Remark~\ref{remark:h_T}, will appear in \cite{AvilaCarrilloGonzalezDleon2026B}. Indeed this follows from the fact that  the $h$-polynomial of any convex simple polytope has positive coefficients. We think  it likely that  other results discussed above also hold 
without the  chordality assumption. 
Item (1) of the following conjecture has been tested by computer for connected graphs with up to $7$ vertices.
\begin{conjecture}Let $G$ be a   graph with $n$ vertices and $k$ connected components.  Then 
\begin{enumerate} \item 
$(-1)^{n-k} \mu_G(t) $  is $\gamma$-positive. 
\item If $H$ is a subgraph of $G$ with $n$ vertices and $k$ connected components  then 
$$(-1)^{n-k}(\mu_G({t}) -   \mu_H({t}))$$  is $\gamma$-positive.
\end{enumerate}
 \end{conjecture}

It is well known that all real-rooted palindromic polynomials with nonnegative real coefficients are $\gamma$-positive; see \cite{Branden2015}.  Since by Proposition~\ref{corollary:palindromic}, $\mu_G(t)$ is palindromic for all $G$, we conjecture the following.

\begin{conjecture}Let $G$ be a   graph with $n$ vertices and $k$ connected components.   Then
\begin{enumerate} \item 
$(-1)^{n-k} \mu_G(t) $  
is real-rooted. 
\item If $H$ is a subgraph of $G$ with $n$ vertices and $k$ connected components  then 
$$(-1)^{n-k}(\mu_G({t}) -   \mu_H({t}))$$  
is real-rooted.
\end{enumerate}
 \end{conjecture}

We have verified (1) by computer for all connected graphs with up to $7$ vertices.   It follows from (\ref{equation:complete_formula})  that (1) holds for the complete graph.  It is known that the Narayana polynomials are real-rooted (see \cite[Exercise 4.7]{Petersen2015}) and  more recently, Haglund and Zang \cite{HaglundZhang2019} established real-rootedness for the binomial-Eulerian polynomials. A proof of (1) in the case of the cycle graph will appear in \cite{AvilaCarrilloGonzalezDleon2026B}.
Hence by Theorems~\ref{theorem:path_narayana} and~\ref{theorem:star_binom_euler}, and Proposition~\ref{proposition:product_components}, item (1) is valid for all graphs whose connected components are  complete graphs, path graphs,   star graphs, or cycle graphs.

 Next we consider the polynomial $(-1)^{n-k}\mu_G(t) - h_{\mathcal P_G}(t)$ of  Conjecture~\ref{conjecture:h_T}.  

 Palindromicity of this polynomial follows from the fact that $\mu_G(t)$ and $h_{\mathcal P_G}(t)$ are both palindromic with the same center of symmetry. Note that this polynomial fails to be $\gamma$-positive already for $G=K_3$. Indeed, $\mu_{K_3}(t) - h_{\mathcal P_{K_3}}(t)=1+t+t^2$.  However, note that the weaker unimodality property  still  holds for $K_3$. In fact, as we see in the next result, it holds for all complete graphs.

 \begin{theorem} For all $n \ge 1$, the  palindromic polynomial
$(-1)^{n-1}\mu_{K_n}(t) - h_{\mathcal P_{K_n}}(t)$ is unimodal, has degree $n-1$, and has positive constant term.
     \end{theorem}

\begin{proof} 
   By (\ref{equation:complete_tree}), (\ref{equation:combinatorial_euler}) and (\ref{equation:h_complete}), we have $$(-1)^{n-k}\mu_{K_n}(t) - h_{\mathcal P_{K_n}}(t) = \sum_{T \in \mathcal {UT}_n} \, \sum_{w\in L(T)}  t^{\des(T,w)} - \sum_{\sigma \in \mathfrak S_n} t^{\des(\sigma)}, $$ where $\mathcal {UT}_n$ is the set of unlabeled nonplanar rooted trees with $n$ nodes and $L(T)$ is the set of  labelings of $T$.  Since permutations can be viewed as labeled linear nonplanar rooted trees, we have that 
   $$(-1)^{n-k}\mu_{K_n}(t) - h_{\mathcal P_{K_n}}(t) = \sum_{T \in \overline{\mathcal {UT}}_n} \, \sum_{w\in L(T)}  t^{\des(T,w)},$$
   where $\overline{\mathcal {UT}}_n$ is the set of {\it nonlinear} unlabeled nonplanar rooted trees with $n$ nodes.  Now for each $T\in {\mathcal {UT}}_n $, we claim that  $\sum_{w\in L(T)}  t^{\des(T,w)}$ is palindromic and unimodal.  Indeed, this follows from \cite[Lemma 2.2 and Theorem 2.3]{GradyPoznanovic2020}, which gives the claimed result for rooted {\it planar} trees.  Since the polynomial $\sum_{w\in L(T)}  t^{\des(T,w)}$ when $T$ is planar is a constant multiple of the polynomial when $T$ is nonplanar, we can conclude that our claim holds.  Since every tree has both an increasing and a decreasing labeling, the polynomial $\sum_{w\in L(T)}  t^{\des(T,w)}$ has strictly positive constant term and degree $n-1$ for all  $T\in {\mathcal {UT}}_n $.  The result now follows from the fact that the sum of palindromic unimodal polynomials with the same center of symmetry is palindromic and unimodal.
\end{proof}
     
 The  result for $K_n$ leads us to make the following conjecture, which is stronger than Conjecture~\ref{conjecture:h_T} and is still open.  We have checked this by computer for all graphs with up to $n=7$ vertices.  

   \begin{conjecture}  For any graph $G$  with $n$ vertices and $k$ connected components, the palindromic polynomial
 $(-1)^{n-k}\mu_G(t) - h_{\mathcal P_G}(t)$ is  unimodal.
 \end{conjecture}

 The following even stronger conjecture has also been checked by computer for all graphs with up to $n=7$ vertices.
  \begin{conjecture}  For a graph $G$  with $n$ vertices and $k$ connected components, the palindromic polynomial
 $(-1)^{n-k}\mu_G(t) - h_{\mathcal P_G}(t)$ is  log-concave, in the sense of  (\ref{equation:log_concave}).
 \end{conjecture}

\section{Multiweighted bond posets and M\"obius symmetric functions} \label{section:multiweight}

In \cite{GonzalezDLeon2016} Gonz\'alez D'Le\'on considers a generalization of the weighted partition poset based on the observation   that a weighted block $B^u$ of a weighted partition  can be written as $B^{(u,|B|-1-u)}$ thereby using weights that are  pairs of nonnegative integers rather than single nonnegative integers.  To merge two bi-weighted blocks, $B_1^{(\mu(1),\mu(2))} $ and $B_2^{(\nu(1),\nu(2))} $, now means to form either the bi-weighted block 
$(B_1 \cup B_2)^{(\mu(1) + \nu(1) +1, \mu(2) + \nu(2) )} $ or the biweighted block $(B_1 \cup B_2)^{(\mu(1) + \nu(1) , \mu(2) + \nu(2) +1)} $.    By using weights that are $r$-tuples instead of pairs, one gets the more general  $r$-weighted partition poset studied in \cite{GonzalezDLeon2016}. 
In this section we  generalize the notion of weighted bond posets  to $r$-weighted bond posets.  This leads to a symmetric function generalization of the M\"obius polynomial $\mu_G(t)$ and  symmetric function generalizations of the results for  $\mu_G(t)$ discussed in earlier sections.   

\subsection{Preliminaries}

For $r>0$, a \emph{weak $r$-composition}  of an integer $m\ge 0$ is an  $r$-tuple $\nu=(\nu(1),\mu(2),\dots,\nu(r))$ of nonnegative integers  whose sum $\sum_{i= 1}^r \nu(i)$ equals $ m$.  We  denote by $\wcomp_{m,r}$ the set of   weak $r$-compositions of $m$.  
  For $\mu = (\mu(1), \dots,\mu(r)) \in \wcomp_{m,r}$ and $\nu = (\nu(1),\dots,\nu(r)) \in \wcomp_{n,r}$, let $$\mu+\nu:=(\mu(1)+\nu(1), \dots,\mu(r)+\nu(r)) \in\wcomp_{m+n,r}.$$
  We say $\mu \ge \nu$ if $$\mu-\nu:=(\mu(1)-\nu(1), \dots,\mu(r)-\nu(r)) \in\wcomp_{m-n,r}.$$
Let $\bf 0$ denote the $r$-tuple of $0$'s and for each $i=1,\dots,r$, let ${\bf e_i}$ denote the $r$-tuple with a $1$ in the $i$th component and $0$'s elsewhere.

 By padding with $0$'s we can view  weak $r$-compositions of $m$ as  $s$-compositions of $m$  for $s >r$.
  Thus we can say
  $$\wcomp_{m,1} \subset \wcomp_{m,2} \subset \wcomp_{m,3} \subset \cdots,$$
 and define
 $$ \wcomp_{m} := \wcomp_{m,\infty}:= \bigcup_{r \ge 1} \wcomp_{m,r}.$$
We can view $\wcomp_{m}$  as the set of   infinite sequences of nonnegative integers whose sum is $m$.  We call these infinite sequences {\em weak $\infty$-compositions} of $m$ or just {\em weak compositions} of $m$.  

For $r\in \ZZ_{>0} \cup \{\infty\}$, an {\em $r$-weighted partition} of a set $S$ is defined to be a set  of the form 
$\bpi =\{B_1^{\nu_1},B_2^{\nu_2},...,B_k^{\nu_k}\}$,
where $\{B_1,B_2,...,B_k\}$ is a partition of $S$ and $\nu_i \in \wcomp_{|B_i|-1,r}$ for all $i\in [k]$. In what follows, it will be convenient to let $[\infty]$ denote $\ZZ_{>0}$.

The \emph{poset of $r$-weighted partitions} $\Pi_S^r$ is the set of $r$-weighted partitions of $S$ with order relation defined by
$$\{A_1^{\mu_1},A_2^{\mu_2},...,A_j^{\mu_j}\}\le\{B_1^{\nu_1}, B_2^{\nu_2},...,B_k^{\nu_k}\}$$ if the following
conditions hold:
\begin{itemize}
 \item $\{A_1,A_2,...,A_j\} \le \{B_1,B_2,...,B_k\}$ in $\Pi_S$
 \item if $B_m=A_{i_1}\cup A_{i_2}\cup ... \cup A_{i_l} $ then 
 $$\nu_m \ge \mu_{i_1} + \mu_{i_2} + ... + \mu_{i_l}.$$
\end{itemize}

The \ covering relation $\lessdot$ is given by
$$\{A_1^{\mu_1},A_2^{\mu_2},...,A_j^{\mu_j}\} \lessdot \{B_1^{\nu_1}, B_2^{\nu_2},...,B_k^{\nu_k}\}$$ if and only if the
following conditions hold:
\begin{itemize}
 \item $\{A_1,A_2,\dots,A_j\} \lessdot \{B_1,B_2,\dots,B_k\}$ in $\Pi_S$
 \item if $B_m=A_{h}\cup A_{i}$, where $h \ne i$, then $\nu_m-(\mu_{h} + \mu_{i})= \mathbf{e_s} $ for some $s \in [r]$.
 \item if $B_m = A_i$ then $\nu_m = \mu_i$.
 \end{itemize}

 Note that  when $r=1$, the $r$-weighted partition poset $\Pi_S^r$ is isomorphic to  $\Pi_S$ and when $r=2$, the $r$-weighted partition poset $\Pi_S^r$ is isomorphic to the weighted partition poset $\mathcal W \Pi_S$.  Indeed,  the map taking the biweighted partition  $B_1^{\nu_1}|\cdots |B_{k}^{\nu_{k}}$   to the single weighted partition $B_1^{\nu_1(1)}|\cdots |B_{k}^{\nu_{k}(1)}$ is an isomorphism from $\Pi_S^2$ to  $\mathcal W\Pi_S$.  
 
 Let $\Pi^r_n:=\Pi_{[n]}^r$.  See Figure \ref{figure:example_weighted_partition_poset_n3r3} for the Hasse diagram of $\Pi_3^3$. 

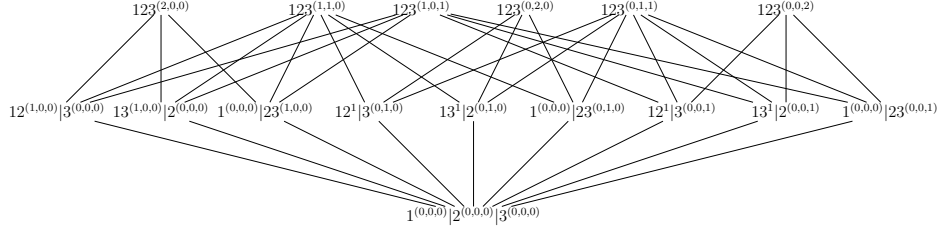
\begin{figure}
    \centering
     \resizebox{\columnwidth}{!}{
    \begin{tikzpicture}[line join=bevel,scale=0.8]

\tikzstyle{every node}=[inner sep=0pt, scale=0.6, minimum width=4pt]
 \node (n102030) at (0,0)  {$1^{(0,0,0)}| 2^{(0,0,0)}| 3^{(0,0,0)}$};

  \node (n12i30) at (-8,2) {$12^{(1,0,0)}| 3^{(0,0,0)}$};
  \node (n13i20) at (-6,2) {$13^{ (1,0,0)}| 2^{ (0,0,0)}$};
  \node (n1023i) at (-4,2)  {$1^{(0,0,0)}| 23^{(1,0,0)}$};

 \node (n12j30) at (-2,2)  {$12^{ 1}| 3^{(0,1,0)}$};
 \node (n13j20) at (0,2)  {$13^{1}| 2^{(0,1,0)}$};
  \node (n1023j) at (2,2) {$1^{(0,0,0)}| 23^{(0,1,0)}$};

\node (n12k30) at (4,2)  {$12^{ 1}| 3^{(0,0,1)}$};
 \node (n13k20) at (6,2)  {$13^{1}| 2^{(0,0,1)}$};
  \node (n1023k) at (8,2) {$1^{(0,0,0)}| 23^{(0,0,1)}$};

 \node (n123ii) at (-6,4){$123^ {(2,0,0)}$};
 \node (n123ij) at (-3,4){$123^ {(1,1,0)}$};
 \node (n123ik) at (-1,4){$123^ {(1,0,1)}$};
 \node (n123jj) at (1,4){$123^ {(0,2,0)}$};
 \node (n123jk) at (3,4){$123^ {(0,1,1)}$};
 \node (n123kk) at (6,4){$123^ {(0,0,2)}$};

 \draw (n123ii) -- (n1023i) ;
 \draw (n123ii)-- (n13i20);
  \draw (n123ii) -- (n12i30);  

\draw (n123jj) -- (n1023j) ;
 \draw (n123jj)-- (n13j20);
  \draw (n123jj) -- (n12j30);  

\draw (n123kk) -- (n1023k) ;
 \draw (n123kk)-- (n13k20);
  \draw (n123kk) -- (n12k30); 

 \draw (n123ik) -- (n1023i) ;
 \draw (n123ik)-- (n13i20);
  \draw (n123ik) -- (n12i30);  
\draw (n123ik) -- (n1023k) ;
 \draw (n123ik)-- (n13k20);
  \draw (n123ik) -- (n12k30);

 \draw (n123ij) -- (n1023i) ;
 \draw (n123ij)-- (n13i20);
  \draw (n123ij) -- (n12i30);  
\draw (n123ij) -- (n1023j) ;
 \draw (n123ij)-- (n13j20);
  \draw (n123ij) -- (n12j30);
	
 \draw (n123jk) -- (n1023j) ;
 \draw (n123jk)-- (n13j20);
  \draw (n123jk) -- (n12j30);  
\draw (n123jk) -- (n1023k) ;
 \draw (n123jk)-- (n13k20);
  \draw (n123jk) -- (n12k30);

  \draw [] (n13i20) -- (n102030);
  \draw [] (n12i30)-- (n102030);
  \draw [] (n1023i)--(n102030);
 \draw [] (n1023j)  --  (n102030);
  \draw [] (n12j30) --  (n102030);
  \draw [] (n13j20) --(n102030);
 \draw [] (n1023k)  --  (n102030);
  \draw [] (n12k30) --  (n102030);
  \draw [] (n13k20) --(n102030);
\end{tikzpicture}
    }
    \caption{The poset of $3$-weighted partitions $\Pi_3^3$}
    \label{figure:example_weighted_partition_poset_n3r3}
\end{figure}

\begin{remark}
In \cite{GonzalezDLeon2016} Gonz\'alez D'Le\'on obtains results connecting the $r$-weighted partition poset $\Pi_n^r$ to the free Lie algebra with $r$ compatible brackets.   These results generalize classical results in the $r=1$ case, and results of both authors \cite{GonzalezDLeonWachs2016} in the $r=2$ case.  
\end{remark}

 Given a graph $G=(V,E)$ and $r \in \ZZ_{> 0} \cup \{ \infty \}$, we define the {\em $r$-weighted bond poset} $\Pi_G^r$ of $G$  to be the induced subposet of $\Pi_{V}^r$ consisting of $r$-weighted partitions  $\{B_1^{\nu_1},\dots,B_k^{\nu_k}\}$ whose underlying partition $\{B_1,\dots,B_k\}$ is  in $\Pi_G$, that is, for each $i$, the induced subgraph $G|_{B_i}$ is connected.  Note that $\Pi_G^r$ is finite if and only if $r$ is finite or  $E=\emptyset$. 
 
  Clearly, for all $r$,  
  $$\Pi_{K_n}^r = \Pi_n^r.$$  Note that for all graphs $G$,
$$\Pi_G^1 \simeq \Pi_G \,\,\mbox{ and } \,\, \Pi_G^2 \simeq \mathcal W\Pi_G.$$   See Figure~\ref{figure:weighted_bond_poset_P3r3} for the Hasse diagram of $\Pi_{P_3}^3$. 

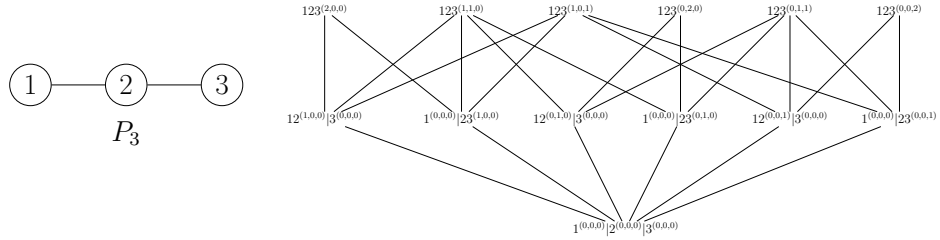
\begin{figure}
    \centering
     \resizebox{\columnwidth}{!}{
    \begin{tikzpicture}[line join=bevel,scale=1]
\begin{scope}[scale=0.7,yshift=2cm]
        \node  at (-4,0) { $P_3$};
\tikzstyle{every node}=[draw,circle,scale=0.7]
\node (n0) at (-6,1) {\Large $1$};
\node (n1) at (-4,1) {\Large $2$};
\node (n2) at (-2,1) {\Large $3$};
\draw (n0) -- (n1) -- (n2);
\end{scope}
\begin{scope}[xshift=4.5cm,scale=0.8]

\tikzstyle{every node}=[inner sep=0pt, scale=0.45, minimum width=4pt]
 \node (n102030) at (0,0)  {$1^{(0,0,0)}| 2^{(0,0,0)}| 3^{(0,0,0)}$};

  \node (n12i30) at (-5.5,2) {$12^{(1,0,0)}| 3^{(0,0,0)}$};
  
  \node (n1023i) at (-3,2) {$1^{(0,0,0)}| 23^{(1,0,0)}$};

 \node (n12j30) at (-1,2) {$12^{ (0,1,0)}| 3^{(0,0,0)}$};
 
  \node (n1023j) at (1,2) {$1^{(0,0,0)}| 23^{(0,1,0)}$};

\node (n12k30) at (3,2) {$12^{ (0,0,1)}| 3^{(0,0,0)}$};
 
  \node (n1023k) at (5,2) {$1^{(0,0,0)}| 23^{(0,0,1)}$};

 \node (n123ii) at (-5.5,4) {$123^ {(2,0,0)}$};
 \node (n123ij) at (-3,4) {$123^ {(1,1,0)}$};
 \node (n123ik) at (-1,4){$123^ {(1,0,1)}$};
 \node (n123jj) at (1,4) {$123^ {(0,2,0)}$};
 \node (n123jk) at (3,4) {$123^ {(0,1,1)}$};
 \node (n123kk) at (5,4) {$123^ {(0,0,2)}$};

 \draw (n123ii) -- (n1023i) ;

  \draw (n123ii) -- (n12i30);  

\draw (n123jj) -- (n1023j) ;

  \draw (n123jj) -- (n12j30);  

\draw (n123kk) -- (n1023k) ;

  \draw (n123kk) -- (n12k30); 

 \draw (n123ik) -- (n1023i) ;

  \draw (n123ik) -- (n12i30);  
\draw (n123ik) -- (n1023k) ;

  \draw (n123ik) -- (n12k30);

 \draw (n123ij) -- (n1023i) ;

  \draw (n123ij) -- (n12i30);  
\draw (n123ij) -- (n1023j) ;

  \draw (n123ij) -- (n12j30);
	
 \draw (n123jk) -- (n1023j) ;

  \draw (n123jk) -- (n12j30);  
\draw (n123jk) -- (n1023k) ;

  \draw (n123jk) -- (n12k30);

  \draw [] (n12i30)-- (n102030);
  \draw [] (n1023i)--(n102030);
 \draw [] (n1023j)  --  (n102030);
  \draw [] (n12j30) --  (n102030);
 \draw [] (n1023k)  --  (n102030);
  \draw [] (n12k30) --  (n102030);
\end{scope}
\end{tikzpicture}
    }
    \caption{The 3-weighted bond poset $\Pi_{P_3}^3$}
    \label{figure:weighted_bond_poset_P3r3}
\end{figure}

 Just as for the $r=2$ case discused in Section~\ref{section:definitions}, for $ \balpha \le \bbeta$ in $ \Pi^r_G$,   let $[\balpha , \bbeta]_G$ denote the closed interval $\{ \bpi \in  \Pi_G^r : \balpha \le \bpi \le \bbeta\}$.  When $G=(V,E)$ is connected, again we drop the set brackets for the maximal elements $\{V^\nu\}$ of $ \Pi_G^r$.   
 
 For each graph $G=(V,E)$ with  $k$ connected components, and $r \in \ZZ_{\ge 1} \cup \{ \infty \} $, the poset $\Pi_G^r$ is pure of length $|V|-k$.  It has minimum element $\hat 0 := v_1^{\mathbf{0}}|\cdots | v_n^{\mathbf{0}}$, where $V=\{v_1,\dots,v_n\}$, and if $G$ is connected, $\Pi_G^r$ has maximal element $\{V^{\nu}\}$ (which we write as $V^{\nu}$) for each $\nu \in \wcomp_{n-1,r}$.  
 
  Proposition~\ref{prop:product} generalizes to  $r$-weighted bond posets $\Pi_G^r$.
 
 \begin{proposition} \label{proposition:multi_isomorphism_product} Let $G$ be a graph whose connected components are  $G_1,\dots, G_k$.
Then for all $r \in \ZZ_{\ge 1} \cup \{ \infty \} $, the map $$\varphi:  \Pi^r_{G_1} \times \cdots \times  \Pi^r_{G_k} \to \Pi^r_G $$  defined by $\varphi(\bpi_1,\dots, \bpi_k)= \bpi_1 \cup  \cdots \cup \bpi_k$ is a poset isomorphism.
\end{proposition}

 The following result  is straightforward to prove.
\begin{proposition} \label{proposition:isomorphic_intervals} Let $G$ be a  graph on vertex set $V$ and let $r,s \in \ZZ_{>0} \cup \{\infty\}$.
\begin{enumerate} 
\item If $r <s$   then each maximal interval of $\Pi_G^r$ is isomorphic to a maximal interval of  $\Pi_G^s$
\item Suppose  $\alpha \in \wcomp_{n-1,r} $ and  $\beta \in \wcomp_{n-1,s} $ have the same multiset of  nonzero entries.  If $G$ is connected   then the maximal  intervals $[\hat 0,V^\alpha]_G$ and $[\hat 0,V^\beta]_G$ are isomorphic.  
\end{enumerate} \end{proposition}

 \subsection{M\"obius symmetric functions}

We are now ready to define the   graph invariant that generalizes the M\"obius polynomial defined in Section~\ref{section:definitions}.

For any graph $G$,  the {\em M\"obius symmetric function} of $G$ in  infinitely many variables $\xx = (x_1,x_2,\dots)$ is defined by
$$M_{G}(\xx) := \sum_{\bpi \in \Max(\Pi_G^\infty)} \mu_{\Pi_G^\infty} (\hat 0, \bpi) \xx^{w(\bpi)} ,$$
where, $\xx^{\nu}= x_1^{\nu(1)} x_2^{\nu(2)}\cdots$ for $\nu \in \wcomp_{m,\infty}$, and
$w(\bpi)$ is the sum of the weights of the blocks of $\bpi$.  (Recall $\Max(P)$ is the set of maximal elements of a poset $P$.) It follows from  Proposition~\ref{proposition:isomorphic_intervals}  that $M_{G}(\xx)$ is indeed a symmetric function. Moreover, if $G$ has $n$ vertices and $k$ connected components then  $M_{G}(\xx)$ is a homogeneous symmetric function of degree $n-k$. This is in fact the highest degree homogeneous component of the symmetric function defined in 
Equation~(\ref{equation:new_chromatic_symmetric_definition}) which, as we observe in Section \ref{section:new_chromatic_symmetric_function}, is a symmetric function analog of the chromatic polynomial of $G$.

Note that if $G$ is a connected graph on vertex set $V$ then  $$M_{G}(\xx) = \sum_{\nu \in \wcomp_{|V|-1}}\mu_{\Pi_{G}^\infty}(\hat 0, V^\nu) \xx^{\nu}.$$
Moreover, it follows from Proposition~\ref{proposition:isomorphic_intervals} that by viewing partitions  of $m$ as weak compositions of $m$ (by padding with zeros), we have for connected $G$,  $$ M_G(\xx) = \sum_{\lambda \vdash |V|-1} \mu_{\Pi_{G}^{\infty}} (\hat 0, V^\lambda) m_\lambda(\xx), $$
where $m_\lambda$ is the monomial symmetric function indexed by the partition~$\lambda$.

It also follows  from Proposition~\ref{proposition:isomorphic_intervals}
 that 
 for all $r$,
$$ M_{G}(x_1,\dots, x_r, 0,0,\dots) = \sum_{\bpi \in \Max(\Pi_G^r)} \mu_{\Pi_G^r} (\hat 0, \bpi) x_1^{w(\bpi)(1)} \cdots x_r^{w(\bpi)(r)}.$$
For each $r$, define the {\em M\"obius symmetric polynomial} $M_{G}(x_1,\dots, x_r)$  in $r$ variables to be $ M_{G}(x_1,\dots, x_r, 0,0,\dots)$.

Note that $M_{G}(1)$ is the M\"obius invariant $\mu(\Pi_G)$ of the bond lattice $\Pi_G$ and $M_{G}(1,t)$ is the M\"obius polynomial $\mu_G(t)$.  From Figure~\ref{figure:example_weighted_partition_poset_n3r3},  one can see that 
$$M_{K_3}(x_1,x_2,x_3)=2 (x_1^2+x_2^2+x_3^2)+ 5(x_1x_2+x_1x_3+ x_2x_3)$$ and from
Figure~\ref{figure:weighted_bond_poset_P3r3}, one can see that 
$$M_{P_3}(x_1,x_2,x_3)=(x_1^2+x_2^2+x_3^2)+ 3(x_1x_2+x_1x_3+ x_2x_3).$$ 

The isomorphism given in Proposition~\ref{proposition:multi_isomorphism_product}  yields the following generalization of Proposition~\ref{proposition:product_components}.
\begin{proposition} \label{proposition:M_product} Let $G$ be a graph whose connected components are $G_1,\dots,G_k$.  Then 
$$M_{G}(\xx) = \prod_{i=1}^k  M_{G_i} (\xx) .$$
\end{proposition}

\subsection{Recurrence relations for M\"obius symmetric functions}

The recursive structure of the weighted partition poset  is described in the following  propositions, whose straightforward proofs are left to the reader.  Recall the definitions of $G|B$ and $G/\pi$ given in Section~\ref{section:recursive_structure}. Equation~(\ref{equation:consequence_mobius}) below is a consequence of the multiplicative property of the  M\"obius function (see \cite[Proposition~3.8.2]{Stanley2012}).

\begin{proposition} \label{proposition:multiMobius_prod} Let $G$ be a  graph and let $r \in \ZZ_{>0}\cup \{\infty\}$.  If $$\bpi = \{B_1^{\nu_1},\dots, B_k^{\nu_k}\} \in \Pi_G^r,$$  then
$$[\hat 0, \bpi]_G \cong [\hat 0, B_1^{\nu_1}]_{G|_{B_1}} \times \cdots \times  [\hat 0, B_k^{\nu_k}]_{G|_{B_k}} .$$
Consequently,
\begin{equation} \label{equation:consequence_mobius} \mu_{\Pi_G^r}(\hat 0, \bpi) = \prod_{i=1}^k \mu_{\Pi^r_{G|_{B_i}}}(\hat 0, B_i^{\nu_i}) .\end{equation}
 \end{proposition}

For $\bpi \in \Pi^r_G$, let $U^r_G(\bpi)$ be the upper order ideal (order filter) generated by $\bpi$. 
For connected $G$ on vertex set $V$ and $\nu \in \wcomp_{|V|-1,r}$, the maximal elements of $U^r_G(\bpi)$
are of the form $V^\nu$, where 
\begin{equation} \label{equation:multibounds} \nu-w(\bpi) \in \wcomp_{|\bpi|-1,r}. \end{equation}

\begin{proposition} \label{proposition:multi_upper_order} Let $G$ be a connnected graph on vertex set $V$, let $r \in \ZZ_{>0} \cup \{\infty\}$ and let $\bpi = \{B_1^{\nu_1},\dots, B_k^{\nu_k}\} \in \Pi_G^r$. Then
 $$U_G(\bpi) \cong \Pi^r_{G/{\{B_1,\dots,B_k\}}}.$$
Moreover, for all $\nu \in \wcomp_{|V|-1,r}$ such that  $\nu-w(\bpi) \in \wcomp_{k-1,r}$, we have
 $$[\bpi,V^\nu]_G\cong [\hat 0, \{\{B_1,\dots,B_k\}^{\nu-w(\bpi)}\} ]_{G/{\{B_1,\dots,B_k\}}}.$$
 Consequently,
 $$\mu_{\Pi_{G}^r}( \bpi,V^\nu) = \mu_{\Pi_{G/\{B_1,\dots,B_k\}}^r }(\hat 0,  \{\{B_1,\dots,B_k\}^{\nu-w(\bpi)} \}) .$$
 \end{proposition}
 
These recursive properties yield the following recurrence relations for the M\"obius symmetric functions, which reduce to those for the M\"obius polynomials given in Theorem~\ref{theorem:recurrence}.  Indeed,  $M_H(1,t) = \mu_H(t)$ for all graphs $H$ and $h_m(1,t) =  [m+1]_t$ for all $m \ge 0$.

\begin{theorem} \label{theorem:multi_recurrence} Let  $G$ be a connected graph with more than one vertex.  Then
\begin{align}\label{equation:recursion_symmetric_1}
   M_G(\xx) = -\sum_{\pi \in \Pi_{G} \setminus \{\hat 1\}} h_{|\pi|-1} (\xx)\prod_{B\in \pi} M_{G|_B}(\xx),
\end{align}
and
\begin{align}\label{equation:recursion_symmetric_2}
   M_{G}(\xx) = - \sum_{\pi \in \Pi_{G}\setminus \{\hat{0}\}} M_{G/\pi}(\xx)\prod_{B\in \pi} h_{|B|-1}(\xx),
\end{align}
where $h_n(\xx)$ is the complete homogenous symmetric function of degree~$n$.
\end{theorem}

\begin{proof} Let $V$ be the vertex set of $G$.

By the recursive definition of the M\"obius function, for $\nu \in \wcomp_{|V|-1}$,
$$\sum_{\bpi \in [\hat{0}, V^{\nu}]_G} \mu_{\Pi_G^\infty}(\hat{0},\bpi)=0.$$
Therefore by (\ref{equation:multibounds}),

\begin{align*}
 0 
 &=   \sum_{\nu \in \wcomp_{|V|-1}} \xx^\nu \sum_{\bpi \in [\hat{0}, V^{\nu}]_G} \mu_{\Pi_G^\infty}(\hat{0},\bpi)\\
   &=  \sum_{\bpi\in \Pi_G^{\infty}} \mu_{\Pi_G^\infty}(\hat{0},\bpi)\xx^{w(\bpi)} \sum_{\eta \in \wcomp_{|\bpi|-1}}\xx^{\eta} \\
 &=  \sum_{\bpi\in \Pi_G^{\infty}} \mu_{\Pi_G^\infty}(\hat{0},\bpi)\xx^{w(\bpi)}\, h_{|\bpi|-1}(\xx).  
\end{align*}

Now by  \eqref{equation:consequence_mobius},
\begin{align*} 
0 &=\sum_{k=1}^n h_{k-1}(\xx) \sum_{\{B_1,\dots,B_k\} \in \Pi_G}\,\, \sum_{\substack{\nu_1,\nu_2, \dots, \nu_k\\ \nu_i \in \wcomp_{ |B_i|-1}}}\,\, \prod_{i=1}^{k} \mu_{ \Pi^\infty_{G|_{B_i} }} (\hat 0, B_i^{\nu_i}) \xx^{\nu_i} 
\\&= \sum_{k=1}^n h_{k-1}(\xx) \sum_{\{B_1,\dots,B_k\} \in \Pi_G}\,\, \prod_{i=1}^k \sum_{\nu \in \wcomp_{|B_i|-1}} \mu_{\Pi^\infty_{G|_{B_i} }} (\hat 0, B_i^{\nu}) \xx^{\nu}
\\ &= \sum_{k=1}^n h_{k-1}(\xx)\sum_{\{B_1,\dots,B_k\} \in \Pi_G}\,\,  \prod_{i=1}^k M_{G|_{B_i}} (\xx),
\end{align*}
which proves (\ref{equation:recursion_symmetric_1}).

By the dual version of the recursive definition of the M\"obius function, and (\ref{equation:multibounds}),
\begin{align*}
 0 
 &= \sum_{\nu \in \wcomp_{|V|-1}} \xx^\nu \sum_{\bpi \in [\hat{0}, V^{\nu}]_G} \mu_{\Pi_G^\infty}(\bpi, V^\nu) \\
 &= \sum_{\bpi \in \Pi_G^{\infty}} \xx^{w(\bpi)} \sum_{\nu \in \wcomp_{|\bpi|-1}} \mu_{\Pi_G^\infty}(\bpi, V^{\nu+w(\bpi)} ) \xx^{\nu}.
\end{align*}

Now by Proposition~\ref{proposition:multi_upper_order},
\begin{align*}
0 
&= \sum_{k=1}^n \sum_{\pi =\{B_1,\dots,B_k\} \in \Pi_G}\,\,  \sum_{\substack{\nu_1,\nu_2, \dots, \nu_k\\ \nu_i \in \wcomp_{|B_i|-1}}} \xx^{\sum_{i=1}^k \nu_i} \sum_{\nu \in \wcomp_{k-1}}  \mu_{\Pi^\infty_{G/\pi}}  (\hat 0, V(G/\pi)^\nu ) \xx^{\nu}
\\& = \sum_{\pi \in \Pi_G} \prod_{B\in \pi} \sum_{\nu \in \wcomp_{|B|-1}} \!\!\!\!\! \!\!\xx^\nu  \,\,\, M_{G/\pi} (\xx), 
\end{align*}
which proves \eqref{equation:recursion_symmetric_2}.\end{proof}

\subsection{Examples} In \cite{GonzalezDLeon2016},  Gonz\'alez D'Le\'on uses the complete graph case of the recurrence relations of Theorem~\ref{theorem:multi_recurrence}  to obtain
the exponential generating function formula 
$$\sum_{n\ge 1} M_{K_n}(\xx) \frac{y^n} {n!} = \left(\sum_{n\ge 1} h_{n-1}(\xx)\dfrac{y^n}{n!}\right)^{\langle -1 \rangle},$$
 which is a symmetric function generalization of (\ref{equation:K_n_generating}).  In \cite{GonzalezDLeon2016} it is  shown  that the coefficient of $\xx^\nu$ in $(-1)^{n-1} M_{K_{n}}(\xx)$ is the dimension of a subspace of the multilineal component, associated with $\nu$, of the free Lie algebra with multiple compatible brackets.

The recurrence relations of Theorem~\ref{theorem:multi_recurrence} in the case of the star graph reduce to 
$$\sum_{k=0}^{n-1} \binom {n-1} k h_k(\xx)\,M_{St_{n-k}}(\xx) = \delta_{n,1}.$$  This was used by Gonz\'alez D'Le\'on \cite{GonzalezDLeon2018} to 
 obtain
the exponential generating function formula 
\begin{equation} \label{equation:symmetric_star_exponential}  \sum_{n\ge0}M_{St_{n+1}}(\xx)\frac{y^n}{n!}  = \left ( \sum_{n \ge 0} h_n(\xx) \frac {y^n}{n!} \right )^{-1} ,\end{equation}
which is a symmetric function generalization of \eqref{equation:star_exponential_2}.   It is shown  in \cite{GonzalezDLeon2018}  that the coefficient of $\xx^\nu$ in $(-1)^{n-1} M_{St_{n}}(\xx)$ is the dimension of a subspace of the multilinear component, associated with $\nu$, of the colored exterior algebra.

\begin{remark}\label{remark:guise} The $r$-weighted bond poset $\Pi_{St_n}^r$  appeared in  the guise of the $r$-weighted Boolean algebra in the work of Gonz\'alez D'Le\'on \cite{GonzalezDLeon2018} on colored exterior algebras.    
For each $\bpi \in \Pi^r_{St_n}$, let $B^\nu$ be the  weighted block of $\bpi$ that contains $1$. Now remove $1$ from $B$  and reduce the remaining   elements by $1$ to get the set $\bar B$.  The map taking $\bpi$ to $\bar B^\nu$ defines an isomorphism from  $\Pi^r_{St_n}$ to the $r$-weighted Boolean algebra consisting of weighted sets $S^\nu$, where $S\subseteq [n-1]$ and $\nu \in \wcomp_{|S|,r}$.  Note that in the $r=2$ case, this $r$-weighted Boolean algebra is a variant of the weighted Boolean algebra $\mathcal {WB}_n$ mentioned in the Introduction.  
\end{remark}
 
Next we turn to the path graph $P_n$.  By using the recurrence relations of Theorem~\ref{theorem:multi_recurrence}, we will show that the  M\"obius symmetric functions for the path graphs are related to  a class of symmetric functions known as parking function symmetric functions.
 
 A \emph{parking function} on $[n]$ is a sequence $(p_1,\dots,p_n) \in [n]^n$ whose  weakly increasing rearrangement $p_{\sigma(1)}\le\cdots \le p_{\sigma(n)}$ satisfies  $p_{\sigma(i)}\le i$ for all $i \in [n]$; 
 see~\cite[Exercise 5.49]{Stanley1999}.  Let $\PF_n(\xx)$ denote the Frobenius characteristic of the permutation representation associated with the action of  $\sym_n$ on the set of parking functions on $[n]$, where a permutation acts by permuting the entries of the parking function.  The symmetric function $\PF_n(\xx)$ is known as the \emph{parking function symmetric function} and was first considered in this context by Haiman \cite{Haiman1994}.  
 In \cite{Stanley1997} Stanley observed that 
 \begin{equation} \label{equation:parking} PF_n(\xx) = \mbox{ coeff.  of   $y^n$ in } \frac 1 {n+1} \left (\sum_{j\ge 0} h_j(\xx) y^j \right )^{n+1}.\end{equation}

Expansions of $\omega PF_n(\xx)$ in the monomial symmetric function basis $\{m_\lambda: \lambda \vdash n\}$ and in the elementary symmetric function basis $\{e_\lambda: \lambda \vdash n\}$ can be found   in \cite{Stanley1997} as well, where  $\omega$ is the involution on the ring of symmetric functions that takes  $h_n$ to  $e_n$.
 Given $\pi \in \Pi_n$, let $\lambda(\pi)$ be the  partition of $n$ given by 
\begin{equation} \label{equation:set_to_number} \lambda(\pi)= (|B_1|,|B_2|,\dots, |B_k|),\end{equation}
where $B_1,\dots, B_k$ are the blocks of $\pi$ listed in weakly decreasing order of size.
For all $n \ge 1$,
\begin{equation} \label{equation:parking2} \omega PF_n(\xx) = \frac{1}{n+1} \sum_{\lambda \vdash n} \prod_{i} \binom {n+1}{\lambda(i) } m_\lambda(\xx) \end{equation}
and
\begin{equation} \label{equation:parking3} \omega PF_n(\xx) = \sum_{\pi \in \mathcal{NC}_n} e_{\lambda(\pi)}(\xx),\end{equation}
where $\mathcal{NC}_n$ is the set of noncrossing partitions of $[n]$.

\begin{theorem} \label{theorem:parking} For all $n \ge 1$,
$$(-1)^{n-1} M_{P_n}(\xx) = \omega PF_{n-1} (\xx) .$$
\end{theorem}

\begin{proof}  The recurrence relations of Theorem~\ref{theorem:multi_recurrence} in the case of $G=P_n$ reduce to 
$$
 \sum_{k=1}^n  h_{k-1}(\xx) \sum_{(n_1,\dots, n_k) \vDash n   }   \prod_{i=1}^{k} M_{P_{n_i}}(\xx) = \delta_{n,1}$$ and
$$ \sum_{k=1}^{n}  M_{P_k}(\xx) \sum_{ (n_1,\dots, n_k) \vDash n   }   \prod_{i=1}^{k} h_{n_i-1}(\xx)  = \delta_{n,1}.$$
Both  equations are equivalent to the generating function formula
$$ \sum_{n\ge1}(-1)^{n-1} M_{P_n}(\xx)y^n = \left(\sum_{n\ge1}(-1)^{n-1}h_{n-1}(\xx) \, y^n\right )
^{\left\langle -1 \right\rangle}  = \left(\frac{y}{\prod_{i\ge 1} (1+x_i y)}\right)^{\left\langle -1 \right\rangle},$$
which is a symmetric function analog of (\ref{equation:path_generating}).    Following the proof of \cite[Proposition~2.2 (b)]{Stanley1997}, by the Lagrange inversion formula  (see \cite[Theorem 5.4.2]{Stanley1999}), we thus have that $ M_{P_n}(\xx)$ is equal to the coefficient of $y^{n-1}$ in $$\frac 1 n \left( \prod_{i\ge 1} (1+x_i y)\right)^n = \frac 1 n \left (\sum_{j\ge 0} e_j(\xx) y^j\right)^n.  $$  By (\ref{equation:parking}), this coefficient is precisely $\omega PF_{n-1}(\xx)$, which yields the desired result.  
\end{proof}

We  have the following consequence of Theorem~\ref{theorem:parking} and Equation~(\ref{equation:parking2}).
 \begin{corollary} \label{corollary:parking} For all $r, n \ge 1$,  $$(-1)^{n-1} M_{P_n}(x_1,\dots,x_r) =  \sum_{\nu \in \wcomp_{n-1,r}} \frac{1}{n} \prod_{i=1}^r \binom {n}{ \nu(i)} x_1^{\nu(1)} \cdots x_r^{\nu(r)}.$$
\end{corollary}

The coefficients $N_{\nu}:=\frac{1}{n} \prod_{i=1}^r \binom {n}{ \nu(i)}$ are sometimes called {\em generalized Fuss-Narayana numbers}, in part, because they refine the 
Fuss-Narayana numbers $N_r(n,j) :=\frac {1} {n}  \binom{n}{ j } \binom{(r-1)n}{n-j-1}  $, which in turn refine the Fuss-Catalan numbers $C_{n,r} := \frac 1 {n}  \binom {rn}{n-1}$.  
Indeed Vandermonde's identity yields,
$$\sum_{\substack{\nu \in \wcomp_{n-1,r}\\ \nu(r)=j}} N_{\nu} = N_r(n,j) $$
 and  $$\sum_{j = 0}^{n-1} N_r(n,j) =C_{n,r} .$$  It therefore follows from Corollary~\ref{corollary:parking}
 that
 $$(-1)^{n-1} M_{P_n}(\stackrel {r}{\overbrace{1,1,\dots,1}}) = C_{n,r}.$$
 Note that  $N_r(n,j)$ is the usual Narayana number $N(n,j) $ and  $C_{n,r}$ is the usual  Catalan number $C_n$ when $r=2$. Thus Corollary~\ref{corollary:parking} reduces to Theorem~\ref{theorem:path_narayana} when $r=2$.
 
Another  consequence of  Theorem~\ref{theorem:parking} is that .$(-1)^{n-1} M_{P_n}(\xx)$ is $e$-positive, that is when expanded in the basis of elementary symmetric functions, all the coefficients are nonnegative.   Indeed, by (\ref{equation:parking3}),
\begin{equation} \label{equation:parking_e} (-1)^{n-1} M_{P_n}(\xx) = \sum_{\pi \in \mathcal{NC}_{n-1}} e_{\lambda(\pi)}(\xx).\end{equation}
In the next section we establish $e$-positivity for all chordal graphs  and conjecture that it holds for all graphs.

\section{EL-shellability, \texorpdfstring{$e$}{}-positivity, and \texorpdfstring{$\gamma$}{}-positivity} \label{section:e_positivity}

In this section, we prove $e$-positivity results for the M\"obius symmetric functions. These results reduce to the $\gamma$-positivity results for the M\"obius polynomials  discussed in Section~\ref{subsection:chordal_graphs}.  Our proofs rely on the theory of  lexicographic shellability, which was first introduced by Bj\"orner in \cite{Bjorner1980} and further devoloped by Bj\"orner and Wachs in \cite{BjornerWachs1982,BjornerWachs1983,BjornerWachs1996,BjornerWachs1997}.
We briefly review the basics in Section~\ref{subsection:ELshellability}.  See \cite{Wachs2007} for further information on this topic.

\subsection{Review: EL labeling of the multiweighted partition poset}  \label{subsection:ELshellability}
  
An \emph{edge-labeling} of  a finite bounded poset $P$ is a function $\lambda:\E(P)\rightarrow \Lambda$  from the set $\E(P)$ of cover relations  of $P$ to some other poset $\Lambda$.  An edge labeling $\lambda:\E(P)\rightarrow \Lambda$   of $P$ is said to be an \emph{EL-labeling} if it
   satisfies the following conditions:
\begin{enumerate}
    \item In every closed interval $[x,y]$ of $P$ there is a unique \emph{(strictly) increasing maximal chain}, i.e., a maximal chain $\cc:=(x=x_0\lessdot x_1 \lessdot \cdots \lessdot x_\ell=y)$ such that
    $$\lambda(x_0\lessdot x_1)<\lambda(x_1\lessdot x_2)<\cdots<\lambda(x_{\ell-1}\lessdot x_\ell).$$
    \item The \emph{word of labels}
    $$\lambda(\cc):=\lambda(x_0\lessdot x_1)\lambda(x_1\lessdot x_2)\cdots\lambda(x_{\ell-1}\lessdot x_\ell)$$ associated to the
    unique increasing maximal chain $\cc$ comes, in the lexicographic order, before the word of labels of any other maximal chain in the interval. 
\end{enumerate}

If $P$ admits an EL-labeling, we say that it is  \emph{edge lexicographically shellable (EL-shellable)}. Note that the EL-labeling  restricts to an EL-labeling of every closed interval of $P$.  A maximal chain $x_0\lessdot x_1 \lessdot \cdots \lessdot x_\ell$ of a closed interval of  $P$ is said to be \emph{ascent-free} if its word of labels has no ascents, that is,  for all $i\in [\ell-1]$ we have $\lambda(x_{i-1}\lessdot x_i)\not < \lambda(x_i\lessdot x_{i+1})$. 

For any maximal chain $\cc$ of a closed interval $[x,y]$,  let $\bar\cc:=\cc-\{x,y\}$.  A chain $x_0 < x_1 <\cdots < x_{\ell}$ of length $\ell$ in $P$ will be referred to as an {\em $\ell$-chain}. 

EL-shellability has strong topological and algebraic implications  for the order complex $\Delta(x,y)$ of each open interval $(x,y)$ of the poset.  
 This will be discussed in Section \ref{section:topological_consequences}, while here we  only need its connection with the M\"obius function as given in the following result.

 \begin{theorem}[Stanley \cite{Stanley1974}, 
 cf.\! Bj\"orner {\cite[Theorem~2.7]{Bjorner1980}}]  \label{corollary:shellability_pure}
Let $\lambda$ be an EL-labeling of a pure bounded finite poset $P$. Then 
for every $x \le_P y$ in $P$, $$\mu(x,y)=(-1)^{\ell(x,y)}|\{\bar \cc\mid \cc \text{ is an ascent-free maximal chain of } [x,y]\}|.$$
\end{theorem}

\begin{remark} This theorem is implicit in \cite{Stanley1974}, where a less stringent labeling, now known as an ER-labeling (or R-labeling), was used.  EL-shellabillity was introduced later in \cite{Bjorner1980} in order to obtain stronger  topological conclusions; see Theorem~\ref{theorem:BjornerWachs}.  In both papers the increasing maximal chains were only required to be weakly increasing and the ascents were defined as weak ascents.  The theorem is valid either way; see \cite[Remark 3.25]{Wachs2007}.
\end{remark}

Recall that the poset $\Pi_G^r$ is pure, but not bounded.  Hence we consider the {\em augmented $r$-weighted bond poset} $\widehat{ \Pi_G^r} := \Pi_G^r \cup \{\hat 1\} $, which is pure and bounded.  We now present an EL-labeling of  $\widehat{\Pi_{K_n}^r }= \widehat{\Pi_{n}^r}$ obtained by  Gonz\'alez D'Le\'on and Wachs \cite{GonzalezDLeonWachs2016}  in the $r=2$ case and by Gonz\'alez D'Le\'on \cite{GonzalezDLeon2016} for general   $r \in \ZZ_{>0} \cup \{\infty\}$. Note that  the definition of EL-labeling  makes sense not only for finite bounded posets, but also for {\em infinite} bounded posets in which all maximal chains are finite.  Thus we can extend the definition of EL-labeling to such posets, which include the poset $\widehat{\Pi_{n}^\infty}$.   Note also that all the closed intervals of $\Pi_{n}^\infty $ are finite.  So the restriction of any EL-labeling of $\widehat{\Pi_{n}^\infty} $  to any closed interval of $\Pi_{n}^\infty $ is an EL-labeling in the traditional sense.

For  fixed $r \in \ZZ_{>0}\cup \{\infty\}$ and $a\in [n]$, let $$\Gamma_{a,r}=\{(a,b,j)\mid a<b\le n+1 \text{ and } j\in[r]\}$$ with order relation given by $(a,b,j)\leq (a,b',j')$ whenever $b\le b'$ and $j\leq j'$. (Here $[\infty] := \ZZ_{>0}$.) We define then $\Lambda_{n,r}$ as the ordinal sum
$$\Lambda_{n,r}=\Gamma_{1,r}\oplus \Gamma_{2,r}\oplus \cdots\oplus \Gamma_{n,r}.$$
  See Figure~\ref{figure:example_poset_labelsn3k3} for the Hasse diagrams of $\Lambda_{3,2}$ and $\Lambda_{2,3}$.
 
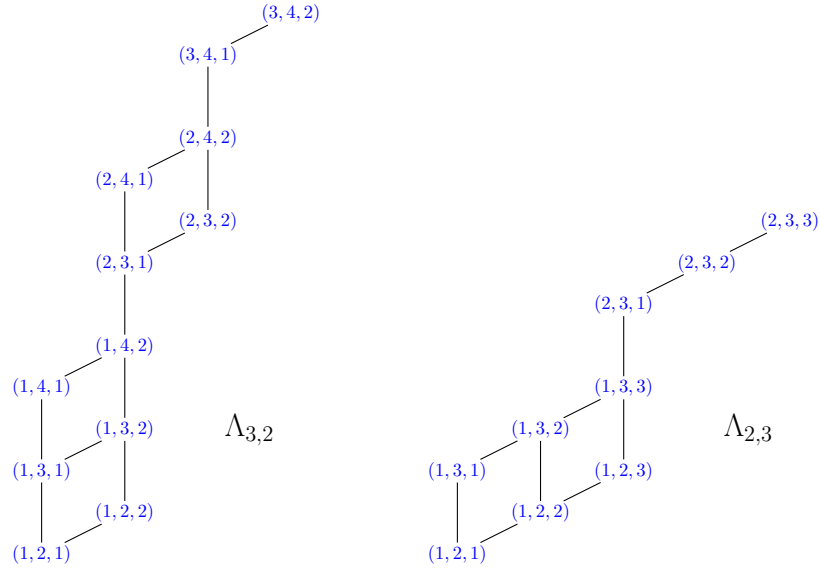
\begin{figure}
      \centering
    \begin{tikzpicture}[scale=1.1]

\begin{scope}

 \tikzstyle{every node}=[inner sep=1pt, minimum width=14pt,scale=0.7, font=\footnotesize]
\draw (0,0) node (n121) {\color{blue}$(1,2,1)$};
\draw (0,1) node (n131) {\color{blue}$(1,3,1)$};
\draw (0,2) node (n141) {\color{blue}$(1,4,1)$};

\draw (1,0.5) node (n122) {\color{blue}$(1,2,2)$};
\draw (1,1.5) node (n132) {\color{blue}$(1,3,2)$};
\draw (1,2.5) node (n142) {\color{blue}$(1,4,2)$};

\draw (n132) -- (n131) ;
\draw (n122) -- (n121) ;
\draw (n131) -- (n121) ;
\draw (n132) -- (n122) ;
\draw (n131) -- (n141) ;
\draw (n132) -- (n142) ;
\draw (n141) -- (n142) ;

\draw (1,3.5) node (n231) {\color{blue}$(2,3,1)$};
\draw (2,4) node (n232) {\color{blue}$(2,3,2)$};
\draw (1,4.5) node (n241) {\color{blue}$(2,4,1)$};
\draw (2,5) node (n242) {\color{blue}$(2,4,2)$};

\draw (n231) -- (n142) ;
\draw (n232) -- (n231) ;
\draw (n241) -- (n242) ;
\draw (n232) -- (n242) ;
\draw (n231) -- (n241) ;

\draw (2,6) node (n341) {\color{blue}$(3,4,1)$};
\draw (3,6.5) node (n342) {\color{blue}$(3,4,2)$};

\draw (n242) -- (n341) ;
\draw (n342) -- (n341) ;

\node at (2.5,1.5) {\Large $\Lambda_{3,2}$};

\end{scope}

\begin{scope}[shift={(5,0)}]

 \tikzstyle{every node}=[inner sep=1pt, minimum width=14pt,scale=0.7, font=\footnotesize]
\draw (0,0) node (n121) {\color{blue}$(1,2,1)$};
\draw (0,1) node (n131) {\color{blue}$(1,3,1)$};

\draw (1,0.5) node (n122) {\color{blue}$(1,2,2)$};
\draw (1,1.5) node (n132) {\color{blue}$(1,3,2)$};

\draw (2,1) node (n123) {\color{blue}$(1,2,3)$};
\draw (2,2) node (n133) {\color{blue}$(1,3,3)$};

\draw (n133) -- (n132) -- (n131) ;
\draw (n123) -- (n122) -- (n121) ;
\draw (n131) -- (n121) ;
\draw (n132) -- (n122) ;

\draw (2,3) node (n231) {\color{blue}$(2,3,1)$};

\draw (3,3.5) node (n232) {\color{blue}$(2,3,2)$};

\draw (4,4) node (n233) {\color{blue}$(2,3,3)$};

\draw (n233) -- (n232) -- (n231) ;

\draw (n231)  -- (n133) -- (n123);

\node at (3.5,1.5) {\Large $\Lambda_{2,3}$};
\end{scope}

\end{tikzpicture}
    \caption{Posets of labels $\Lambda_{3,2}$ and $\Lambda_{2,3}$}
         \label{figure:example_poset_labelsn3k3}
\end{figure}

Recall that if $\balpha \lessdot \bbeta$ in $\Pi^r_n$   then $\bbeta$ is obtained from $\balpha$ by merging two $r$-weighted blocks $A^\mu$ and $B^\nu$ to obtain the $r$-weighted block $(A \cup B)^{{\bf e}_j+\mu+\nu}$, where $j \in [r]$.
 Assume $\min A < \min B$ and let $$\lambda(\balpha \lessdot \bbeta) = (\min A, \min B, j).$$    For all 
$\nu \in \wcomp_{n-1,r}$, let
$$  \lambda([n]^{\nu} \lessdot \hat{1})=(1,n+1,1).$$ This defines an edge labeling
$\lambda:\E(\widehat{\Pi^r_n}) \to \Lambda_{n,r}$.

\begin{theorem}[\cite{GonzalezDLeon2016},  \cite{GonzalezDLeonWachs2016} r=2 case] \label{theorem:ellabelingposet}
 For all $r \in \ZZ_{>0} \cup \{\infty\}$, the edge labeling $\lambda:\E(\widehat{\Pi^r_n})\rightarrow \Lambda _{n,r}$ defined above is an
EL-labeling of $\widehat{\Pi^r_n}$.  
\end{theorem}

 Given any chain ${\bf c}$ of $\Pi_n^r$, the underlying partitions of the weighted partitions of  ${\bf c}$ form a chain  $c$ in $\Pi_n$, which we call the {\em underlying chain} of ${\bf c}$.
 We now describe the underlying chain of the unique increasing maximal chain of each interval of $\Pi^r_n$ under the EL-labeling $\lambda$. 
 \begin{definition} \label{definition:fundamental}
Let $\alpha < \beta \in \Pi_n$
and let
$$B_1,\dots, B_m$$ be the  blocks of  $\beta$, listed in increasing order of their minimum elements.   For each $i=1,\dots, m$, let 
$$A_{i,1}, \dots, A_{i,k_i} $$
be the  blocks  of  $\alpha$ that are merged to form  
the  block $B_i$,   also listed in increasing order of their minimums.  
Now  define the {\em fundamental chain}  $c(\alpha,\beta)$ of the interval  $[\alpha, \beta]$ to be  the concatenation of the  chains $c_0, c_1,\dots,c_m$, where 
$c_0$ is the chain consisting only of the  partition  $\alpha$ and 
for each $i \ge 1$,   the chain $ c_i$   is

$$ (A_{i,1} \cup A_{i,2}) \lessdot_{\Pi_n} (\bigcup_{j=1}^{3}A_{i,j})  \lessdot_{\Pi_n} \cdots \lessdot_{\Pi_n} (\bigcup_{j=1}^{k_i}A_{i,j}) = B_i,$$
where only the newly merged block is listed.  In other words, $c_i$ is the chain $$ \pi_{i,1} \lessdot_{\Pi_n} \pi_{i,2} \lessdot_{\Pi_n} \cdots  \lessdot_{\Pi_n} \pi_{i,k_i-1}$$ where for each $p \in [k_i-1]$, the partition $\pi_{i,p} $ is equal to
$$ \{B_1,\dots,B_{i-1}\} \cup \{ \bigcup_{j=1}^{p+1}A_{i,j}, A_{i,p+2},\dots, A_{i,k_i}\} \cup \bigcup_{j=i+1}^m \{ A_{j,1},\dots, A_{j,k_j} \}.$$ 
\end{definition}

 The following result is extracted from the proof of Theorem~\ref{theorem:ellabelingposet} given in \cite{GonzalezDLeon2016, GonzalezDLeonWachs2016}. 

\begin{lemma}\label{theorem:increasingchains} Let $\balpha < \bbeta \in \widehat{\Pi^r_n}$.
\begin{enumerate}
 \item If $\bbeta \ne \hat 1$ then  the underlying chain of the unique increasing maximal chain  of $[\balpha,\bbeta]$ is the fundamental chain
 $c(\alpha, \beta)$, where $\alpha$ is the underlying partition of $\balpha$ and $\beta$ is the underlying partition of $\bbeta$.
 \item Suppose {$\bbeta= \hat 1$}.  Then the unique increasing maximal chain of $[\balpha,\bbeta]$ is
 the concatenation of the unique increasing maximal chain of $[\balpha, [n]^{w(\balpha)+(|\balpha|-1){\bf e}_1}]$ and $\hat 1$, where $w(\balpha)$ is the sum of the weights of the blocks of $\balpha$.
 \end{enumerate}
\end{lemma}

\subsection{Review: Lyndon-colored trees and ascent-free chains}
In this section we present the description given in \cite{GonzalezDLeon2016, GonzalezDLeonWachs2016} of the ascent-free maximal chains of each maximal interval $[\hat 0, [n]^\nu]$ of $\Pi_n^r$ under the EL-labeling $\lambda$ of Theorem~\ref{theorem:ellabelingposet} in terms of certain leaf labeled binary trees with colored internal nodes.  This, by Corollary~\ref{corollary:shellability_pure}, gives  a combinatorial description of the coefficients of the M\"obius symmetric function.

Recall that for each $r$, each maximal interval $[\hat 0,[n]^\nu]$ of $\Pi_n^r$ can be viewed as a maximal interval of $\Pi_n^\infty $ by viewing the weak $r$-compositions as weak $\infty$-compositions.  Hence we can restrict ourselves to $r=\infty$ in what follows.  

Recall the discussion of leaf-labeled binary trees given in Section~\ref{subsection:chordal_graphs}.  Now we add {\em color} to the internal nodes of these trees.  Let $\internal(T)$ be the set of internal nodes of a  binary tree $T$.  A coloring of a  binary tree $T$ is a function $\clr:\internal(T)\rightarrow \ZZ_{>0}$ that assigns a positive integer or \emph{color} to every internal node of $T$.  For each colored leaf-labeled binary tree $T$, define the {\em color type} of $T$ to be the weak $\infty$-composition $\nu(T) = (\nu(1),\nu(2), \cdots) $, where $\nu(i)$ is the number of internal nodes of $T$ that have  color $i$. 
 For each  $\nu \in \wcomp_{n-1}$, let $\mathcal{CT}_{n,\nu}$ denote the set of  colored  binary 
 trees on leaf set $[n]$ that have color type $\nu$.  

Recall the discussion of  normalized leaf-labeled binary tree and  Lyndon nodes  in Section~\ref{subsection:chordal_graphs}.
A {\em Lyndon-coloring} of a normalized leaf-labeled binary tree $T$  is a coloring $\clr$ that satisfies: 
for all internal nodes  $x$ of $T$ whose left child $L(x)$ is not a leaf, $$\min A_{R(L(x))}< \min A_{R(x)} \implies  \clr(L(x))>\clr(x).$$ In other words, each nonLyndon internal node must be assigned a color that is less than the color assigned to its left child.   Thus only Lyndon trees can have a monochromatic Lyndon-coloring.   For $\nu \in \wcomp_{n-1}$,  let $\Lyn_{n,\nu}$ denote the set of Lyndon-colored normalized binary trees on leaf set $[n]$ that have color type $\nu$. 

An example of a Lyndon-colored normalized binary tree in $\Lyn_{6,(2,2,1)}$ is given in Figure~\ref{figure:example_colored_lyndon_tree_with_linear_extension}.  The small number next to each internal node $x$ is  $\min A_x$ and the circled numbers give an ordering of the internal nodes which will be explained below.   In Figure~\ref{figure:lyndon_basis_example},  all the Lyndon-colored normalized  binary trees on leaf set $[3]$ with colors in $[2]$ are listed. From this we see that
$$|\Lyn_{3,(2,0)}|   = 2\,\,  \mbox{ and } \,\, |\Lyn_{3,(1,1)}|  = 5  \,\,  \mbox{ and } \,\,  |\Lyn_{3,(0,2)}|  = 2$$

 \begin{figure}
    \centering
    \begin{tikzpicture}[scale=1.4]

\tikzstyle{every node}=[fill,draw,inner sep=3pt, scale=1.1, minimum width=4pt,scale=1.5]

\node[circle, red,pin={[red,draw,circle, inner sep =0.5pt, pin distance = 8pt, pin edge={red, thick, <-}]170:\footnotesize$3$}] (v2) at (-2,-0.5) {};

\node [rectangle, blue, inner sep = 4pt, pin={[red,draw,circle,inner sep =0.5pt, pin distance = 8pt, pin edge={red,  thick, <-}]170:\footnotesize$1$}] (v8) at (-0.5,-0.5) {};
\node [diamond, green, inner sep = 2.5pt, pin={[red,draw,circle,inner sep =0.5pt, pin distance = 8pt, pin edge={red, thick, <-}]170:\footnotesize$2$}] (v9) at (0,0) {};
\node[circle, red,pin={[red,draw,circle,inner sep =0.5pt, pin distance = 8pt, pin edge={red, thick, <-}]170:\footnotesize$4$}] (v3) at (-0.75,0.75) {};
\node[rectangle, blue, inner sep = 4pt,pin={[red,draw,circle,inner sep =0.5pt, pin distance = 8pt, pin edge={red, thick, <-}]170:\footnotesize$5$}] (v4) at (-0.25,1.25) {};

\tikzstyle{every node}=[inner sep=0pt, scale=1.1, minimum width=4pt,scale=1.5]

\node (v1) at (-2.5,-1) {$1$};
\node (v6) at (-1.5,-1) {$3$};
\node (v7) at (-1,-1) {$2$};
\node (v10) at (0,-1) {$6$};
\node (v11) at (0.5,-0.5) {$4$};
\node (v5) at (0.25,0.75) {5};
\tikzstyle{every path}=[ thick]

\draw  (v1) edge (v2);
\draw  (v2) edge (v3);
\draw  (v3) edge (v4);
\draw  (v4) edge (v5);
\draw  (v2) edge (v6);
\draw  (v7) edge (v8);
\draw  (v8) edge (v9);
\draw  (v8) edge (v10);
\draw  (v9) edge (v11);
\draw  (v3) edge (v9);
\node[draw, rectangle, blue,inner sep = 3pt, fill] at (-3.425,1.825) {};
\node[draw, circle, red,inner sep = 2pt, fill] at (-3.425,1.325) {};
\node[draw, diamond, green,inner sep = 2pt, fill] at (-3.425,0.825) {};
\node[blue] at (-2.925,1.825) {$1$};
\node[red] at (-2.925,1.325) {$2$};
\node[green] at (-2.925,0.825) {$3$};
\draw [gray] (-3.95,2.5) rectangle (-2.325,0.5);
\node at (-3.125,2.275) {\scriptsize Coloring };
\node[blue] at (-1.75,-0.5) {\tiny 1};
\node[blue] at (-0.5,0.75) {\tiny 1};
\node[blue] at (0,1.25) {\tiny 1};
\node[blue] at (0.25,0) {\tiny 2};
\node[blue] at (-0.25,-0.5) {\tiny 2};
\end{tikzpicture}
    \caption{Example of a Lyndon-colored binary tree in $\Lyn_{6,(2,2,1)}$}    \label{figure:example_colored_lyndon_tree_with_linear_extension}
\end{figure}
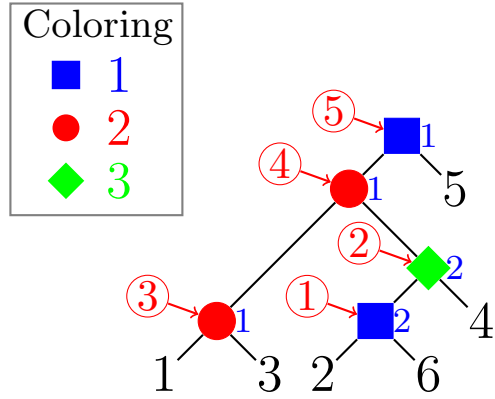

  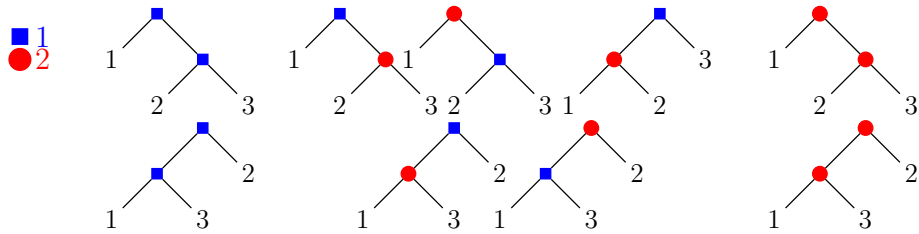
\begin{figure}
    \centering
     \resizebox{\columnwidth}{!}{
     \begin{tikzpicture}[scale=0.5]

\tikzstyle{every node}=[fill, draw,inner sep=3pt,scale=0.8]
    \draw [color=blue] (-1,1.5)  node (i1){};
    \draw [circle, color=red] (-1,1)  node (i2){};
\tikzstyle{every node}=[inner sep=3pt,scale=0.8]
	\draw [color=blue] (-0.5,1.5)  node (i1){$1$};
    \draw [circle, color=red] (-0.5,1)  node (i2){$2$};

\begin{scope}[xshift=1cm,yshift=-2.5cm]

\tikzstyle{every node}=[fill, draw,inner sep=2pt,scale=0.8]
    \draw [color=blue] (1,1)  node (i1){};
    \draw [color=blue] (2,2)  node (i2){};

\tikzstyle{every node}=[inner sep=1pt, minimum width=14pt,scale=0.7]

    \draw (0,0)  node (m){$1$};
    \draw (2,0)  node (l1){$3$};
    \draw (3,1)  node (l2){$2$};

    \draw (m) --  (i1) ;
    \draw (i1) --  (l1) ;
    \draw (i1) --  (i2) ;
    \draw (i2) --  (l2) ;
\end{scope}

\begin{scope}[xshift=0cm]

\tikzstyle{every node}=[fill, draw,inner sep=2pt,scale=0.8]
    \draw [color=blue] (3,1)  node (i1){};
    \draw [color=blue] (2,2)  node (i2){};

\tikzstyle{every node}=[inner sep=1pt, minimum width=14pt,scale=0.7]

    \draw (2,0)  node (m){$2$};
    \draw (4,0)  node (l1){$3$};
    \draw (1,1)  node (l2){$1$};
    
    \draw (m) --  (i1) ;
    \draw (i1) --  (l1) ;
    \draw (i1) --  (i2) ;
    \draw (i2) --  (l2) ;
\end{scope}

\begin{scope}[xshift=4cm]

\tikzstyle{every node}=[fill, draw,inner sep=2pt,scale=0.8]
    \draw [circle,color=red] (3,1)  node (i1){};
    \draw [color=blue] (2,2)  node (i2){};

\tikzstyle{every node}=[inner sep=1pt, minimum width=14pt,scale=0.7]

    \draw (2,0)  node (m){$2$};
    \draw (4,0)  node (l1){$3$};
    \draw (1,1)  node (l2){$1$};
    
    \draw (m) --  (i1) ;
    \draw (i1) --  (l1) ;
    \draw (i1) --  (i2) ;
    \draw (i2) --  (l2) ;
\end{scope}

\begin{scope}[xshift=6.5cm]

\tikzstyle{every node}=[fill, draw,inner sep=2pt,scale=0.8]
    \draw [color=blue] (3,1)  node (i1){};
    \draw [circle,color=red] (2,2)  node (i2){};

\tikzstyle{every node}=[inner sep=1pt, minimum width=14pt,scale=0.7]

    \draw (2,0)  node (m){$2$};
    \draw (4,0)  node (l1){$3$};
    \draw (1,1)  node (l2){$1$};
    
    \draw (m) --  (i1) ;
    \draw (i1) --  (l1) ;
    \draw (i1) --  (i2) ;
    \draw (i2) --  (l2) ;
\end{scope}

\begin{scope}[xshift=11cm]
\tikzstyle{every node}=[fill, draw,inner sep=2pt,scale=0.8]
    \draw [circle,color=red] (1,1)  node (i1){};
    \draw [color=blue] (2,2)  node (i2){};

\tikzstyle{every node}=[inner sep=1pt, minimum width=14pt,scale=0.7]

    \draw (0,0)  node (m){$1$};
    \draw (2,0)  node (l1){$2$};
    \draw (3,1)  node (l2){$3$};

    \draw (m) --  (i1) ;
    \draw (i1) --  (l1) ;
    \draw (i1) --  (i2) ;
    \draw (i2) --  (l2) ;

\end{scope}

\begin{scope}[xshift=6.5cm,yshift=-2.5cm]
\tikzstyle{every node}=[fill, draw,inner sep=2pt,scale=0.8]
    \draw [circle,color=red] (1,1)  node (i1){};
    \draw [color=blue] (2,2)  node (i2){};

\tikzstyle{every node}=[inner sep=1pt, minimum width=14pt,scale=0.7]

    \draw (0,0)  node (m){$1$};
    \draw (2,0)  node (l1){$3$};
    \draw (3,1)  node (l2){$2$};

    \draw (m) --  (i1) ;
    \draw (i1) --  (l1) ;
    \draw (i1) --  (i2) ;
    \draw (i2) --  (l2) ;
\end{scope}

\begin{scope}[xshift=9.5cm,yshift=-2.5cm]
\tikzstyle{every node}=[fill, draw,inner sep=2pt,scale=0.8]
    \draw [color=blue] (1,1)  node (i1){};
    \draw [circle,color=red] (2,2)  node (i2){};

\tikzstyle{every node}=[inner sep=1pt, minimum width=14pt,scale=0.7]

    \draw (0,0)  node (m){$1$};
    \draw (2,0)  node (l1){$3$};
    \draw (3,1)  node (l2){$2$};

    \draw (m) --  (i1) ;
    \draw (i1) --  (l1) ;
    \draw (i1) --  (i2) ;
    \draw (i2) --  (l2) ;
\end{scope}
\begin{scope}[xshift=15.5cm,yshift = -2.5cm]
\tikzstyle{every node}=[fill, draw,inner sep=2pt,scale=0.8]
    \draw [circle,color=red] (1,1)  node (i1){};
    \draw [circle,color=red] (2,2)  node (i2){};

\tikzstyle{every node}=[inner sep=1pt, minimum width=14pt,scale=0.7]

    \draw (0,0)  node (m){$1$};
    \draw (2,0)  node (l1){$3$};
    \draw (3,1)  node (l2){$2$};
    
    \draw (m) --  (i1) ;
    \draw (i1) --  (l1) ;
    \draw (i1) --  (i2) ;
    \draw (i2) --  (l2) ;
\end{scope}

\begin{scope}[xshift=14.5cm]

\tikzstyle{every node}=[fill, draw,inner sep=2pt,scale=0.8]
    \draw [circle,color=red] (3,1)  node (i1){};
    \draw [circle,color=red] (2,2)  node (i2){};

\tikzstyle{every node}=[inner sep=1pt, minimum width=14pt,scale=0.7]

    \draw (2,0)  node (m){$2$};
    \draw (4,0)  node (l1){$3$};
    \draw (1,1)  node (l2){$1$};
    
    \draw (m) --  (i1) ;
    \draw (i1) --  (l1) ;
    \draw (i1) --  (i2) ;
    \draw (i2) --  (l2) ;
\end{scope}
\end{tikzpicture}
    }
    \caption{The Lyndon-colored trees on leaf set $[3]$ with colors in $[2]$.}
    \label{figure:lyndon_basis_example}
\end{figure}

 Binary trees on leaf set $[n]$ are used in \cite{Wachs1998} to encode maximal chains of $\Pi_n=\Pi_n^1$ and bicolored binary trees on leaf set $[n]$ are used in \cite{GonzalezDLeonWachs2016} to encode maximal chains of $\mathcal W\Pi_n= \Pi_n^2$.  The general $r$-colored binary trees on leaf set $[n]$ are used in \cite{GonzalezDLeon2016} to encode maximal chains of $\Pi_n^r$. Now we describe how this is done. 
  
We call an ordering of the internal nodes of a  binary tree $T$ in which each  node appears before its parent, a {\it linear extension} of $T$.  For each $\bT \in \mathcal {CT}_{n,\nu}$ and linear extension $\tau=(x_1,\dots,x_{n-1})$ of $\bT$, we can construct a maximal chain $$c(\bT,\tau)=(\hat{0}=\bpi_0\lessdot \bpi_1\lessdot \cdots \lessdot \bpi_{n-1}=[n]^{\nu})$$ of   $[\hat 0, [n]^\nu]$  such that for all $i \in [n-1]$, the weighted partition $\bpi_i$ is obtained from the weighted partition $\bpi_{i-1}$ by  merging the  weighted blocks $A_{L(x_i)}^{\alpha}$ and $A_{R(x_i)}^{\beta}$ of $\bpi_{i-1}$ to obtain the weighted block $A_{x_i}^{\alpha+\beta+{\bf e}_{\clr(x_i)}}$ of $\bpi_i$.  
 In Figure~\ref{figure:chain_associated_to_a_tree} we illustrate the chain $c(\bT,\tau)$ associated to the colored leaf-labeled tree $\bT$  of  Figure \ref{figure:example_colored_lyndon_tree_with_linear_extension} 
and the linear extension $\tau$ given by the circled numbers. 
 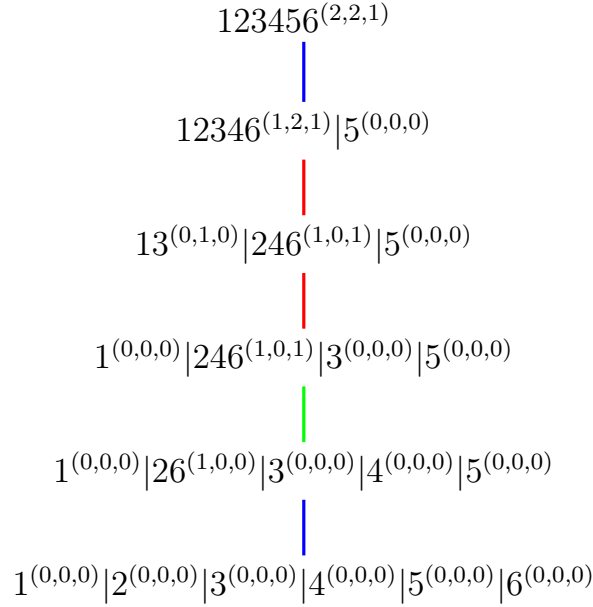
\begin{figure}
    \centering
    \begin{tikzpicture}

\tikzstyle{every node}=[inner sep=3pt, scale=1.1, minimum width=4pt]

\node (v1) at (-2.5,-1) {$1^{(0,0,0)}|2^{(0,0,0)}|3^{(0,0,0)}|4^{(0,0,0)}|5^{(0,0,0)}|6^{(0,0,0)}$};
\node (v2) at (-2.5,0.5) {$1^{(0,0,0)}|26^{(1,0,0)}|3^{(0,0,0)}|4^{(0,0,0)}|5^{(0,0,0)}$};
\node (v3) at (-2.5,2) {$1^{(0,0,0)}|246^{(1,0,1)}|3^{(0,0,0)}|5^{(0,0,0)}$};
\node (v4) at (-2.5,3.5) {$13^{(0,1,0)}|246^{(1,0,1)}|5^{(0,0,0)}$};
\node (v5) at (-2.5,5) {$12346^{(1,2,1)}|5^{(0,0,0)}$};
\node (v6) at (-2.5,6.5) {$123456^{(2,2,1)}$};

\draw[very thick, color=blue]  (v1) edge (v2);
\draw[very thick, color=green]  (v2) edge (v3);
\draw[very thick, color=red]  (v3) edge (v4);
\draw[very thick, color=red]  (v4) edge (v5);
\draw[very thick, color=blue]  (v5) edge (v6);

\end{tikzpicture}
    \caption{The chain $c(\bT,\tau)$ for the colored leaf-labeled tree $\bT$ and linear extension $\tau$ of Figure~\ref{figure:example_colored_lyndon_tree_with_linear_extension}.}
    \label{figure:chain_associated_to_a_tree}
\end{figure}

Note that every maximal chain of $[\hat 0, [n]^\nu]$ can be represented as $c(\bT,\tau)$ for some  $\bT\in \mathcal{CT}_{n,\nu}$ and  linear extension $\tau$ of $\bT$.  However, the pair $(\bT,\tau)$ is not unique unless  we restrict to  normalized trees in $\mathcal{CT}_{n,\nu}$.
  It is not difficult to see that for each  binary tree $T$ on leaf set $[n]$, there is a unique linear extension $(x_1,\dots,x_{n-1} )$  of $T$ such that
$$\min A_{x_1} \ge \min A_{x_2}  \ge \cdots \ge \min A_{x_{n-1}}. $$
We will call this linear extension the \emph{decreasing-minimum} linear extension of $T$ and denote it by $\tau_T$. The linear extension depicted by the circled numbers in Figure \ref{figure:example_colored_lyndon_tree_with_linear_extension} is the decreasing-minimum linear extension $\tau_T$ of the  colored tree $T$ in the figure.

 \begin{theorem}[\cite{GonzalezDLeon2016}, r=2 case \cite{GonzalezDLeonWachs2016}]\label{theorem:ascent_free_chains_K} For all $\nu \in \wcomp_{n-1,r}$,  $$\{ c(\bT,\tau_{\bT})\mid \bT \in \Lyn_{n,\nu}\}$$ is the set of ascent-free maximal chains of the interval $[\hat{0},[n]^\nu]$ under the 
  EL-labeling  of $\widehat{\Pi_n^\infty}$ given inTheorem~\ref{theorem:ellabelingposet}.
\end{theorem}
Now by Corollary~\ref{corollary:shellability_pure}, we have the following result.

\begin{corollary}[\cite{GonzalezDLeon2016}] For all $n \ge 1$, 
$$(-1)^{n-1} M_{K_n} (\xx) = \sum_{\nu \in \wcomp_{n-1}} |\Lyn_{n,\nu}| \xx^\nu.$$
\end{corollary}

From this corollary, one can show that $(-1)^{n-1} M_{K_n}(\xx)$ is $e$-positive. Details are given in a more general setting in Section~\ref{subsection:chordal}.  From Figure~\ref{figure:lyndon_basis_example}, one can see that
$$M_{K_3} (x_1,x_2) = 2x_1^2  +5x_1x_2+2x_2^2 = e_2(x_1,x_2)+ 2e_{1,1}(x_1,x_2)$$

\subsection{Generalization to chordal graphs} \label{subsection:chordal}
In this section, we generalize the results of the previous section to multiweighted bond posets of chordal graphs.  

Recall the definitions of  PEO and  {\em perfectly labeled chordal graph} from Section~\ref{section:gamma_positive}. 
The following observation gives an equivalent definition of a  PEO.
\begin{proposition}\label{lemma:equivalent_condition_PEO}
An ordering $v_1,v_2,\cdots,v_n$ of the vertices of a graph $G$ is a PEO if and only if whenever $i<j<k$, and $\{v_i,v_k\}$ and $\{v_j,v_k\}$ are edges of $G$ then $\{v_i,v_j\}$ must also be an edge of $G$, i.e., there are no induced subgraphs that are paths of the form $v_i-v_k-v_j$.
\end{proposition}

\begin{corollary} \label{corollary:induced_PEO} Let $G$ be a perfectly labeled chordal graph and let $H$ be an induced subgraph.  Then $H$ is a perfectly labeled chordal graph. 
\end{corollary}

Let $G$ be a graph on vertex set $\{v_1<\dots< v_n \} \subseteq \ZZ_{>0}$.  We say that a sequence $v_{j_1},v_{j_2},\dots,v_{j_m}$  is  an {\em increasing path} from $v_{j_1}$ to $v_{j_m}$ if  $j_1<j_2< \dots< j_m$ and $\{v_{j_i} ,v_{j_{i+1}}\}$ is an edge of $G$  for all $i \in [m-1]$.

\begin{lemma}\label{lemma:increasing_path}
Let $G$ be a connected perfectly labeled chordal graph on vertex set $\{v_1<\dots< v_n \} \subseteq \ZZ_{>0}$. Then for any $j\in [n]$ there is an increasing path from $v_1$ to $v_j$.
\end{lemma}
\begin{proof}
Using the fact that any induced subgraph of a perfectly labeled chordal graph is a perfectly labeled chordal graph (Corollary~\ref{corollary:induced_PEO}), we will proceed by induction on the number $n$ of vertices of $G$. In graphs with one vertex the statement is trivial, so we assume that the lemma is true for any graph with less than $n$ vertices and let $G$ be a graph with $n$ vertices. Suppose there are no increasing paths from $v_1$ to $v_j$. Note then in particular,  $j>1$ and $\{v_1,v_j\}$ cannot be an edge of $G$. If the graph $G-\{v_j\}$ has more than one component, we can consider the induced subgraph formed by $v_j$ and the component of $G-\{v_j\}$ that contains $v_1$. Since this is a connected graph with less than $n$ vertices, we will find an increasing path from $v_1$ to $v_j$, which gives a contradiction. This implies that  $G-\{v_j\}$ has to be connected. 

Let $v_k$ be the neighbor of $v_j$ in $G$ with $k$ as small as possible. We know that $k>1$ but it must happen that $k>j$, otherwise if $k<j$ we could find an increasing path $p$ from $v_1$ to $v_k$ in $G-\{v_j\}$ and then  $p$ concatenated with $v_j$ is an increasing path  from $v_1$ to $v_j$ in $G$. So let now $p$ be an increasing path in $G-\{v_j\}$ from $v_1$ to $v_k$ and let $\{v_i,v_k\}$ be its last edge. Since $p$ is increasing $i<k$. 
Taking into account the minimality of $k$, the vertices $v_i,v_j, v_k$ provide a violation to the condition of a PEO given in Proposition~\ref{lemma:equivalent_condition_PEO}, which gives us a contradiction. Hence there must be an increasing path from $v_1$ to $v_j$ for all $j \in [n]$.
\end{proof}

We will need the following lemma to prove that $r$-weighted bond posets of chordal graphs are EL-shellable.

\begin{lemma}\label{lemma:connectedincreasing}
Let $G$ be a connected perfectly labeled chordal graph on $[n]$. Let $B_1,B_2,\dots, B_k$ be a collection of disjoint subsets of $[n]$ such that 
\begin{itemize}
    \item $\min B_1< \min B_2<\cdots<\min B_k$
    \item $G|_{B_j}$ is connected for all $j$
    \item $G|_{B_1\cup B_2\cup \cdots \cup B_k}$ is connected.
\end{itemize}
  Then for every $s\le k$, the subgraph $G|_{B_1\cup B_2\cup \cdots \cup B_s}$ is connected.
\end{lemma}
\begin{proof}
We proceed by induction on $s$.  If $s=1$ we have from the hypothesis that $G|_{B_1}$ is connected; so let us assume that $G|_{B_1\cup B_2\cup \cdots \cup B_{s-1}}$ is connected. Since $G|_{ B_1\cup B_2\cup \cdots \cup B_k}$ is connected we know from Corollary~\ref{corollary:induced_PEO} and Lemma \ref{lemma:increasing_path} that there is an increasing path $p$ from $\min B_1$ to $\min B_s$. Let $\{j,\min B_s\}$ be the last edge of $p$. Since $p$ is increasing we have that $j<\min B_s$ and hence, by the fact that the $B_i$'s are listed in increasing order of their minimums, we have that $j\in B_1\cup\cdots\cup B_{s-1}$. Hence $G|_{ B_1\cup \cdots \cup B_s}$ is connected.
\end{proof}

We are now ready to prove that the augmented $r$-weighted bond posets of chordal graphs are EL-shellable. Let $G$ be a connected graph on $[n]$ and $r \in \ZZ_{>0} \cup \{\infty\}$. Note that each maximal chain of  $ \widehat{\Pi^r_G}$ is a maximal chain of 
$ \widehat{\Pi^r_n}$.
 Hence any edge labeling of $ \widehat{\Pi^r_n}$ can be restricted to an edge labeling of $ \widehat{\Pi^r_G}$.

 \begin{theorem} \label{theorem:ELbond} Let $G$ be a connected perfectly labeled chordal graph on $[n]$ and let $r \in \ZZ_{>0} \cup \{\infty\}$. Then the  restriction to $\E( \widehat{\Pi^r_G})$ of the EL labeling $\lambda:\E( \widehat{\Pi^r_n}) \to \Lambda_{n,r}$ given in Theorem~\ref{theorem:ellabelingposet} is an EL-labeling of $ \widehat{\Pi^r_G}$.
\end{theorem}

\begin{proof} Let $\balpha<\bbeta$ in $ \widehat{\Pi^r_G}$. One need only show that the unique increasing maximal chain of the interval $[\balpha,\bbeta]$ in $ \widehat{\Pi^r_n}$ is also a chain $ \widehat{\Pi^r_G}$.  The result will then follow from Theorem~\ref{theorem:ellabelingposet}.

Suppose $\bbeta \ne \hat 1$ and let $\alpha$ and $\beta$ be the underlying partitions of $\balpha$ and $\bbeta$, respectively.   Let $A_{i,j}$, $1\le i \le m$, $1 \le j \le k_i$, be as in Definition~\ref{definition:fundamental}. Since $\alpha,\beta \in  \Pi_G$, each induced subgraph  $G|_{A_{i,j}}$ is connected as is $G|_{\bigcup_{j=1}^{k_i} A_{i,j}}$. It therefore follows from Lemma~\ref{lemma:connectedincreasing} that 
the induced subgraph $G|_{\bigcup_{j=1}^s A_{i,j} }$ is connected for all $i= 1,\dots,m$ and  $s = 1,\dots, k_i$.  Hence, the chains $c_i$ of Definition~\ref{definition:fundamental} are in $ \Pi_G$, which means that the fundamental chain $c(\alpha,\beta)$ of $[\alpha,\beta]$ (i.e, the concatenation  of $c_0,c_1,\dots,c_m$) is as well.  Now by Lemma~\ref{theorem:increasingchains} (1),   the unique increasing maximal chain of the interval $[\balpha, \bbeta]$ of $ \widehat{\Pi^r_n}$ is also a chain in $ \widehat{\Pi^r_G}$.  

Now suppose $\beta = \hat 1$.  Let $\nu \in \wcomp_{n-1,r}$ and let ${\bf c}_\nu$ be the unique increasing maximal chain of $[\balpha, [n]^\nu]$ in $ \widehat{\Pi^r_n}$.  By the previous case, ${\mathbf c}_\nu$ is also in $ \widehat{\Pi^r_G}$. Hence so is the concatenation of ${\bf c}_\nu$ with $\hat 1$.  This means that the unique increasing maximal chain of $[\balpha, \bbeta]$ in $ \widehat{\Pi^r_n}$, as identified by Lemma~\ref{theorem:increasingchains} (2), is also in $ \widehat{\Pi^r_G}$.  \end{proof}

\begin{corollary} If $G$ is a  chordal graph then all closed intervals of $\Pi^r_G$ are EL-shellable for all $r \in \ZZ_{>0} \cup \{\infty\}$.  
\end{corollary}

\begin{proof} We only need to prove this for maximal closed intervals.  Since $G$ is chordal,  its connected components $G_1,\dots, G_k$ are chordal as well.   Since every chordal graph can be perfectly  labeled, each $\widehat{\Pi_{G_i}^r}$ is EL-shellable by Theorem~\ref{theorem:ELbond}.  It follows that each maximal closed interval of $\Pi^r_{G_i}$ is EL-shellable. By Proposition~\ref{proposition:multi_isomorphism_product} each maximal closed interval of $\Pi^r_G$ is isomorphic to a product $P_1 \times \cdots \times P_k$, where each $P_i$ is a maximal closed interval of $\Pi^r_{G_i}$.  The result now follows from \cite[Theorem~4.3]{Bjorner1980}.
\end{proof}

\begin{question}Is $ \widehat{\Pi^r_G}$ EL-shellable for every graph $G$ and for all $r \in \ZZ_{>0} \cup \{\infty\}$?
\end{question}

We have the following generalization of Theorem~\ref{theorem:ascent_free_chains_K}.
\begin{theorem} \label{theorem:ascent_free_G} Let $G$ be a connected perfectly labeled chordal graph on $[n]$ and let $\nu \in \wcomp_{n-1}$. Then the set of ascent-free maximal chains of the interval $[\hat{0},[n]^{\nu}]_G$ for the 
  EL-labeling  of $\widehat{\Pi^\infty_G}$ given in Theorem~\ref{theorem:ELbond} is $$\{c(\bT,\tau_{T})\mid \bT \in \Lyn_{G,\nu}\},$$ where $$\Lyn_{G,\nu} := \{ \bT \in \Lyn_{n,\nu} : G|_{A_x} \mbox{ is connected for all nodes $x$ of $\bT$}\}.$$
  \end{theorem}

\begin{proof} Since the EL-labeling of $\widehat{\Pi^\infty_G}$ is the restriction of the EL-labeling of $\widehat{\Pi^\infty_{K_n}}$, the ascent-free chains of any interval $[\balpha, \bbeta]$ of  $\widehat{\Pi^\infty_G}$ are the ascent-free chains of the  interval $[\balpha, \bbeta]$ of  $\widehat{\Pi^\infty_{K_n}}$ that are also in $\widehat{\Pi^\infty_G}$.  The ascent-free maximal chains of the interval $[\hat 0, [n]^\nu]_{K_n}$, as given in Theorem~\ref{theorem:ascent_free_chains_K}, have the form $c(\bT,\tau_T) $, where $\bT \in \Lyn_{n,\nu}$. Clearly,  the chain $c(\bT,\tau_{\bT})$ is in $\widehat{\Pi^\infty_G}$ if and only if $G|_{A_x} $ is  connected for all nodes $x$ of $\bT$.  Thus the ascent-free  chains of $[\hat 0, [n]^\nu]$ are precisely the chains $c(\bT,\tau_{\bT}) $ for which $\bT \in \Lyn_{G,\nu}$.
\end{proof}

Now by Corollary~\ref{corollary:shellability_pure} we have the following consequence.
\begin{corollary}  \label{corollary:lyndon_mobius} Let $G$ be a connected perfectly labeled chordal graph on $[n]$.  Then  $$(-1)^{n-1} M_G({\bf x}) =  \sum_{\nu \in \wcomp(n-1)} | \Lyn_{G,\nu}| {\bf x}^\nu .$$
\end{corollary}

\begin{example} For $G=P_3$, the only trees of Figure~\ref{figure:lyndon_basis_example} that are in $\Lyn_{G,\nu}$ for some $\nu$ are the ones in the first row.  Hence by Corollary~\ref{corollary:lyndon_mobius}, $$M_{P_3}(x_1,x_2) = x_1^2 + 3x_1x_2 + x_2^2 = e_2(x_1,x_2) + e_{1,1}(x_1,x_2).$$
\end{example}

Given  a normalized binary tree $T$ on node set $[n]$, define the {\em Lyndon set partition} $\pi(T)$ to be the finest partition of the set $\internal(T)$ that satisfies if $x\in \internal(T)$ is not a Lyndon node, then $x$ and $L(x)$ are in the same block of $\pi$.  The \emph{Lyndon type} $\lambda(T)$ of $T$ is the partition of $n-1$ whose parts are the sizes of the blocks of the Lyndon set partition $\pi(T)$; that is 
\begin{equation} \label{equation:def_lambda_T} \lambda(T) = \lambda(\pi(T)),\end{equation}  where $\lambda(\pi) $ is the number partition associated with the set partition $\pi$, as in \eqref{equation:set_to_number}.
For example, for the tree $T$ of Figure \ref{figure:example_colored_lyndon_tree_with_linear_extension} (with colors ignored) we have that $\lambda(T)=(2,1,1,1)$.

\begin{theorem}  \label{theorem:e_positive} Let $G$ be a connected perfectly labeled chordal graph on $[n]$.  Then
$$  (-1)^{n-1} M_G({\bf x})=\sum_{T\in \mathcal N_G}e_{\lambda(T)}(\xx),
$$
where $\mathcal N_G$ is the set of normalized  binary trees $T$ on leaf set $[n]$  for which $G|_{A_x }$ is connected for all nodes $x$ of $T$.

Consequently, if $G$ is a connected chordal graph then $(-1)^{n-1} M_G({\bf x})$ is $e$-positive.
\end{theorem}
\begin{proof} Let $T \in \mathcal N_G$. Each block  of $\pi(T)$ forms a path $y_1,y_2, \dots, y_k$  of internal nodes of $T$ such that for each $i \in [k-1]$,  $L(y_i) = y_{i+1} $ and  $y_i$  is  nonLyndon.  It follows that a coloring map $\clr:\internal(T) \to \ZZ_{>0}$ is a  Lyndon-coloring  if and only if 
$$ \clr(y_1)< \clr(y_2)< \dots < \clr(y_k) $$ for all blocks.

Hence,
\begin{equation} \label{equation:fixed_type} \sum_{\substack{{\bf T} \in \Lyn_{G} \\{\bf T} = (T, \clr)} } {\bf x}^{\nu({\bf T})} =  e_{\lambda(T)}({\bf x}),\end{equation}
where 
$$\Lyn_{G} := \bigcup_{\nu \in \wcomp_{n-1}} \Lyn_{G,\nu} $$
and $T$ is the underlying uncolored tree of ${\bf T}$.

From  Corollary \ref{corollary:lyndon_mobius} and Equation \eqref{equation:fixed_type} we have 
\begin{align*}
(-1)^{n-1} M_G({\bf x})    &=\sum_{\nu \in \wcomp_{n-1}}|\Lyn_{G,\nu}|\xx ^{\nu}\\
   &=\sum_{\bT \in \Lyn_{G}}\xx^{\nu(\bT)}\\
   &=\sum_{T\in \mathcal N_G}\,\sum_{\substack{{\bf T} \in \Lyn_{G} \\{\bf T} = (T, \clr)} } {\bf x}^{\nu({\bf T})}\\
   &= \sum_{T\in \mathcal N_G} \,e_{\lambda(T)}(\xx).
\end{align*}

\end{proof}

\begin{proof}[Proof of Theorem~\ref{theorem:gamma_positive}] This is a special case of Theorem~\ref{theorem:e_positive}.  Indeed, we use the facts that $\mu_G(t) = M_G(1,t)$ and 
$$e_\lambda(1,t) = \begin{cases} t^{m_2(\lambda) } (1+t)^{m_1(\lambda)} &\mbox{ if $\lambda$ has no parts greater than $2$} \\
0 &\mbox{ otherwise,}
\end{cases}$$
where $m_i(\lambda)$ is the number of parts of $\lambda$ that are equal to $i$.
Clearly for $T \in \mathcal N_G$,  the partition $\lambda(T)$ has no parts greater than $2$ if and only if $T \in \widehat{\mathcal N_G} $.  In this case $m_2(\lambda(T))$ equals  the number  $m(T)$ of nonLyndon nodes of $T$ and $m_1(\lambda(T))= n-1-2m_2(\lambda(T)).$ Hence the result follows from Theorem~\ref{theorem:e_positive}.
\end{proof}

\begin{corollary} \label{corollary:e_positive} Let $G$ be a  chordal graph with $n$ vertices and $k$ connected components. 
Then the symmetric function $(-1)^{n-k} M_G(\xx)$ is $e$-positive.  
\end{corollary}

\begin{proof}This follows from Theorem~\ref{theorem:e_positive}, Proposition~\ref{proposition:M_product}, and the fact that  products of $e$-positive symmetric functions are $e$-positive.  
\end{proof}

\begin{corollary}\label{corollary:e_positive_difference} Let $G$ and $H$ be    perfectly labeled chordal graphs  with $n$ vertices and $k$ connected components.  If  $H$ is a subgraph of $G$  
then $$(-1)^{n-k}(M_{G}({\xx}) -  M_{H}({\xx}))$$ is $e$-positive. 
\end{corollary}

\begin{proof}  

Suppose $G$ is connected.  Then $H$ must be as well since it has the same number of connected components as $G$.   Clearly $\mathcal N_H \subseteq \mathcal N_G$.  Hence by Theorem~\ref{theorem:e_positive}
$$(-1)^{n-1}(M_{G}({\xx}) -  M_{H}({\xx})) = \sum_{T \in \mathcal N_G - \mathcal N_H} e_{\lambda(T)}(\xx).$$

Now suppose $G$ has connected components $G_i=(V_i,E_i)$, where $i \in [k]$.  Then $H$ must have connected components $H_i= (V_i, E_i^\prime)$ where $E_i^\prime \subseteq E_i$.  It follows from Theorem~\ref{theorem:e_positive} and Corollary~\ref{corollary:e_positive_difference} in the connected case that all $(-1)^{|V_i|-1}M_{G_i}({\xx})$, $(-1)^{|V_i|-1}M_{H_i}({\xx})$, and  $(-1)^{|V_i|-1}( M_{G_i}({\xx}) - M_{H_i}({\xx}))$ are $e$-positive.  
By Proposition~\ref{proposition:M_product},
$$ (-1)^{n-k} (M_G({\xx}) -  M_H(\xx) )= \prod_{i=1}^k (-1)^{|V_i|-1}M_{G_i}({\xx}) - \prod_{i=1}^k (-1)^{|V_i|-1} M_{H_i}({\xx}).$$ 
We now use the fact that if $f_i,g_i$, where $ i \in [k]$, are symmetric functions and the $f_i$'s, $g_i$'s and $f_i-g_i$'s are $e$-positive then $\prod_{i=1}^k f_i -  \prod_{i=1}^k g_i  $ is $e$-positive.  This follows by induction from the $k=2$ case, which holds because $$f_1 f_2 - g_1 g_2 = (f_1-g_1)f_2 + (f_2-g_2)g_1 $$ and products and sums of $e$-positive symmetric functions are $e$-positive.  Thus $(-1)^{n-k} (M_G({\xx}) -  M_H(\xx) )$ is $e$-positive.
\end{proof}

\begin{conjecture} \label{conjecture:e_positive} Let $G$ be a   graph with $n$ vertices and $k$ connected components.  Then 
\begin{enumerate} \item 
$(-1)^{n-k} M_G(\xx) $  is $e$-positive. 
\item If $H$ is a subgraph of $G$ with $n$ vertices and $k$ connected components  then 
$$(-1)^{n-k}(M_G({\xx}) -   M_H({\xx}))$$ 
 is $e$-positive.
\end{enumerate}
 \end{conjecture}

We have verified (1) for this conjecture computationally for all  graphs with up to $7$ vertices. 

\subsection{Examples}

Let $\mathcal {UT}_{n,r}$ denote the set of unlabeled (i.e. no leaf labels) $r$-ary trees with $n$ internal nodes.  For each $\nu \in \wcomp_{n,r}$, let $\mathcal {UT}_{\nu,r}$ be the set of $r$-ary planar trees $T$  such that,  for each $i$,  $\nu(i)$ is the number of internal nodes of $T$ whose   $i$th child (from left to right) is an internal node.   

\begin{theorem} \label{theorem:r_ary_trees} For all $r, n \ge 1$,  $$(-1)^{n-1} M_{P_n}(x_1,\dots,x_r) =  \sum_{\nu \in \wcomp_{n,r}}  |\mathcal {UT}_{\nu,r} | \,\, x_1^{\nu(1)} \cdots x_r^{\nu(r)}. 
$$
\end{theorem}

\begin{proof} For any graph $G$ on $[n]$, let $\Lyn_{G,r} = \bigcup_{\nu \in \wcomp_{n-1,r} } \Lyn_{G,\nu}$.  The result follows from Corollary~\ref{corollary:lyndon_mobius} and the direct bijection between $\Lyn_{P_n,r}$ and $\mathcal {UT}_{n,r}$ given below.  By the discussion in Example~\ref{example:path}, we can  view the trees in $\Lyn_{P_n,r}$ as  {\em unlabeled} binary trees whose internal nodes are colored with colors in $[r]$.  We represent such a binary tree $T$ as
$T=T_L \overset{c}{\land} T_R$, where $T_L$ is the colored left subtree of $T$,   $T_R$ is the colored right subtree of $T$, and $c$ is the color of the root of $T$.
Then the (unlabeled) tree $T \in \Lyn_{P_n,r}$ if and only if $$T=( \dots (T_0  \overset{c_1}{\land} T_1) \overset{c_2}\land \cdots )\overset{c_{k}} \land T_k,$$ where $T_0$ is a single (unlabeled) leaf, each (unlabeled) $T_i \in \Lyn_{P_{n_i},r}$, for some $n_i$, and 
$r\ge c_1 > c_2> \dots > c_{k}   \ge 1$.  Now let $\psi(T)$ be the $r$-ary tree defined recursively by letting the $i$th subtree of the root of $\psi(T)$ be $\psi(T_j)$ if $i = c_j$ and be a leaf if 
 $i \in [r]-\{c_1,\dots, c_k\}$. This defines a map from $\Lyn_{P_n,r}$ to $\mathcal {UT}_{n,r}$, which is easily seen to be a bijection.  Also,   for  each $\nu \in \wcomp_{n,r}$, we have $T \in \Lyn_{P_n,\nu}$ if and only if $\psi(T) \in \mathcal {UT}_{\nu,r}$.  Hence $|\Lyn_{P_n,\nu}| = |\mathcal {UT}_{\nu,r}|$.
\end{proof}  

\begin{remark} Note that Corollary~\ref{corollary:parking} and Theorem~\ref{theorem:r_ary_trees} imply that  $|\mathcal {UT}_{\nu,r} |$ is equal to the generalized Fuss-Narayana number $N_\nu$ 
 for all $\nu \in \wcomp_{n-1,r}$. This is a refinement of the well known result that the Fuss-Catalan numbers count $r$-ary trees.
 \end{remark}
 
 \begin{remark}An alternative proof of  Theorem~\ref{theorem:parking} can be given that uses   Theorem~\ref{theorem:e_positive} and a  bijection from  $\mathcal N_{P_n}$   to the set of noncrossing partitions $\mathcal {NC}_{n}$.  By the discussion in Example~\ref{example:path},  we can ignore the leaf labels and view the trees in $\mathcal N_{P_n}$  as trees in $\mathcal{UT}_{n,2}$. The bijection is as follows.  For  each $T \in \mathcal{UT}_{n,2}$, label its internal nodes  
 in inorder (i.e. left subtree labeled in inorder first, then root, then right subtree labeled in inorder)  and let $\pi(T)$ 
be the partition  of $[n]$ whose blocks  are the  maximal left paths  of labeled internal nodes of $T$,  that is,  the 
blocks of $\pi(T)$ are of the form $\{x_1,\dots, x_m\}$, where $L(x_i) = x_{i+1}$ for all $i\in [m-1]$, $x_1$ is not a left 
child and the left child of $x_m$ is a leaf.  We leave it to the reader to check that this defines a bijection  $\pi:\mathcal N_{P_n}=\mathcal{UT}_{n,2} \to \mathcal {NC}_{n}$.    Note that  the Lyndon type $\lambda(T)$ of each   $T\in \mathcal N_{P_n}$ is equal to the block type $\lambda(\pi(T))$.  Hence $\pi$ is a bijection from $\mathcal N_{P_n}$ to $\mathcal {NC}_{n}$ that takes Lyndon type to block type.  Theorem~\ref{theorem:parking} now follows from (\ref{equation:parking3}) and Theorem~\ref{theorem:e_positive}.
 \end{remark}

Next we apply Theorem~\ref{theorem:e_positive} to the star graph $St_n$.  
For each subset $S=\{s_1,s_2,\dots,s_{k-1}\}_<$ of $[n-1]$,  the composition of $n$ associated with $S$ is defined by, 
$$\co(S) = (s_1,s_2-s_1,s_3-s_2,\dots,n-s_{k-1}).$$ 
Given a composition $\nu$ with $k$ parts, let $e_\nu$ be the elementary symmetric function $e_{\nu(1)}e_{\nu(2)}\cdots e_{\nu(k)}$. Recall, the {\em descent set} $\Des(\sigma)$ of a permutation $\sigma \in \mathfrak S_n$ is defined to be
$$\Des(\sigma) := \{i \in [n-1] : \sigma(i) > \sigma(i+1) \}.$$

\begin{theorem}[\cite{GonzalezDLeon2018}] \label{corollary:star_e_pos} For all $n \ge 1$, $$(-1)^{n} M_{St_{n+1}}(\xx) = \sum_{\sigma \in \mathfrak S_{n} } e_{\co(\Des(\sigma))}(\xx).$$
\end{theorem}

\begin{proof} Consider the bijection, described in  Example~\ref{example:star}, from $\mathcal N_{St_{n+1}}$ to  $\mathfrak S_{n}$   that takes $T$ to $\sigma_T$.  Also recall the definition of the Lyndon-set partition $\pi(T)$ used to define the Lyndon type   as $\lambda(T) = \lambda(\pi(T))$ in (\ref{equation:def_lambda_T}).

 Note that for all $T \in \mathcal N_{St_{n+1}}$, 
 each block of the partition $\pi(T)$  is in bijective correspondence with  a maximal ascending run of the permutation $\sigma_T$, where a maximal ascending run of a permutation $\sigma \in \mathfrak S_{n}$ is a set of the form $\{\sigma(s_{i}+1),\sigma(s_{i}+2), \dots, \sigma(s_{i+1})\}$, where 
 $\{s_0,s_1,\dots,s_k\}_< = \Des(\sigma) \cup \{0,n\}$.  It follows that the Lyndon type $\lambda(T)$ is the weakly decreasing rearrangement of $\co(\Des(\sigma_T))$.  Hence for all $T \in \mathcal N_{St_{n+1}}$, we have  $e_{\lambda(T)} = e_{\co(\Des(\sigma_T))} $.  The result now follows from  Theorem~\ref{theorem:e_positive}.
\end{proof}

Note that since $\tilde A_n(t) = (-1)^{n} \mu_{St_{n+1}}(t) = (-1)^{n} M_{St_{n+1}} (1,t)$,  Equation  (\ref{equation:tildeA_gamma}) is a consequence of Theorem~\ref{corollary:star_e_pos}.

\begin{remark}
 Let $\Lambda_{\QQ}^n$ be the vector space  of homogeneous symmetric functions of degree $n$ over $\QQ$.  Let $\eta: \Lambda_{\QQ}^n \to \QQ[t]$ be the homomorphism defined on the basis $\{e_\lambda: \lambda \vdash n\}$ by $\eta(e_\lambda) = t^{\ell(\lambda)-1}$.

(1) In \cite{GonzalezDLeon2016} Gonz\'alez D'Le\'on shows that $(-1)^{n-1}\eta( M_{K_n}(\xx)) $ is equal to a polynomial whose coefficients are  called {\em second order Eulerian numbers} in \cite{GrahamKnuthPatashnik1989}.  Stanley and Gessel \cite{StanleyGessel1978} give a combinatorial interpretation of these polynomials in terms of descent number of certain multiset permutations that they call {\em Stirling permutations}.   In \cite{GonzalezDLeon2019} Gonz\'alez D'Le\'on gives a combinatorial interpretation of  $(-1)^{n-1} M_{K_n}(\xx)$ in terms of Stirling permutations, which, under the specialization $\eta$, reduces to the Stanley-Gessel interpretation of the second order Eulerian polynomials.  

(2)  It follows from  Equation~(\ref{equation:parking_e}) and a well-known interpretation of the Narayana polynomials as $N_{n}(t) = \sum_{\pi \in \mathcal{NC}_n} t^{|\pi|-1}$ that
$$(-1)^{n-1}  \eta(M_{P_n}(\xx))= N_{n-1}(t) .
$$ 

(3) It follows from Theorem~\ref{corollary:star_e_pos} that $$(-1)^{n-1} \eta(M_{St_n}(\xx)) = A_{n-1}(t),$$  which  appeared in \cite{GonzalezDLeon2019}.
\end{remark}

\subsection{Topological consequences}
\label{section:topological_consequences}

Although not needed in the rest of the paper, we discuss some topological consequences of the EL-labeling of Theorem \ref{theorem:ELbond}. First we recall the well-known Philip Hall's theorem stating that given any open interval $(x,y)$ of a pure bounded finite poset $P$, the value of the
M\"obius function $\mu_P(x,y)$  is equal to the reduced Euler characteristic of the order complex $\Delta(x,y)$ of the interval $(x,y)$, where the order complex of a poset is the simplicial complex whose faces are the chains of the poset.
An EL-labeling has the following further topological and homological implications for the order complex $\Delta(x,y)$.

\begin{theorem}[Bj\"orner \cite{Bjorner1980}, Bj\"orner and Wachs {\cite[Theorem 5.9]{BjornerWachs1996}}, cf. Wachs {\cite[Theorem 3.2.4] {Wachs2007}}]\label{theorem:BjornerWachs}
Let $\lambda$ be an EL-labeling of a pure bounded finite poset $P$. Then for each  open interval $(x,y)$ of $P$ we have that, 
\begin{enumerate}
    \item the order complex $\Delta(x,y)$ has the homotopy type of a wedge of 
    $$|\{\cc\mid \cc \text{ is an   ascent-free maximal chain of }[x,y] \}|$$
    $\ell(x,y)$-spheres, and hence the poset $P$ is \emph{Cohen-Macaulay}, i.e., $\widetilde H^d(\Delta(x,y))=0$ for every  $d\neq \ell(x,y)$;
    \item the set 
    $$\{\bar \cc\mid \cc \text{ is an   ascent-free maximal chain of }[x,y] \}$$ forms a basis for the reduced cohomology $\widetilde H^{\ell(x,y)}(\Delta(x,y))$. 
   
\end{enumerate}

\end{theorem}

The following are therefore topological consequences of Theorems \ref{theorem:ELbond} and \ref{theorem:ascent_free_G}. These consequences were proved in \cite{GonzalezDLeon2016} ($r=2$ case in \cite{GonzalezDLeonWachs2016}) for $G=K_n$.

\begin{theorem}
Let $G$ be a connected perfectly labeled chordal graph on $[n]$ and let $r \in \ZZ_{>0} \cup \{\infty\}$. Then $\widehat{\Pi^r_G}$ is Cohen-Macaulay. 
\end{theorem}

\begin{theorem} Let $G$ be a connected perfectly labeled chordal graph on $[n]$ and let $\nu \in \wcomp_{n-1,r}$, where $r \in \ZZ_{>0} \cup \{\infty\}$. Then the set $$\{\bar c(\bT,\rho_{T})\mid \bT \in \Lyn_{G,\nu}\},$$ where $\bar c(\bT,\rho_{T}):=c(\bT,\rho_{T})\setminus \{\hat{0},[n]^{\nu}\}$ and  $\rho_T$ is any linear extension of $T$, is a basis for  top cohomology $\widetilde{H}^{n-3}(\Delta(\hat{0},[n]^\nu))$ of the order complex $\Delta(\hat{0},[n]^\nu)$ of the open interval $(\hat 0, [n]^\nu)$ of $\Pi^r_G$.
\end{theorem}

\begin{proof} It follows from   Theorems~\ref{theorem:BjornerWachs} and~\ref{theorem:ascent_free_G} that $$\{\bar c(\bT,\tau_{T})\mid \bT \in \Lyn_{G,\nu}\},$$ where $\tau_T$ is the decreasing-minimum linear extension, is a basis for  top cohomology of the open interval $(\hat 0, [n]^\nu)$. We claim  that if $\tau$ and $\tau'$ are any two linear extensions of the internal nodes of   $\bT\in \Lyn_{G,\nu}$, then in the top cohomology  group  of the order complex of the interval $(\hat{0},[n]^\nu)$ of $\Pi_G^r$ we have that 
$$\bar c(\bT,\tau)=\pm \bar c(\bT,\tau').$$  Indeed,  this was proved in \cite{GonzalezDLeon2016} for $G=K_n$ and the proof goes through for any connected perfectly labeled chordal graph $G$.
It follows that the decreasing-minimum linear extension $\tau_T$ can be replaced by any linear extension $\rho_T$.
\end{proof}

\section{A symmetric function analog of the chromatic polynomial}\label{section:new_chromatic_symmetric_function}

We consider a symmetric function analog of the characteristic polynomial of a bond lattice. Thanks to Whitney's formula in Equation \eqref{equation:chromatic}, this is a symmetric function analog of the chromatic polynomial of a graph, which is different from Stanley's well-known chromatic symmetric function  \cite{Stanley1995}.

\subsection{Preliminaries}
Given a pure (or ranked) poset $P$, the {\em characteristic polynomial} of a pure poset $P$ with a minimum element $\hat 0$ is defined as
 $${\rm ch}_P(t) := \sum_{x \in P} \mu_P(\hat 0,x) t^{\rho(x)},$$ where $\rho(x)$ is the rank of $x$, i.e., the length of the interval $[\hat 0,x]$.

For any graph $G$, let
\begin{align}\label{equation:new_chromatic_symmetric_definition}
    \Psi_G(\xx) &= \sum_{\bpi \in \Pi_G^\infty} \mu_{\Pi_G^\infty} (\hat 0, \bpi)\xx^{w(\bpi)}
    \end{align}
Equivalently,
\begin{align*}
 \Psi_G(\xx) &= \sum_{\pi \in \Pi_G} M_{G|_\pi}(\xx)
 \end{align*}
where $G|_\pi$ is the subgraph of $G$ whose connected components are of the form $G|_B$ where $B
 \in \pi$.  For example,
 \begin{align}
     \Psi_{P_3}(\xx) &= 1-2e_1(\xx)+e_{1,1}(\xx)+e_{2}(\xx)\label{equation:examples_P_3}
 \end{align}

 The specialization
 $\Psi_G(\stackrel {r}{\overbrace{t,t,\dots,t}},0,0,\dots)$ is equal to the characteristic polynomial  of the $r$-weighted bond poset $\Pi^r_G$.
  By (\ref{equation:chromatic}), the chromatic polynomial $\chi_G(t) $ is obtained from $\Psi_G(t,0,0,\dots) $  by reversing its coefficients.  That is 
 $$\Psi_G(t,0,0,\dots)=t^n\chi_G(t^{-1}) ,$$  where $n$ is the number of vertices of $G$.  Thus we may think of $\Psi_G(\xx)$ as a symmetric function analog of the chromatic polynomial.

There is a well-known classical formula for the characteristic polynomial of $\Pi_{K_n}=\Pi_{n}$ given by,
\begin{equation} \label{equation:char_K_n} \Psi_{K_n}(t,0,0,0,\dots) =  (t-1)(t-2)\cdots(t-n+1).\end{equation} 
 The formula for the characteristic polynomial of $\Pi^2_{K_n} = \mathcal{W}\Pi_{K_n}$ was given in \cite[Theorem 2.8]{GonzalezDLeonWachs2016}, 
\begin{equation} \label{equation:char_K_n_2} \Psi_{K_n}(t,t,0,0,\dots) = (t-n)^{n-1}.\end{equation} 

\begin{question}
    Are there nice formulas for $\Psi_{K_n}(\stackrel {r}{\overbrace{t,t,\dots,t}},0,0,\dots)$ when $r>2$?
\end{question}

\subsection{Alternating $e$-positivity and Schur-log-concavity}
The next theorem  generalizes, in the case of chordal graphs, the fact that the characteristic polynomials of bond lattices  alternate in sign.  Given a symmetric function $f(\xx)$, let  $f(\xx)|_{\deg d}$ denote the degree $d$  homogeneous component of $f(\xx)$.  We say that $f(\xx)$ is {\em alternating $e$-positive} if  $(-1)^d f(\xx)|_{\deg d}$ is $e$-positive for all $d \ge 0$.

\begin{theorem} \label{theorem:chromatic_symmetric} 
 If $G$ is a chordal graph then $\Psi_G(\xx)$ is alternating $e$-positive.
\end{theorem}
\begin{proof}  Let $n$ be the number of vertices of $G$.
 Note  that  $G|_\pi$ has $n$ vertices and $|\pi|$ connected components for all   $\pi \in \Pi_G$.
 Since $M_{G|_\pi}(\xx)$ is homogeneous of degree $n-|\pi|$,  for all $d$ we have $$\Psi_G(\xx)|_{\deg d} = \sum_{\substack{\pi \in \Pi_G \\ |\pi| = n-d}} M_{G|_\pi}(\xx).$$
 
Since $G$ is chordal, each $G|_\pi$  is chordal.  We can therefore apply Corollary~\ref{corollary:e_positive} to each 
 $G|_\pi$.  Thus, if $|\pi| =n-d$ then $(-1)^{d}  M_{G|_\pi}(\xx)$ is $e$-positive.  It follows that each $(-1)^d \Psi_G(\xx)|_{\deg d}$ is $e$-positive giving the desired result.
\end{proof}

A similar argument shows that Conjecture~\ref{conjecture:e_positive} (1) is equivalent to  the conjecture that Theorem~\ref{theorem:chromatic_symmetric} is true for any graph.

A polynomial $\sum_{j=0}^d a_j t^j \in \RR[t]$ is said to be \emph{log-concave} if for all $i\in [d-1]$, 
\begin{equation}\label{equation:log_concave}
a_i^2\ge a_{i-1}a_{i+1}.  
\end{equation}
 For sequences of positive real numbers, log-concavity implies unimodality (see \cite[Lemma 1.1]{Branden2015}).  

A longstanding conjecture of  Hoggar \cite{Hoggar1974}, proved in 2012 by Huh  \cite{Huh2012}, asserts
that for any graph $G$, the chromatic polynomial $\chi_G(t)$ is  log-concave. Since the coefficients of the chromatic polynomial  alternate in sign,  Huh's result also settles a conjecture of Read \cite{Read1968} stating that $(-1)^n \chi_G(-t) $ is unimodal. 

Note that log-concavity  of $\chi_G(t)$ is equivalent to  log-concavity  of $\Psi_G(t,0,0,\dots)$ since the sequences of coefficients are the reverse of each other.  This leads to the following  conjectural generalization.
We will say that a symmetric function $f(\xx)$ is {\em Schur-log-concave} if for all $d \ge 1$,
$$(f(\xx)|_{\deg d})^2 - f(\xx)|_{\deg d-1}f(\xx)|_{\deg d+1} $$ is Schur-positive (i.e., when expanded in the Schur basis 
$\{s_\lambda(\xx) : \lambda \vdash 2d \}$, the coefficients are nonnegative). The following conjecture has been shown computationally for all connected graphs $G$ with at most $7$ vertices.

\begin{conjecture} \label{conjecture:Schur_log} For any graph $G$, the symmetric function
$\Psi_G(\xx)$ is Schur-log-concave.
\end{conjecture}

\begin{example} Checking the Schur-log-concavity condition at $d=1$ for $\Psi_{P_3}(\xx)$  
 from Equation \eqref{equation:examples_P_3} yields
\begin{align*}
    (-2e_{1}(\xx) )^2- (1) (e_{1,1}(\xx) +e_{2}(\xx)) &= 3 s_{1,1}(\xx) +2 s_2(\xx),
\end{align*}
 which is Schur-positive.  
\end{example}

Note that Schur-log-concavity of $\Psi_G(\xx)$ implies log-concavity of the characteristic polynomial $\Psi_G(t,\dotsm,t,0,0,\dots)$ of $\Pi_G^r$ for all $r$.   Huh's result (i.e. the r=1 case) is therefore a specialization of Conjecture~\ref{conjecture:Schur_log} for $\Psi_G(\xx)$.  
For $r=2$ and $G=K_n$, log-concavity of $\Psi_{K_n}(t,t, 0,0,\dots)$ follows from~\eqref{equation:char_K_n_2} and the fact that  real-rooted polynomials with nonnegative coefficients are log-concave (see \cite[Lemma 1.1]{Branden2015}).

In a follow-up paper, we discuss Conjecture~\ref{conjecture:Schur_log} further, as well as a related symmetric function generalization of the rank polynomial.

\section*{Acknowledgements}
R. S. Gonz\'alez D'Le\'on would like to thank Universidad Sergio Arboleda and the program of postdoctoral fellowships of Minciencias (Colombian Ministry of Science)  since part of this project was completed with their support. He is also thankful for the partial support for this research provided by an AMS-Simons Research Enhancement Grant for Primarily Undergraduate Institution Faculty at Loyola University Chicago.

\bibliographystyle{plain}
\bibliography{weighted_bond_posets}

@article{Stanley1997,
	author = "Stanley, Richard P.",
	fjournal = "Electronic Journal of Combinatorics",
	issn = "1077-8926",
	journal = "Electron. J. Combin.",
	mrclass = "05A17 (05E05)",
	mrnumber = "1444167",
	note = "The Wilf Festschrift (Philadelphia, PA, 1996)",
	number = "2",
	pages = "Research Paper 20, approx. 14",
	title = "{Parking functions and noncrossing partitions}",
	url = "http://www.combinatorics.org/Volume_4/Abstracts/v4i2r20.html",
	volume = "4",
	year = "1997"
}

@article{Stanley1974,
author = {Richard P. Stanley},
title = {Finite lattices and {J}ordan-{H}{\"o}lder sets},
journal = {Algebra Universalis},
volume = {4},
number = {1},
pages = {361--371},
year = {1974},
doi = {10.1007/BF02485748}
}

@article{BjornerWachs1996,
	author = "Bj{\"o}rner, Anders and Wachs, Michelle L.",
	coden = "TAMTAM",
	doi = "10.1090/S0002-9947-96-01534-6",
	fjournal = "Transactions of the American Mathematical Society",
	issn = "0002-9947",
	journal = "Trans. Amer. Math. Soc.",
	mrclass = "06A08 (05E99 52B99)",
	mrnumber = "1333388",
	mrreviewer = "T. S. Blyth",
	number = "4",
	pages = "1299--1327",
	title = "{Shellable nonpure complexes and posets. {I}}",
	url = "http://dx.doi.org/10.1090/S0002-9947-96-01534-6",
	volume = "348",
	year = "1996"
}

@article{BjornerWachs1997,
	author = "Bj{\"o}rner, Anders and Wachs, Michelle L.",
	coden = "TAMTAM",
	doi = "10.1090/S0002-9947-97-01838-2",
	fjournal = "Transactions of the American Mathematical Society",
	issn = "0002-9947",
	journal = "Trans. Amer. Math. Soc.",
	mrclass = "06A08 (05E99)",
	mrnumber = "1401765",
	mrreviewer = "Volkmar Welker",
	number = "10",
	pages = "3945--3975",
	title = "{Shellable nonpure complexes and posets. {II}}",
	url = "http://dx.doi.org/10.1090/S0002-9947-97-01838-2",
	volume = "349",
	year = "1997"
}

@article{Bjorner1980,
	author = "Bj{\"o}rner, Anders",
	coden = "TAMTAM",
	fjournal = "Transactions of the American Mathematical Society",
	issn = "0002-9947",
	journal = "Trans. Amer. Math. Soc.",
	mrclass = "06A10 (13H10 52A25)",
	mrnumber = "570784 (81i:06001)",
	mrreviewer = "P. McMullen",
	number = "1",
	pages = "159--183",
	title = "{Shellable and {C}ohen-{M}acaulay partially ordered sets}",
	volume = "260",
	year = "1980"
}

@article{BjornerWachs1983,
	author = "Bj{\"o}rner, Anders and Wachs, Michelle",
	coden = "TAMTAM",
	fjournal = "Transactions of the American Mathematical Society",
	issn = "0002-9947",
	journal = "Trans. Amer. Math. Soc.",
	mrclass = "06A10 (05A99 52A25 57Q05)",
	mrnumber = "690055 (84f:06004)",
	mrreviewer = "R. P. Dilworth",
	number = "1",
	pages = "323--341",
	title = "{On lexicographically shellable posets}",
	volume = "277",
	year = "1983"
}

@incollection{Wachs2007,
	address = "Providence, RI",
	author = "Wachs, Michelle L.",
	booktitle = "{Geometric combinatorics}",
	mrclass = "06B30 (05E10 52C35 55R80)",
	mrnumber = "2383132",
	pages = "497--615",
	publisher = "Amer. Math. Soc.",
	series = "{IAS/Park City Math. Ser.}",
	title = "{Poset topology: tools and applications}",
	volume = "13",
	year = "2007"
}

@article{Wachs1998,
	author = "Wachs, Michelle L.",
	coden = "DSMHA4",
	doi = "10.1016/S0012-365X(98)00147-2",
	fjournal = "Discrete Mathematics",
	issn = "0012-365X",
	journal = "Discrete Math.",
	mrclass = "05E25 (17B01 20C30)",
	mrnumber = "1661375 (2000b:05134)",
	mrreviewer = "Viorel Mihai Gontineac",
	note = "Selected papers in honor of Adriano Garsia (Taormina, 1994)",
	number = "1-3",
	pages = "287--319",
	title = "{On the (co)homology of the partition lattice and the free {L}ie algebra}",
	url = "http://dx.doi.org/10.1016/S0012-365X(98)00147-2",
	volume = "193",
	year = "1998"
}

@article{AvilaCarrilloGonzalezDLeon2026A,
  author       = "Avila, N. and Carrillo, S. A. and Gonz{\'a}lez D'Le{\'o}n, R.~S.",
  title        = {A M{\"o}bius function computing $h$-polynomials of nestohedra},
  journal      = {arXiv preprint arXiv:2609.15784},
  year         = {2026},
  eprint       = {2609.15784},
  archivePrefix= {arXiv},
  primaryClass = {math.CO}
}

@article{AvilaCarrilloGonzalezDleon2026B,
	author = "Avila, N. and Carrillo, S. A. and Gonz{\'a}lez D'Le{\'o}n, R.~S.",
	journal = "In preparation",
	title = "The $\mu$-polynomials of nestohedra",
	year = "2026"
}

@article {Read1968,
    AUTHOR = {Read, Ronald C.},
     TITLE = {An introduction to chromatic polynomials},
   JOURNAL = {J. Combinatorial Theory},
  FJOURNAL = {Journal of Combinatorial Theory},
    VOLUME = {4},
      YEAR = {1968},
     PAGES = {52--71},
      ISSN = {0021-9800},
   MRCLASS = {05.55},
  MRNUMBER = {224505},
MRREVIEWER = {K. Wagner},
}

@article {Hoggar1974,
    AUTHOR = {Hoggar, S. G.},
     TITLE = {Chromatic polynomials and logarithmic concavity},
   JOURNAL = {J. Combinatorial Theory Ser. B},
  FJOURNAL = {Journal of Combinatorial Theory. Series B},
    VOLUME = {16},
      YEAR = {1974},
     PAGES = {248--254},
      ISSN = {0095-8956},
   MRCLASS = {05C15},
  MRNUMBER = {342424},
MRREVIEWER = {Ruth Bari},
       DOI = {10.1016/0095-8956(74)90071-9},
       URL = {https://doi.org/10.1016/0095-8956(74)90071-9},
}

@article {Whitney1932,
    AUTHOR = {Whitney, Hassler},
     TITLE = {A logical expansion in mathematics},
   JOURNAL = {Bull. Amer. Math. Soc.},
  FJOURNAL = {Bulletin of the American Mathematical Society},
    VOLUME = {38},
      YEAR = {1932},
    NUMBER = {8},
     PAGES = {572--579},
      ISSN = {0002-9904},
   MRCLASS = {DML},
  MRNUMBER = {1562461},
       DOI = {10.1090/S0002-9904-1932-05460-X},
       URL = {https://doi.org/10.1090/S0002-9904-1932-05460-X},
}

@article {GonzalezDLeon2016,
    AUTHOR = {Gonz\'{a}lez D'Le\'{o}n, Rafael S.},
     TITLE = {On the free {L}ie algebra with multiple brackets},
   JOURNAL = {Adv. in Appl. Math.},
  FJOURNAL = {Advances in Applied Mathematics},
    VOLUME = {79},
      YEAR = {2016},
     PAGES = {37--97},
      ISSN = {0196-8858},
   MRCLASS = {05E45 (05A18 05E18 17B01 20C30)},
  MRNUMBER = {3505221},
       DOI = {10.1016/j.aam.2016.02.008},
       URL = {https://doi.org/10.1016/j.aam.2016.02.008},
}

@article {GonzalezDLeonWachs2016,
    AUTHOR = {Gonz\'{a}lez D'Le\'{o}n, Rafael S. and Wachs, Michelle L.},
     TITLE = {On the (co)homology of the poset of weighted partitions},
   JOURNAL = {Trans. Amer. Math. Soc.},
  FJOURNAL = {Transactions of the American Mathematical Society},
    VOLUME = {368},
      YEAR = {2016},
    NUMBER = {10},
     PAGES = {6779--6818},
      ISSN = {0002-9947},
   MRCLASS = {05E45 (05E15 05E18 06A11 17B01)},
  MRNUMBER = {3471077},
MRREVIEWER = {Steven Klee},
       DOI = {10.1090/tran/6483},
       URL = {https://doi.org/10.1090/tran/6483},
}

@article{DotsenkoKhoroshkin2007,
	author = "Dotsenko, V. V. and Khoroshkin, A. S.",
	doi = "10.1007/s10688-007-0001-3",
	fjournal = "Rossi\u\i skaya Akademiya Nauk. Funktsional\cprime ny\u\i\ Analiz i ego Prilozheniya",
	issn = "0374-1990",
	journal = "Funktsional. Anal. i Prilozhen.",
	mrclass = "18D50 (17B63 55P48)",
	mrnumber = "2333979 (2008d:18006)",
	mrreviewer = "Eugen Paal",
	number = "1",
	pages = "1--22, 96",
	title = "{Character formulas for the operad of a pair of compatible brackets and for the bi-{H}amiltonian operad}",
	url = "http://dx.doi.org/10.1007/s10688-007-0001-3",
	volume = "41",
	year = "2007"
}

@article {CarrDevadoss2006,
    AUTHOR = {Carr, Michael P. and Devadoss, Satyan L.},
     TITLE = {Coxeter complexes and graph-associahedra},
   JOURNAL = {Topology Appl.},
  FJOURNAL = {Topology and its Applications},
    VOLUME = {153},
      YEAR = {2006},
    NUMBER = {12},
     PAGES = {2155--2168},
      ISSN = {0166-8641},
   MRCLASS = {52B11 (05B45 14H10 14P25)},
  MRNUMBER = {2239078},
MRREVIEWER = {Seth Sullivant},
       DOI = {10.1016/j.topol.2005.08.010},
       URL = {https://doi.org/10.1016/j.topol.2005.08.010},
}

@article{GonzalezDLeon2016-2,
	author = "{Gonz{\'a}lez D'Le{\'o}n}, Rafael S.",
	journal = "The Electronic Journal of Combinatorics",
	number = "1",
	pages = "P1--20",
	title = "{A Note on the $\gamma$-Coefficients of the Tree Eulerian Polynomial}",
	volume = "23",
	year = "2016"
}

@article{StanleyGessel1978,
	author = "Gessel, Ira and Stanley, Richard P.",
	fjournal = "Journal of Combinatorial Theory. Series A",
	journal = "J. Combinatorial Theory Ser. A",
	mrclass = "05A15",
	mrnumber = "0462961 (57 \#2926)",
	mrreviewer = "Stephen Tanny",
	number = "1",
	pages = "24--33",
	title = "{Stirling polynomials}",
	volume = "24",
	year = "1978"
}

@article{PostnikovReinerWilliams2008,
  title={Faces of generalized permutohedra},
  author={Postnikov, Alex and Reiner, Victor and Williams, Lauren},
  journal={Doc. Math},
  volume={13},
  number={207-273},
  pages={51},
  year={2008}
}

@article{Haiman1994,
	author = "Haiman, Mark D.",
	coden = "JAOME7",
	doi = "10.1023/A:1022450120589",
	fjournal = "Journal of Algebraic Combinatorics. An International Journal",
	issn = "0925-9899",
	journal = "J. Algebraic Combin.",
	mrclass = "20C30 (05E05)",
	mrnumber = "1256101 (95a:20014)",
	mrreviewer = "A. O. Morris",
	number = "1",
	pages = "17--76",
	title = "{Conjectures on the quotient ring by diagonal invariants}",
	url = "http://dx.doi.org/10.1023/A:1022450120589",
	volume = "3",
	year = "1994"
}

@article{GonzalezDLeon2019,
  title={A family of symmetric functions associated with Stirling permutations},
  author={Gonz{\'a}lez D’Le{\'o}n, Rafael S},
  journal={Journal of Combinatorics},
  volume={10},
  number={4},
  pages={675--709},
  year={2019},
  publisher={International Press of Boston}
}

@article{GonzalezDLeon2018,
  title={The colored symmetric and exterior algebras},
  author={Gonz{\'a}lez D’Le{\'o}n, Rafael S},
  journal={Journal of Algebra},
  volume={496},
  pages={187--215},
  year={2018},
  publisher={Elsevier}
}

@misc{Ellzey2014,
  author = "Ellzey, Brittney",
  date = "2014",
  howpublished = "personal communication"
}

@article{DeConciniProcesi1995,
  title={Wonderful models of subspace arrangements},
  author={De Concini, Corrado and Procesi, Claudio},
  journal={Selecta Mathematica},
  volume={1},
  number={3},
  pages={459--494},
  year={1995},
  publisher={Birkh{\"a}user-Verlag}
}

@article {AguiarArdila2023,
    AUTHOR = {Aguiar, Marcelo and Ardila, Federico},
     TITLE = {Hopf monoids and generalized permutahedra},
   JOURNAL = {Mem. Amer. Math. Soc.},
  FJOURNAL = {Memoirs of the American Mathematical Society},
    VOLUME = {289},
      YEAR = {2023},
    NUMBER = {1437},
     PAGES = {vi+119},
      ISSN = {0065-9266,1947-6221},
      ISBN = {978-1-4704-6708-1; 978-1-4704-7592-5},
   MRCLASS = {52B05 (05B35 05C65 16T30 52B40)},
  MRNUMBER = {4651496},
MRREVIEWER = {Joseph\ Kung},
       DOI = {10.1090/memo/1437},
       URL = {https://doi.org/10.1090/memo/1437},
}

@article{Postnikov2009,
  title={Permutohedra, associahedra, and beyond},
  author={Postnikov, Alexander},
  journal={International Mathematics Research Notices},
  volume={2009},
  number={6},
  pages={1026--1106},
  year={2009},
  publisher={OUP}
}

@article{Gal2005,
  title={Real root conjecture fails for five-and higher-dimensional spheres},
  author={Gal, Swiatoslaw R},
  journal={Discrete \& Computational Geometry},
  volume={34},
  number={2},
  pages={269--284},
  year={2005},
  publisher={Springer}
}

@article {Stanley1980,
    AUTHOR = {Stanley, Richard P.},
     TITLE = {The number of faces of a simplicial convex polytope},
   JOURNAL = {Adv. in Math.},
  FJOURNAL = {Advances in Mathematics},
    VOLUME = {35},
      YEAR = {1980},
    NUMBER = {3},
     PAGES = {236--238},
      ISSN = {0001-8708},
   MRCLASS = {52A25 (05B99 14M99)},
  MRNUMBER = {563925},
MRREVIEWER = {I. Dolgachev},
       DOI = {10.1016/0001-8708(80)90050-X},
       URL = {https://doi.org/10.1016/0001-8708(80)90050-X},
}

@article{Stanley1995,
  title={A symmetric function generalization of the chromatic polynomial of a graph},
  author={Stanley, Richard P},
  journal={Advances in Mathematics},
  volume={111},
  number={1},
  pages={166--194},
  year={1995},
  publisher={Elsevier}
  }

@book{Stanley1999,
	address = "Cambridge",
	author = "Stanley, Richard P.",
	doi = "10.1017/CBO9780511609589",
	isbn = "0-521-56069-1; 0-521-78987-7",
	mrclass = "05A15 (05-02 05E05 05E10 68R05)",
	mrnumber = "1676282 (2000k:05026)",
	mrreviewer = "Ira Gessel",
	note = "With a foreword by Gian-Carlo Rota and appendix 1 by Sergey Fomin",
	pages = "xii+581",
	publisher = "Cambridge University Press",
	series = "{Cambridge Studies in Advanced Mathematics}",
	title = "{Enumerative combinatorics. {V}ol. 2}",
	url = "http://dx.doi.org/10.1017/CBO9780511609589",
	volume = "62",
	year = "1999"
}

@book{Stanley2012,
	address = "Cambridge",
	author = "Stanley, Richard P.",
	edition = "Second",
	isbn = "978-1-107-60262-5",
	mrclass = "05-02 (05A15 06-02)",
	mrnumber = "2868112",
	pages = "xiv+626",
	publisher = "Cambridge University Press",
	series = "{Cambridge Studies in Advanced Mathematics}",
	title = "{Enumerative combinatorics. {V}olume 1}",
	volume = "49",
	year = "2012"
}

@article{BjornerWelker2005,
	author = "Bj{\"o}rner, Anders and Welker, Volkmar",
	coden = "JPAAA2",
	doi = "10.1016/j.jpaa.2004.11.013",
	fjournal = "Journal of Pure and Applied Algebra",
	issn = "0022-4049",
	journal = "J. Pure Appl. Algebra",
	mrclass = "06A11 (05E25 16S37)",
	mrnumber = "2132872 (2005m:06006)",
	mrreviewer = "Joseph Neggers",
	number = "1-3",
	pages = "43--55",
	title = "{Segre and {R}ees products of posets, with ring-theoretic applications}",
	url = "http://dx.doi.org/10.1016/j.jpaa.2004.11.013",
	volume = "198",
	year = "2005"
}

@article{ShareshianWachs2009,
	author = "Shareshian, John and Wachs, Michelle L.",
	fjournal = "Electronic Journal of Combinatorics",
	issn = "1077-8926",
	journal = "Electron. J. Combin.",
	mrclass = "06A07 (05A30 05E10 05E15)",
	mrnumber = "2576383 (2011e:06003)",
	number = "2, Special volume in honor of Anders Bjorner",
	pages = "Research Paper 20, 29",
	title = "{Poset homology of {R}ees products, and {$q$}-{E}ulerian polynomials}",
	url = "http://www.combinatorics.org/Volume_16/Abstracts/v16i2r20.html",
	volume = "16",
	year = "2009"
}

@article {ShareshianWachs2020,
    AUTHOR = {Shareshian, John and Wachs, Michelle L.},
     TITLE = {Gamma-positivity of variations of {E}ulerian polynomials},
   JOURNAL = {J. Comb.},
  FJOURNAL = {Journal of Combinatorics},
    VOLUME = {11},
      YEAR = {2020},
    NUMBER = {1},
     PAGES = {1--33},
      ISSN = {2156-3527},
 PUBLISHER = {International Press of Boston},
   MRCLASS = {05E05 (11B65 11B68 52B20)},
  MRNUMBER = {4015851},
MRREVIEWER = {Ivica Martinjak},
       DOI = {10.4310/JOC.2020.v11.n1.a1},
       URL = {https://doi.org/10.4310/JOC.2020.v11.n1.a1},
}

@article {HaglundZhang2019,
    AUTHOR = {Haglund, James and Zhang, Philip B.},
     TITLE = {Real-rootedness of variations of {E}ulerian polynomials},
   JOURNAL = {Adv. in Appl. Math.},
  FJOURNAL = {Advances in Applied Mathematics},
    VOLUME = {109},
      YEAR = {2019},
     PAGES = {38--54},
      ISSN = {0196-8858},
   MRCLASS = {05A15 (05A05 26C10 52B05)},
  MRNUMBER = {3954084},
MRREVIEWER = {Mihai Cipu},
       DOI = {10.1016/j.aam.2019.05.001},
       URL = {https://doi.org/10.1016/j.aam.2019.05.001},
}

@book{Drake2008,
	author = "Drake, Brian",
	isbn = "978-0549-69926-2",
	mrclass = "Thesis",
	mrnumber = "2712031",
	note = "Thesis (Ph.D.)--Brandeis University",
	pages = "114",
	publisher = "ProQuest LLC, Ann Arbor, MI",
	title = "{An inversion theorem for labeled trees and some limits of areas under lattice paths}",
	url = "http://gateway.proquest.com/openurl?url_ver=Z39.88-2004&rft_val_fmt=info:ofi/fmt:kev:mtx:dissertation&res_dat=xri:pqdiss&rft_dat=xri:pqdiss:3316489",
	year = "2008"
}

@article {LinussonShareshianWachs2012,
    AUTHOR = {Linusson, Svante and Shareshian, John and Wachs, Michelle L.},
     TITLE = {Rees products and lexicographic shellability},
   JOURNAL = {J. Comb.},
  FJOURNAL = {Journal of Combinatorics},
    VOLUME = {3},
      YEAR = {2012},
    NUMBER = {3},
     PAGES = {243--276},
      ISSN = {2156-3527},
   MRCLASS = {06A11 (05A05 05A30 05E05 05E45 06A07)},
  MRNUMBER = {3029437},
MRREVIEWER = {Peter R. W. McNamara},
       DOI = {10.4310/JOC.2012.v3.n3.a1},
       URL = {https://doi.org/10.4310/JOC.2012.v3.n3.a1},
}

@book {Petersen2015,
    AUTHOR = {Petersen, T. Kyle},
     TITLE = {Eulerian numbers},
    SERIES = {Birkh\"{a}user Advanced Texts: Basler Lehrb\"{u}cher. [Birkh\"{a}user
              Advanced Texts: Basel Textbooks]},
      NOTE = {With a foreword by Richard Stanley},
 PUBLISHER = {Birkh\"{a}user/Springer, New York},
      YEAR = {2015},
     PAGES = {xviii+456},
      ISBN = {978-1-4939-3090-6; 978-1-4939-3091-3},
   MRCLASS = {05-02 (05A15 05Exx 06A07 11B65 11B75 20F55)},
  MRNUMBER = {3408615},
MRREVIEWER = {Damir Yeliussizov},
       DOI = {10.1007/978-1-4939-3091-3},
       URL = {https://doi.org/10.1007/978-1-4939-3091-3},
}

@book {FoataSchutzenberger1970,
    AUTHOR = {Foata, Dominique and Sch\"{u}tzenberger, Marcel-P.},
     TITLE = {Th\'{e}orie g\'{e}om\'{e}trique des polyn\^{o}mes eul\'{e}riens},
    SERIES = {Lecture Notes in Mathematics, Vol. 138},
 PUBLISHER = {Springer-Verlag, Berlin-New York},
      YEAR = {1970},
     PAGES = {v+94},
   MRCLASS = {05.10},
  MRNUMBER = {0272642},
MRREVIEWER = {David P. Roselle},
}

@incollection {Branden2015,
    AUTHOR = {Br\"{a}nd\'{e}n, Petter},
     TITLE = {Unimodality, log-concavity, real-rootedness and beyond},
 BOOKTITLE = {Handbook of enumerative combinatorics},
    SERIES = {Discrete Math. Appl. (Boca Raton)},
     PAGES = {437--483},
 PUBLISHER = {CRC Press, Boca Raton, FL},
      YEAR = {2015},
   MRCLASS = {05Axx (05D40 05E10)},
  MRNUMBER = {3409348},
}

@article {BjornerWachs1982,
    AUTHOR = {Bj\"{o}rner, Anders and Wachs, Michelle},
     TITLE = {Bruhat order of {C}oxeter groups and shellability},
   JOURNAL = {Adv. in Math.},
  FJOURNAL = {Advances in Mathematics},
    VOLUME = {43},
      YEAR = {1982},
    NUMBER = {1},
     PAGES = {87--100},
      ISSN = {0001-8708},
   MRCLASS = {20H15 (06A10 13F20 14L30 52A43)},
  MRNUMBER = {644668},
MRREVIEWER = {S. Milne},
       DOI = {10.1016/0001-8708(82)90029-9},
       URL = {https://doi.org/10.1016/0001-8708(82)90029-9},
}

@inproceedings {Huh2018,
    AUTHOR = {Huh, June},
     TITLE = {Combinatorial applications of the {H}odge-{R}iemann relations},
 BOOKTITLE = {Proceedings of the {I}nternational {C}ongress of
              {M}athematicians---{R}io de {J}aneiro 2018. {V}ol. {IV}.
              {I}nvited lectures},
     PAGES = {3093--3111},
 PUBLISHER = {World Sci. Publ., Hackensack, NJ},
      YEAR = {2018},
   MRCLASS = {05A20},
  MRNUMBER = {3966524},
}

@book {GrahamKnuthPatashnik1989,
    AUTHOR = {Graham, Ronald L. and Knuth, Donald E. and Patashnik, Oren},
     TITLE = {Concrete mathematics},
      NOTE = {A foundation for computer science},
 PUBLISHER = {Addison-Wesley Publishing Company, Advanced Book Program,
              Reading, MA},
      YEAR = {1989},
     PAGES = {xiv+625},
      ISBN = {0-201-14236-8},
   MRCLASS = {00A05 (00-01 05-01 68-01 68Rxx)},
  MRNUMBER = {1001562},
MRREVIEWER = {Volker Strehl},
}

@article {Huh2012,
    AUTHOR = {Huh, June},
     TITLE = {Milnor numbers of projective hypersurfaces and the chromatic
              polynomial of graphs},
   JOURNAL = {J. Amer. Math. Soc.},
  FJOURNAL = {Journal of the American Mathematical Society},
    VOLUME = {25},
      YEAR = {2012},
    NUMBER = {3},
     PAGES = {907--927},
      ISSN = {0894-0347},
   MRCLASS = {14B05 (05B35 14C17)},
  MRNUMBER = {2904577},
MRREVIEWER = {Paolo Aluffi},
       DOI = {10.1090/S0894-0347-2012-00731-0},
       URL = {https://doi.org/10.1090/S0894-0347-2012-00731-0},
}

@article{GesselSeo2004,
	author = "Gessel, Ira M. and Seo, Seunghyun",
	fjournal = "Electronic Journal of Combinatorics",
	issn = "1077-8926",
	journal = "Electron. J. Combin.",
	mrclass = "05A15 (05C05)",
	mrnumber = "2224940 (2006m:05010)",
	mrreviewer = "Pavlo Pylyavskyy",
	number = "2",
	pages = "Research Paper 27, 23 pp. (electronic)",
	title = "{A refinement of {C}ayley's formula for trees}",
	url = "http://www.combinatorics.org/Volume_11/Abstracts/v11i2r27.html",
	volume = "11",
	year = "2004/06"
}

@article {GradyPoznanovic2020,
    AUTHOR = {Grady, Amy and Poznanovi\'c, Svetlana},
     TITLE = {Tree descent polynomials: unimodality and central limit
              theorem},
   JOURNAL = {Ann. Comb.},
  FJOURNAL = {Annals of Combinatorics},
    VOLUME = {24},
      YEAR = {2020},
    NUMBER = {1},
     PAGES = {109--117},
      ISSN = {0218-0006,0219-3094},
   MRCLASS = {05E05 (05A05)},
  MRNUMBER = {4078142},
MRREVIEWER = {Jia\ Huang},
       DOI = {10.1007/s00026-019-00484-1},
       URL = {https://doi.org/10.1007/s00026-019-00484-1},
}

\end{document}